\documentclass[final]{amsart}

\usepackage{amssymb,amsmath,color}
\usepackage{graphicx}
\usepackage{ulem}

\newtheorem{theorem}{Theorem}[section]
\newtheorem{lemma}[theorem]{Lemma}
\newtheorem{corollary}[theorem]{Corollary}
\newtheorem{proposition}[theorem]{Proposition}
\theoremstyle{definition}
\newtheorem{remark}[theorem]{Remark}
\newtheorem{definition}[theorem]{Definition}
\newtheorem{illustration}[theorem]{Illustration}

\newcommand{\R}{{\mathbb R}}
\newcommand{\N}{{\mathbb N}}
\newcommand{\argmin}{{\rm argmin}\kern 0.12em}

\newcommand{\Hb}{\mathcal H}

\newcommand{\inttz}{\int_{t_0}^t}
\newcommand{\Elambda}{\mathcal E_\lambda}

\newcommand{\tcr}{\textcolor{red}}

\begin{document}

\title{Fast convex optimization via inertial dynamics with Hessian driven damping}

\author{Hedy Attouch}
\address{Institut Montpelli\'erain Alexander Grothendieck, UMR 5149 CNRS, Universit\'e Montpellier 2, place Eug\`ene Bataillon,
34095 Montpellier cedex 5, France}
\email{hedy.attouch@univ-montp2.fr}

\author{Juan Peypouquet}
\address{Univesidad T\'ecnica Federico Santa Maria, Av Espana 1680, Valparaiso, Chile}
\email{juan.peypouquet@usm.cl}

\author{Patrick Redont}
\address{Institut Montpelli\'erain Alexander Grothendieck, UMR 5149 CNRS, Universit\'e Montpellier 2, place Eug\`ene Bataillon,
34095 Montpellier cedex 5, France}
\email{patrick.redont@univ-montp2.fr}

\date{23 oct. 2015}



\thanks{Effort sponsored by the Air Force Office of Scientific Research, Air Force Material Command, USAF, under grant number FA9550-14-1-0056. Also supported by Fondecyt Grant 1140829, Conicyt Anillo ACT-1106, ECOS-Conicyt Project C13E03, Millenium Nucleus ICM/FIC RC130003, Conicyt Project MATHAMSUD 15MATH-02, Conicyt Redes 140183, and Basal Project CMM Universidad de Chile.}

\begin{abstract}
We first study the fast minimization properties of the trajectories of  the second-order evolution equation
\begin{equation*}
 \ddot{x}(t) + \frac{\alpha}{t} \dot{x}(t) + \beta \nabla^2 \Phi (x(t))\dot{x} (t) + \nabla \Phi (x(t)) = 0,
\end{equation*}
where $\Phi : \mathcal H \to \mathbb R$ is a smooth convex function acting on a real Hilbert space $\mathcal H$, and $\alpha$, $\beta$ are positive parameters. This inertial system combines  an isotropic viscous  damping  which vanishes asymptotically, and
a geometrical Hessian driven damping, which makes it naturally related to Newton's and Levenberg-Marquardt methods.
For $\alpha\geq 3$, and $\beta >0$, along any trajectory, fast convergence of the values
\begin{align*}
  \Phi(x(t))-  \min_{\mathcal H}\Phi =\mathcal O\left(t^{-2}\right)
\end{align*}
is obtained, together with rapid convergence of the gradients $\nabla \Phi (x(t))$ to zero.
For $\alpha > 3$, just assuming that $\argmin \Phi \neq \emptyset,$ we show that any trajectory converges weakly to a minimizer of $\Phi$, and that $ \Phi(x(t))-  \min_{\mathcal H}\Phi = o(t^{-2})$. Strong convergence is established in various  practical situations. In particular, for the strongly convex case, we obtain an even faster speed of convergence which can be arbitrarily fast depending on the choice of $\alpha$. More precisely, we have $\Phi(x(t))-  \min_{\mathcal H}\Phi = \mathcal O(t^{-\frac{2}{3}\alpha})$.
Then, we extend the results to the case of a general proper lower-semicontinuous convex function $\Phi : \mathcal H \rightarrow \mathbb R \cup \lbrace + \infty \rbrace$. This is based on the crucial property that the inertial dynamic with Hessian driven damping can be equivalently  written as a first-order system in time and space,  allowing to extend it by simply replacing the gradient with the subdifferential. By explicit-implicit time discretization, this opens a gate to new $-$ possibly more rapid $-$ inertial algorithms, expanding the field of FISTA methods for convex structured optimization problems.
\end{abstract}
       
\maketitle

\section*{\large{\bf Introduction}}
Throughout the paper, $\mathcal H$ is a real Hilbert space endowed with scalar product $\langle\cdot,\cdot\rangle$ and norm $\|x\|=\sqrt{\langle x,x\rangle}$ for $x\in \mathcal H$.
Let $\Phi:\mathcal H\rightarrow\mathbb R$ be a twice continuously 
differentiable convex function (the case of a  nonsmooth function will be 
considered later on). 
In view of the minimization of $\Phi$, we study the asymptotic behaviour (as $t \to + \infty$) of the trajectories of the second-order differential equation 
\begin{equation} \label{edo01} 
\boxed{
 \mbox{(DIN-AVD)} \quad \quad \ddot{x}(t) + \frac{\alpha}{t} \dot{x}(t) + \beta \nabla^2 \Phi (x(t))\dot{x} (t) + \nabla \Phi (x(t)) = 0
}
\end{equation}
where $\alpha$ and $\beta$  are positive parameters. \\	

This inertial system combines two types of damping:

In the first place, the term $\frac{\alpha}{t} \dot{x}(t)$ furnishes an isotropic linear damping with a viscous parameter $\frac{\alpha}{t}$ which vanishes asymptotically, but not too slowly. The asymptotic behavior of the inertial gradient-like system 
\begin{equation}\label{edo2}
\mbox{(AVD)}\quad\quad\ddot x(t)+a(t)\dot x(t)+\nabla\Phi(x(t))=0,
\end{equation}
with Asymptotic Vanishing Damping ((AVD) for short), has been studied by Cabot, Engler  and Gaddat in \cite{CEG1}-\cite{CEG2}. They proved that, under moderate decrease of $a$ to zero, namely, that $\lim_{t\to+\infty}a(t)=0$ and $\int_0^{\infty}a(t)dt=+\infty$, every solution $x$ of \eqref{edo2} satisfies $\lim_{t\to+\infty}\Phi(x(t))\to\min_{\mathcal H}\Phi$.

Interestingly, with the specific choice $a(t)=\frac{\alpha}{t}$:
\begin{equation}\label{edo4}
\ddot{x}(t) + \frac{\alpha}{t} \dot{x}(t)  + \nabla \Phi (x(t)) = 0,
\end{equation}
Su, Boyd and Cand\`es in \cite{SBC} proved the fast convergence property 
\begin{align}\label{basic-fast}
  \Phi(x(t))-  \min_{\mathcal H}\Phi=\mathcal{O}\left(t^{-2}\right),
\end{align}
provided $\alpha\geq3$. In the same article, the authors show that, for $\alpha =3$, \eqref{edo4} can be seen as a 
continuous-time version of the fast convergent method of Nesterov \cite{Nest1}-\cite{Nest2}-\cite{Nest3}-\cite{Nest4}.
In \cite{APR1}, Attouch, Peypouquet and Redont showed that, for 
$\alpha >3$, each trajectory of (\ref{edo4}) converges weakly to an element of 
$\argmin \Phi$. This result is a continuous-time counterpart to the 
Chambolle-Dossal algorithm \cite{CD}, which is a modified Nesterov algorithm 
specially designed to obtain the convergence of the iterates.

\medskip 
In the second place, a geometrical damping, attached to the term $\beta\nabla^2 \Phi (x(t))\dot{x} (t)$, has a natural link with Newton's method. It gives rise to the so-called Dynamical Inertial Newton system
 ((DIN) for short)
 \begin{equation}\label{edo5}
 \mbox{(DIN)} \quad \quad \ddot{x}(t)+ \gamma \dot{x} (t)  + \beta \nabla^2 \Phi (x(t))\dot{x} (t) + \nabla \Phi (x(t)) = 0,
\end{equation}
which has been introduced by Alvarez, Attouch, Bolte and Redont in \cite{aabr} 
($\gamma$ is a fixed positive parameter). Interestingly, (\ref{edo5}) can be equivalently written as a first-order system involving only the gradient of $\Phi$, which allows its extension to the case of a proper lower-semicontinuous convex function $\Phi$. This led to applications ranging from
optimization algorithms \cite{APR} to unilateral mechanics and partial differential equations \cite{AMR}.

\medskip
As we shall see, (DIN-AVD) inherits the convergence properties of both (AVD) and (DIN), but exhibits other important features, namely (see Theorems \ref{T:Phi_t-2}, \ref{T:weak_convergence}, \ref{T:superconv}, \ref{Thm-sc-basic}, \ref{Thm-g-weak-conv2}, \ref{T:weak_convergence-nonsmooth}, \ref{T:superconv-nonsmooth}):
\begin{itemize}
	\item Assuming $\alpha\geq3$, $\beta>0$ and $\argmin\Phi\neq\emptyset$, we show the fast convergence 
	property of the values (\ref{basic-fast}), together with the fast convergence 
	to zero of the gradients
	 \begin{equation}\label{edo6}
	\int_0^{\infty} t^2 \|\nabla \Phi (x(t)) \|^2  dt   < + \infty.
	\end{equation}
	\item For $\alpha>3$, we complete these results by showing that every trajectory 
	converges weakly, with its limit belonging to $\argmin\Phi$. Moreover, we obtain a faster order of convergence $\Phi(x(t))-  \min_{\mathcal H}\Phi = o(t^{-2})$. 
	\item Also for $\alpha>3$, strong convergence is established in various  practical situations. In particular, for the strongly convex case, we obtain an even faster speed of convergence which can be arbitrarily fast according to the choice of $\alpha$. More precisely, we have $\Phi(x(t))-  \min_{\mathcal H}\Phi =\mathcal O(t^{-\frac{2}{3}\alpha})$.
	\item A remarkable property of the system (DIN-AVD) is that these results can be naturally generalized to the non-smooth convex case. The key argument is that it can be reformulated as a first-order system (both in time and space) involving only the gradient and not the Hessian!
\end{itemize}

Time discretization of (DIN-AVD) provides new ideas for the design of innovative fast converging algorithms, expanding the field of rapid methods for structured convex minimization of Nesterov \cite{Nest1,Nest2,Nest3,Nest4}, Beck-Teboulle \cite{BT}, and Chambolle-Dossal \cite{CD}. This study, however, goes beyond the scope of this paper, and will be carried out in a future research. As briefly evoked  above, the continuous (DIN-AVD) system is also linked to the modeling of non-elastic shocks in unilateral mechanics, and the geometric damping of nonlinear oscillators. These are important areas for applications, which are not considered in this paper.

\section{Smooth potential} \label{S:smooth}

The following minimal hypotheses are in force in this section, and are always 
tacitly assumed:
\begin{itemize}
	\item $\alpha>0$, $\beta>0$;
	\item $\Phi:\mathcal H\rightarrow\mathbb R$ is a twice continuously 
	differentiable convex function; and
	\item $t_0>0$\footnote{Taking $t_0>0$ comes from the singularity of the 
	damping coefficient $a(t)=\frac{\alpha}{t}$ at zero. Since we are only 
	concerned about the asymptotic behaviour of the trajectories, the time origin 
	is unimportant. If one insists in starting from $t_0=0$, then all the results 
	remain valid with $a(t)=\frac{\alpha}{t+1}$.}, $x_0\in\Hb$, $\dot x_0\in\Hb$.
\end{itemize}

In view of minimizing $\Phi$, we study the asymptotic behaviour, as $t\to+\infty$, of a solution 
$x$ to (DIN-AVD) second-order evolution equation \eqref{edo01}. We will successively examine the following points: 
\begin{itemize}
	\item existence and uniqueness of a solution $x$ to (DIN-AVD) with Cauchy data 
	  $x(t_0)=x_0$ and $\dot x(t_0)=\dot x_0$;
	\item minimizing properties of $x$ and convergence of $\Phi(x(t))$ towards 
	  $\inf\Phi$ whenever $\alpha>0$;
	\item fast convergence of $\Phi(x(t))$ towards $\min\Phi$, when the latter is attained and $\alpha\geq3$; 
	\item weak convergence of $x$ towards a minimum of $\Phi$ and 
	   faster convergence of $\Phi(x(t))$, when $\alpha>3$;
	\item some cases of strong convergence of $x$, and faster convergence of $\Phi(x(t))$.
\end{itemize}

\subsection{Existence and uniqueness of solution}

The following result will be derived in Section \ref{section-existence} from a more general result concerning a convex 
lower semicontinuous function $\Phi:\Hb\rightarrow\R\cup\lbrace+\infty\rbrace$ (see Corollary 
\ref{Cor-existence} below):

\begin{theorem}
For any Cauchy data $(x_0,\dot x_0)\in\Hb\times\Hb$, (DIN-AVD) admits a unique twice continuously differentiable global solution 
$x:[t_0,+\infty[\rightarrow\Hb$ verifying $(x(t_0),\dot x(t_0))=(x_0,\dot x_0)$. 
\end{theorem}

\subsection{Lyapunov analysis and minimizing properties of the solutions for $\alpha>0$} \label{SS:Lyapunov}

In this section, we present a family of Lyapunov functions for (DIN-AVD), and use them to derive the main properties of the solutions to this system. As we shall see, the fact that we have more than one (essentially different) of these functions will play a crucial role in establishing that the gradient vanishes as $t\to+\infty$.\\

Let $x:t\in[t_0,\infty[\to\Hb$ satisfy (DIN-AVD) with Cauchy data $x(t_0)=x_0$ and $\dot x(t_0)=\dot x_0$, and let $\theta\in[0,\beta]$. Define $W_\theta:[t_0,+\infty[\to\R$ by
\begin{equation} \label{E:W_theta}
W_\theta(t)=\Phi(x(t))+\frac{1}{2}\|\dot x(t)+\theta\nabla\Phi(x(t))\|^2+\frac{\theta(\beta-\theta)}{2}\|\nabla \Phi(x(t))\|^2.
\end{equation} 
Observe that, for $\theta=0$, we obtain
$$W_0(t)=\Phi(x(t))+\frac{1}{2}\|\dot x(t)\|^2,$$
which is the usual global mechanical energy of the system. We shall see that, for each $\theta\in[0,\beta]$, $W_\theta$ is a strict Lyapunov function for {\rm (DIN-AVD)}.\\

In order to simplify the notation, write 
\begin{equation} \label{u0reg}
u_\theta(t)=x(t)+\theta\int_{t_0}^t\nabla\Phi(x(s))ds,
\end{equation}
so that \ $u_0(t)=x(t),$ and, for each $\theta\in[0,\beta]$, 
\begin{eqnarray} 
\dot u_\theta(t) & = & \dot x(t)+\theta\nabla\Phi(x(t)) \label{uniter} \\
W_\theta(t) & = & \Phi(x(t))+\frac{1}{2}\|\dot u_\theta(t)\|^2+\frac{\theta(\beta-\theta)}{2}\|\nabla \Phi(x(t))\|^2 .\nonumber 
\end{eqnarray}
Using \eqref{uniter} and (DIN-AVD), elementary computations yield 
\begin{equation}\label{ubis}
\ddot u_\theta(t)=\ddot x(t)+\theta\nabla^2\Phi(x(t))\dot x(t)=-\frac{\alpha}{t}\dot x(t)-\nabla \Phi(x(t))+(\theta-\beta)\nabla^2\Phi(x(t))\dot x(t).
\end{equation}

We have the following:

\begin{proposition}\label{P:Lyapunov}
Let $\alpha>0$, and suppose $x:[t_0,+\infty[\rightarrow\Hb$ is a solution of {\rm (DIN-AVD)}. Then, for each $\theta\in[0,\beta]$ and $t\ge\max\{t_0,\frac{\alpha\theta}{2}\}$, we have
$$\dot W_\theta(t)\le-\frac{\alpha}{2t}\|\dot x(t)\|^2 -\frac{\alpha}{2t}\|\dot u_\theta(t)\|^2.$$
\end{proposition}

\begin{proof}
First observe that
\begin{equation} \label{E:Wdot}
\dot W_\theta(t)=\langle \nabla\Phi(x(t)),\dot x(t)\rangle+\langle\ddot u_\theta(t),\dot u_\theta(t)\rangle+\theta(\beta-\theta)\langle\nabla^2\Phi(x(t))\dot x(t),\nabla\Phi(x(t))\rangle.
\end{equation} 
Next, use \eqref{uniter} and \eqref{ubis} to obtain
\begin{eqnarray*}
\langle\ddot u_\theta(t),\dot u_\theta(t)\rangle & = & \langle -\frac{\alpha}{t}\dot x(t)-\nabla \Phi(x(t))+(\theta-\beta)\nabla^2\Phi(x(t))\dot x(t), \dot x(t)+\theta\nabla\Phi(x(t))\rangle  \\
& = & -\frac{\alpha}{t}\|\dot x(t)\|^2-(\frac{\alpha\theta}{t}+1) \langle\nabla\Phi(x(t)),\dot x(t)\rangle-\theta\|\nabla\Phi(x(t))\|^2+(\theta-\beta)\langle\nabla^2\Phi(x(t))\dot x(t),\dot x(t)\rangle\\
&& \quad +\theta(\theta-\beta)\langle\nabla^2\Phi(x(t))\dot x(t),\nabla\Phi(x(t))\rangle\\
& \le & -\frac{\alpha}{t}\|\dot x(t)\|^2-(\frac{\alpha\theta}{t}+1) \langle\nabla\Phi(x(t)),\dot x(t)\rangle-\theta\|\nabla\Phi(x(t))\|^2+\theta(\theta-\beta)\langle\nabla^2\Phi(x(t))\dot x(t),\nabla\Phi(x(t))\rangle,
\end{eqnarray*}
since $\langle\nabla^2\Phi(x(t))\dot x(t),\dot x(t)\rangle\ge 0$ by convexity. From \eqref{E:Wdot}, we obtain
\begin{equation} \label{E:Wdot_bis}
\dot W_\theta(t)\le -\frac{\alpha}{t}\|\dot x(t)\|^2-\frac{\alpha\theta}{t}\langle\nabla\Phi(x(t)),\dot x(t)\rangle-\theta\|\nabla\Phi(x(t))\|^2.
\end{equation} 
On the one hand, when $\theta=0$, it immediately follows that
$$\dot W_0(t)\le-\frac{\alpha}{t}\|\dot x(t)\|^2.$$
On the other hand, if $\theta\in]0,\beta]$, we use \eqref{uniter} in \eqref{E:Wdot_bis} to deduce that
$$\dot W_\theta(t)\le \left(\frac{2}{\theta}-\frac{\alpha}{t}\right)\langle \dot u_\theta(t),\dot x(t)\rangle-\frac{1}{\theta}\|\dot x(t)\|^2-\frac{1}{\theta}\|\dot u_\theta(t)\|^2\le -\frac{\alpha}{2t}\|\dot x(t)\|^2 -\frac{\alpha}{2t}\|\dot u_\theta(t)\|^2,$$
since $\frac{2}{\theta}-\frac{\alpha}{t}\ge 0$ by hypothesis, and the fact that $|\langle\zeta,\xi\rangle|\le\frac{1}{2}\|\zeta\|^2+\frac{1}{2}\|\xi\|^2$ for all $\zeta,\xi\in \Hb$.
\end{proof}

\begin{theorem} \label{T:limit_W_Phi}
Let $\alpha>0$, and suppose $x:[t_0,+\infty[\rightarrow\Hb$ is a solution of {\rm (DIN-AVD)}. Then
$$\lim_{t\to+\infty}W_0(t)=\lim_{t\to+\infty}W_\beta(t)=\lim_{t\to+\infty}\Phi(x(t))=\inf\Phi\in\R\cup\{-\infty\}.$$
\end{theorem}

\begin{proof}
Since we are interested in asymptotic properties of $x$, we can assume 
$t\geq t_1=\max\{t_0,\alpha\beta\}$ throughout the proof. Take $\theta\in\{0,\beta\}$, so that the last term in the definition \eqref{E:W_theta} of $W_\theta$ vanishes. Given $z\in\mathcal H$, we define $h:[t_1,+\infty [ \to\R$ by
$$
h(t)=\frac{1}{2}\|u_\theta(t)-z\|^2.
$$
By the Chain Rule, we have
$$\dot h(t) = \langle u_\theta(t) - z , \dot{u}_\theta(t)  \rangle\qquad\hbox{and}\qquad 
\ddot h(t) = \langle u_\theta(t) - z , \ddot{u}_\theta(t)  \rangle + \| \dot{u}_\theta(t) \|^2.$$

On the other hand, from \eqref{uniter} and \eqref{ubis}, we obtain
\begin{equation} \label{uter}
\ddot u_\theta(t)+\frac{\alpha}{t}\dot u_\theta(t) = \left(\frac{\alpha\theta}{t}-1\right)\nabla\Phi(x(t))+(\theta-\beta)\nabla^2\Phi(x(t))\dot x(t).
\end{equation} 
Set $$I(t):=\frac{1}{2}\left\|\int_{t_0}^t\nabla\Phi(x(s))ds\right\|^2\qquad\hbox{and}\qquad J(t):=\langle x(t)-z,\nabla\Phi(x(t))\rangle-\Phi(x(t)),$$
and observe that  
$$\dot I(t)=\left\langle\int_{t_0}^t\nabla\Phi(x(s))ds,\nabla\Phi(x(t))\right\rangle \qquad\hbox{and}\qquad \dot J(t)=\langle x(t)-z,\nabla^2\Phi(x(t))\dot x(t)\rangle.$$
Next, since $\theta\in\{0,\beta\}$, we can write
\begin{eqnarray*} 
\ddot h(t) + \frac{\alpha}{t} \dot h(t)  
& = & 
  \|\dot u_\theta(t)\|^2 +\left(\frac{\alpha\theta}{t}-1\right)\langle u_\theta(t)-z,\nabla\Phi(x(t))\rangle+(\theta-\beta)\langle u_\theta(t)-z, \nabla^2\Phi(x(t))\dot x(t)\rangle\\
& = & 
  \|\dot u_\theta(t)\|^2 +\left(\frac{\alpha\theta}{t}-1\right)\langle x(t)-z,\nabla\Phi(x(t))\rangle
  +\theta\left(\frac{\alpha\theta}{t}-1\right)
    \dot I(t)+(\theta-\beta)\dot J(t)\\
& \le &   \|\dot u_\theta(t)\|^2 +\left(\frac{\alpha\theta}{t}-1\right)\big(\Phi(x(t))-\Phi(z)\big)
  +\theta\left(\frac{\alpha\theta}{t}-1\right)
    \dot I(t)+(\theta-\beta)\dot J(t),    
\end{eqnarray*}
where the last inequality follows from the convexity of $\Phi$ and the fact that $t\ge\alpha\beta\ge\alpha\theta$. Using the definition \eqref{E:W_theta} of $W_\theta$, and Proposition \ref{P:Lyapunov}, we get
\begin{eqnarray*}
\ddot h(t) + \frac{\alpha}{t} \dot h(t) & = & \left(\frac{3}{2}-\frac{\alpha\theta}{2t}\right)\|\dot u_\theta(t)\|^2 +\left(\frac{\alpha\theta}{t}-1\right)\big(W_\theta(t)-\Phi(z)\big)  +\theta\left(\frac{\alpha\theta}{t}-1\right)  \dot I(t)+(\theta-\beta)\dot J(t)\\
& \le & -\left(\frac{3t}{\alpha}-\theta\right)\dot W_\theta(t)+\left(\frac{\alpha\theta}{t}-1\right)\big(W_\theta(t)-\Phi(z)\big)  +\theta\left(\frac{\alpha\theta}{t}-1\right)  \dot I(t)+(\theta-\beta)\dot J(t).    
\end{eqnarray*}

Dividing by $t$ and rearranging the terms, we have
$$\frac{1}{t}\ddot h(t)+\left(\frac{1}{t}-\frac{\alpha\theta}{t^2}\right)\big(W_\theta(t)-\Phi(z)\big)\le -\left(\frac{3}{\alpha}-\frac{\theta}{t}\right)\dot W_\theta(t)-\left[\frac{\alpha}{t^2}\dot h(t)+\theta\left(\frac{1}{t}-\frac{\alpha\theta}{t^2}\right)  \dot I(t)+\frac{(\beta-\theta)}{t}\dot J(t)\right].$$
Since $h$, $I$ and $J$ are bounded from below, we can integrate this inequality from $t_1$ to $t$, and use Lemma \ref{L:int_bounded} to obtain $C\in\R$ such that
\begin{equation} \label{E:h_dot}
\frac{1}{t}\dot h(t)+\int_{t_1}^t\left(\frac{1}{s}-\frac{\alpha\theta}{s^2}\right)\big(W_\theta(s)-\Phi(z)\big)\,ds \le -\int_{t_1}^t\left(\frac{3}{\alpha}-\frac{\theta}{s}\right)\dot W_\theta(s)\,ds+C.
\end{equation} 

Since $W_\theta$ is nonincreasing, we have
\begin{eqnarray} \label{E:int_W}
\int_{t_1}^t\left(\frac{1}{s}-\frac{\alpha\theta}{s^2}\right)\big(W_\theta(s)-\Phi(z)\big)\,ds 
& \ge &  \big(W_\theta(t)-\Phi(z)\big)\int_{t_1}^t\left(\frac{1}{s}-\frac{\alpha\theta}{s^2}\right)\,ds \nonumber \\
& = & \big(W_\theta(t)-\Phi(z)\big)\left(\ln t-\ln t_1+\frac{\alpha\theta}{t}-\frac{\alpha\theta}{t_1}\right).
\end{eqnarray}

In turn,
\begin{eqnarray} \label{E:int_dot_W}
-\int_{t_1}^t\left(\frac{3}{\alpha}-\frac{\theta}{s}\right)\dot W_\theta(s)\,ds 
& = & \left(\frac{3}{\alpha}-\frac{\theta}{t_1}\right)\big(W_\theta(t_1)-\Phi(z)\big)
-\left(\frac{3}{\alpha}-\frac{\theta}{t}\right)\big(W_\theta(t)-\Phi(z)\big)
+\theta\int_{t_1}^t\frac{W_\theta(s)-\Phi(z)}{s^2}\,ds\nonumber\\
& \le & \left(\frac{3}{\alpha}-\frac{\theta}{t_1}\right)\big(W_\theta(t_1)-\Phi(z)\big)
-\left(\frac{3}{\alpha}-\frac{\theta}{t}\right)\big(W_\theta(t)-\Phi(z)\big)
+\theta\big(W_\theta(t_1)-\Phi(z)\big)\left(\frac{1}{t_1}-\frac{1}{t}\right),\nonumber\\
& \le & \frac{3}{\alpha}\big|W_\theta(t_1)-\Phi(z)\big|-\left(\frac{3}{\alpha}-\frac{\theta}{t}\right)\big(W_\theta(t)-\Phi(z)\big),
\end{eqnarray} 
since $t\mapsto W_\theta(t)-\Phi(z)$ is nonincreasing and $t\ge t_1\ge\alpha\beta\ge\alpha\theta$.

Combining \eqref{E:h_dot}, \eqref{E:int_W} and \eqref{E:int_dot_W}, we deduce that

$$\frac{1}{t}\dot h(t)
+\big(W_\theta(t)-\Phi(z)\big)\left(\ln t+D+\frac{E}{t}\right) 
\le C'$$
for appropriate constants $C',D,E\in\R$.

Now, take $t_2\ge t_1$ such that $\ln s+D+\frac{E}{s}\ge 0$ for all $s\ge t_2$, and integrate from $t_2$ to $t$ to obtain
$$\frac{h(t)}{t}-\frac{h(t_2)}{t_2}+\int_{t_2}^t\frac{h(s)}{s^2}ds
  +\big(W_\theta(t)-\Phi(z)\big)\int_{t_2}^t\left(\log s+D+\frac{E}{s}\right)ds
\leq C'(t-t_2).$$
Since $h$ is nonnegative, this implies
$$-\frac{h(t_2)}{t_2}+\big(W_\theta(t)-\Phi(z)\big)\left(t\ln t-t_2\ln t_2+(D-1)(t-t_2)+E(\ln t-\ln t_2)\right)ds
\leq C'(t-t_2),$$
and so,
\begin{equation}\label{int2}
\big(W_\theta(t)-\Phi(z)\big)\left(t\ln t+(D-1)t+E\ln t+F\right)ds
\leq C't+G,
\end{equation}
for some other constants $F,G\in\R$. As $t\to+\infty$, we obtain $\lim_{t\to+\infty}W_\theta(t)\le\Phi(z)$ (the limit is in $\R\cup\{-\infty\}$). Since $z$ is arbitrary, and $\inf\Phi\leq\Phi(x(t))\leq W_\theta(t)$ for all $t$, the result follows.
\end{proof}

By the weak lower-semicontinuity of $\Phi$, Theorem \ref{T:limit_W_Phi} immediately yields the following:

\begin{corollary} \label{C:limitpoints}
Let $\alpha>0$, and suppose $x:[t_0,+\infty[\rightarrow\Hb$ is a solution of {\rm (DIN-AVD)}. As $t\to+\infty$, every sequential weak cluster point of $x(t)$ belongs to $\argmin\Phi$. In particular, if $\|x(t)\|$ does not tend to $+\infty$ as $t\to+\infty$, then $\argmin\Phi\neq\emptyset$.
\end{corollary}

If the function $\Phi$ is bounded from below, we have the following stability result:

\begin{proposition} \label{P:speed_to_0}
Let $\alpha>0$, and suppose $x:[t_0,+\infty[\rightarrow\Hb$ is a solution of {\rm (DIN-AVD)}. If $\inf\Phi>-\infty$, then 
$$\lim_{t\to+\infty}\|\dot x(t)\|=\lim_{t\to+\infty}\|\nabla\Phi(x(t))\|=0,\quad \int_{t_0}^\infty\frac{1}{t}\|\dot x(t)\|^2dt<+\infty,\quad\hbox{and}\quad \int_{t_0}^\infty\frac{1}{t}\|\nabla\Phi(x(t))\|^2dt<+\infty.$$
\end{proposition}

\begin{proof}
Theorem \ref{T:limit_W_Phi} establishes that $\lim_{t\rightarrow+\infty}W_0(t)=\lim_{t\rightarrow+\infty}W_\beta(t)=\lim_{t\rightarrow+\infty}\Phi(x(t))=\inf\Phi\in\R\cup\{-\infty\}$. If $\Phi$ is bounded below, the limits belong to $\R$. We deduce that
$$\left\{\begin{array}{rcccl}
\lim\limits_{t\to+\infty}\|\dot x(t)\|^2 & = & 2\lim\limits_{t\to+\infty}\big(W_0(t)-\Phi(x(t))\big) & = & 0\medskip\\
\lim\limits_{t\to+\infty}\|\dot u_\beta(t)\|^2 & = & 2\lim\limits_{t\to+\infty}\big(W_\beta(t)-\Phi(x(t))\big) & = & 0.
\end{array}\right.$$
By definition \eqref{u0reg}, we have $\beta\nabla\Phi(x(t))=\dot u_\beta(t)-\dot x(t)$, and so, $\lim_{t\to+\infty}\|\nabla\Phi(x(t))\|=0$. Finally, Proposition \ref{P:Lyapunov} gives
$$\int_{t_1}^{\infty}\frac{\alpha}{2s}(\|\dot x(s)\|^2+\|\dot u_\beta(s)\|^2)\,ds\leq W_\beta(t_1)-\inf\Phi<+\infty.$$
It suffices to use $\beta\nabla\Phi(x(t))=\dot u_\beta(t)-\dot x(t)$ again to complete the proof. 
\end{proof}

\begin{proposition} \label{P:log_estimates}
Let $\alpha>0$, and suppose $x:[t_0,+\infty[\rightarrow\Hb$ is a solution of 
{\rm (DIN-AVD)}. If $\argmin\Phi\neq\emptyset$, then 
\begin{itemize}
	\item [i)] $\Phi(x(t))-\min\Phi=\mathcal O\left(\frac{1}{\log t}\right)$, $\|\dot x(t)\|=\mathcal O\left(\frac{1}{\sqrt{\log t}}\right)$ and $\|\nabla\Phi(x(t))\|=\mathcal O\left(\frac{1}{\sqrt{\log t}}\right)$;
	\item [ii)] $\int_{t_0}^\infty\frac{1}{t}(\Phi(x(t))-\min\Phi)dt<+\infty$.
\end{itemize}
\end{proposition}

\begin{proof}
Fix $\hat z\in\argmin\Phi$.

For i), observe that $0\le \Phi(x(t))-\min\Phi+\frac{1}{2}\|\dot u_\theta(t)\|^2=W_\theta(t)-\min\Phi$.
Next, use \eqref{int2} with $z=\hat z$ to conclude.

For ii), use $z=\hat z$ in inequalities \eqref{E:h_dot} and \eqref{E:int_dot_W}, and combine them to deduce that
\begin{equation} \label{E:int_W/s}
\frac{1}{2t}\frac{d}{dt}\|x(t)-\hat z\|^2+\left(1-\frac{\alpha\theta}{t_1}\right)\int_{t_1}^t\frac{W_0(s)-\Phi(z)}{s}\,ds \le C''
\end{equation} 
for $t\ge t_1$ and some other constant $C''$. On the other hand, since $\lim_{t\to+\infty}\|\dot x(t)\|=0$ by Proposition \ref{P:speed_to_0}, we have $\|\dot x\|_\infty:=\sup_{t\ge t_0}\|\dot x(t)\|<+\infty$. It follows that
$$\frac{1}{t}\frac{d}{dt}\|x(t)-\hat z\|^2\le \frac{1}{t}\|\dot x\|_\infty\,\|x(t)-\hat z\|\le \|\dot x\|_\infty\,\left(\frac{\|x(t_1)-\hat z\|}{t_1}+\|\dot x\|_\infty\right)$$
by the Mean Value Theorem. From \eqref{E:int_W/s}, we deduce that
$$\int_{t_1}^t\frac{W_0(s)-\Phi(z)}{s}\,ds<+\infty$$
which yields the result.
\end{proof}

\begin{remark}
Most of the results in this section can be established without using the differentiability of $\dot x$ and $\nabla\Phi(x)$ independently, but only that of 
$\dot u_\theta=\dot x+\theta\nabla\Phi(x)$, along with relations \eqref{uniter} and \eqref{ubis}, and the chain rule $\frac{d}{dt}\Phi(x(t))=\langle\nabla\Phi(x(t)),\dot x(t)\rangle$. We shall develop these arguments in Section \ref{section-existence}, when we deal with a nonsmooth potential.
\end{remark}

\subsection{Fast convergence of the values for $\alpha\geq3$} \label{SS:age3_smooth}
In this part we mainly analyze the fast convergence of the values of $\Phi$ 
along a trajectory of (DIN-AVD). The value $\alpha=3$ plays a special role: to our 
knowledge, it is the smallest for which fast convergence results are proved to hold.  

Suppose $\alpha\geq3$ and $x^\ast\in\argmin\Phi$. Let $x$ be a solution of 
(DIN-AVD) with Cauchy data $(x(t_0),\dot x(t_0))=(x_0,\dot x_0)\in\Hb\times\Hb$. For 
$\lambda\in[2,\alpha-1]$ we define the function $\Elambda:[t_0,+\infty[\rightarrow\R$ by  
\begin{equation}\label{rfast2}
\Elambda(t)=
  t(t-\beta(\lambda+2-\alpha))(\Phi(x(t))-\min\Phi)+
  \frac{1}{2}\|\lambda(x(t)-x^\ast)+t\dot u_\beta(t)\|^2+
  \lambda(\alpha-\lambda-1)\frac{1}{2}\|x(t)-x^\ast\|^2,
\end{equation}
where $u_\beta$ is given by \eqref{u0reg}, with $\theta=\beta$.
To compute $\frac{d}{dt}\Elambda(t)$ we first differentiate each 
term of $\Elambda$ in turn (we use \eqref{ubis} in the second derivative).
$$
\frac{d}{dt}[t(t-\beta(\lambda+2-\alpha))(\Phi(x(t))-\min\Phi)] 
  =(2t-\beta(\lambda+2-\alpha))(\Phi(x(t))-\min\Phi)+ 
  t(t-\beta(\lambda+2-\alpha))\langle\dot x(t),\nabla\Phi(x(t))\rangle; 
$$
\begin{eqnarray*}
\frac{d}{dt}\frac{1}{2}\|\lambda(x(t)-x^\ast)+t\dot u_\beta(t)\|^2 
& = & 
  \langle\lambda(x(t)-x^\ast)+t\dot u_\beta(t),
  \lambda\dot x(t)+\dot u_\beta(t)+t\ddot u_\beta(t)\rangle \\
& = & 
  \langle\lambda(x(t)-x^\ast)+t\dot u_\beta(t),
  (\lambda+1-\alpha)\dot x(t)-(t-\beta)\nabla\Phi(x(t))\rangle \\
& = & 
  \lambda(\lambda+1-\alpha)\langle x(t)-x^\ast,\dot x(t)\rangle 
  -t(\alpha-\lambda-1)\|\dot x(t)\|^2 
  -\beta t(t-\beta)\|\nabla\Phi(x(t))\|^2 \\
& & 
  -\lambda(t-\beta)\langle x(t)-x^\ast,\nabla\Phi(x(t))\rangle
  -t(t-\beta(\lambda+2-\alpha))\langle\dot x(t),\nabla\Phi(x(t))\rangle; 
\end{eqnarray*}
\begin{eqnarray*}
\frac{d}{dt}\lambda(\alpha-\lambda-1)\frac{1}{2}\|x(t)-x^\ast\|^2
& = &
\lambda(\alpha-\lambda-1)\langle x(t)-x^\ast,\dot x(t)\rangle.
\end{eqnarray*}
Whence
\begin{eqnarray}
\frac{d}{dt}\Elambda(t)
& = & 
  (2t-\beta(\lambda+2-\alpha))(\Phi(x(t))-\min\Phi)
  -\lambda(t-\beta)\langle x(t)-x^\ast,\nabla\Phi(x(t))\rangle 
  \label{drfast2} \\
& &
  -t(\alpha-\lambda-1)\|\dot x(t)\|^2 
  -\beta t(t-\beta)\|\nabla\Phi(x(t))\|^2. \nonumber  
\end{eqnarray}
Now, $\langle x(t)-x^\ast,\nabla\Phi(x(t))\rangle\geq\Phi(x(t)-\Phi(x^\ast)$. If $t\geq \max\{t_0,\beta\}$, we deduce, from \eqref{drfast2}, that
\begin{equation} \label{rfastoche}
\frac{d}{dt}\Elambda(t)\leq
  -((\lambda-2)t-\beta(\alpha-2))(\Phi(x(t))-\min\Phi)
  -t(\alpha-\lambda-1)\|\dot x(t)\|^2-\beta t(t-\beta)\|\nabla\Phi(x(t))\|^2. 
\end{equation}

\begin{remark} \label{R:Elambda_dec}
Recall that $\Elambda$ is nonnegative. Let us give a closer look at the coefficients on the right-hand side: First, $(\lambda-2)t-\beta(\alpha-2)\ge 0$ for $t\ge t_1=\max\{t_0,\beta,\frac{\beta(\alpha-2)}{\lambda-2}\}$ provided $\lambda>2$. Next, $\alpha-\lambda-1\ge 0$ whenever $\lambda\le \alpha-1$. A compatibility condition for these two relations to hold is that $2<\lambda\le\alpha-1$, thus $\alpha>3$. The limiting case $\lambda=2$ (thus $\alpha=3$) will be included in Lemma \ref{L:Elambda_decrease} below. Finally, $\beta t(t-\beta)\ge 0$ for $t\ge \beta$. Summarizing, if $\lambda\in]2,\alpha-1]$, we immediately deduce that $\Elambda$ is nonincreasing on the interval $[t_1,+\infty[$, and $\lim_{t\rightarrow+\infty}\Elambda(t)$ exists.
\end{remark}

\begin{lemma} \label{L:Elambda_decrease}
Let $\alpha\geq3$ and $\argmin\Phi\neq\emptyset$. Suppose  $x:[t_0,+\infty[\rightarrow\Hb$ is a solution of {\rm (DIN-AVD)}. If 
$\lambda\in[2,\alpha-1]$, then the function $$t\mapsto \left(\frac{t}{t-\beta}\right)^{\alpha-2}\Elambda(t)$$ 
is nonincreasing and $\lim_{t\rightarrow+\infty}\Elambda(t)$ exists. 
\end{lemma}

\begin{proof}
Since we are interested in asymptotic properties of $x$, we can assume $t>\max\{t_0,\beta\}$. From \eqref{rfastoche} we deduce 
$$
\frac{d}{dt}\Elambda(t)\leq\beta(\alpha-2)(\Phi(x(t))-\min\Phi). 
$$
Multiplying by $t(t-\beta)$ and noticing $\lambda+2-\alpha\leq1$ we obtain
\begin{eqnarray*}
t(t-\beta)\frac{d}{dt}\Elambda(t)
& \leq &
\beta(\alpha-2)t(t-\beta)(\Phi(x(t))-\min\Phi) \\
& \leq &
\beta(\alpha-2)t(t-\beta(\lambda+2-\alpha))(\Phi(x(t))-\min\Phi) \\
& \leq &
\beta(\alpha-2)\Elambda(t).
\end{eqnarray*}
Now, multiplying by $t^{\alpha-3}(t-\beta)^{1-\alpha}$ we obtain
$$
\left(\frac{t}{t-\beta}\right)^{\alpha-2}\frac{d}{dt}\Elambda(t) \leq 
\beta(\alpha-2)\frac{t^{\alpha-3}}{(t-\beta)^{\alpha-1}}\Elambda(t),  
$$
whence we deduce 
$$
\frac{d}{dt}\left[\left(\frac{t}{t-\beta}\right)^{\alpha-2}\Elambda(t)\right] 
\leq0.
$$
Therefore, the function $t^{\alpha-2}(t-\beta)^{2-\alpha}\Elambda(t)$ is nonincreasing. Since it is nonnegative, it has a limit as $t\to+\infty$, and, 
clearly, so does $\Elambda$.
\end{proof}

An important consequence is the following:

\begin{theorem} \label{T:Phi_t-2}
Let $\alpha\geq3$ and $\argmin\Phi\neq\emptyset$. Suppose  $x:[t_0,+\infty[\rightarrow\Hb$ is a solution of {\rm (DIN-AVD)}. Then, $x$ is bounded. Moreover, set $\lambda\in[2,\alpha-1]$ and $t_1=\max\{t_0,\beta\}$. For all $t\geq s>t_1$, we have 
$$\Phi(x(t))-\min\Phi\leq   \frac{1}{t^2}\left(\frac{s}{s-\beta}\right)^{\alpha-2}\Elambda(s)= \mathcal O(t^{-2}).$$
\end{theorem}

\begin{proof}
Take $\lambda\in[2,\alpha-1]$. By the definition \eqref{rfast2} of $\Elambda$, we have 
\begin{equation}\label{MM}
\frac{1}{2}\|\lambda(x(t)-x^\ast)+t(\dot x(t)+\beta\nabla\Phi(x(t)))\|^2
\leq  
\Elambda(t) 
\end{equation}
Since $\lim_{t\rightarrow+\infty}\Elambda(t)$ exists by Lemma \ref{L:Elambda_decrease}, we can take an upper bound $M$ for $\Elambda$. Expanding the square we get
$$
\lambda^2\frac{1}{2}\|x(t)-x^*\|^2
  +\lambda t\langle x(t)-x^*,\dot x(t)\rangle
  +\lambda t\langle x(t)-x^*,\beta\nabla\Phi(x(t))\rangle
  +t^2\frac{1}{2}\|\dot x(t)+\beta\nabla\Phi(x(t))\|^2
\leq  
M.
$$
Neglecting the last two terms of the left-hand side, which are nonnegative, we 
deduce that
\begin{equation*}
\lambda\frac{1}{2}\|x(t)-x^\ast\|^2+t\langle x(t)-x^\ast,\dot x(t)\rangle
\leq
\frac{M}{\lambda}.
\end{equation*}
Set $h(t)=\frac{1}{2}\|x(t)-x^\ast\|^2$ and multiply by $t^{\lambda-1}$ to 
obtain
\begin{equation*}
\lambda t^{\lambda-1}h(t)+t^\lambda\dot h(t)
=\frac{d}{dt}\left(t^\lambda h(t)\right)
\leq\frac{M}{\lambda}t^{\lambda-1}.
\end{equation*}
Integrating from $t_1$ to $t>t_1$ we derive
\begin{equation*}
t^\lambda h(t)-t_1^\lambda h(t_1)\leq\frac{M}{\lambda^2}(t^\lambda-t_1^\lambda).
\end{equation*}
Hence 
\begin{equation*}
h(t)\leq h(t_1)+\frac{M}{\lambda^2}.
\end{equation*}
We conclude that $h$ is bounded, and so is $x$.

For the rate of convergence, let us return to the definition of $\Elambda$. We have 
$$
\Phi(x(t))-\min\Phi\leq\frac{\Elambda(t)}{t(t-\beta(\lambda+2-\alpha))}
  \leq\frac{\Elambda(t)}{t(t-\beta)}.
$$
By Lemma \ref{L:Elambda_decrease} again, the function $t^{\alpha-2}(t-\beta)^{2-\alpha}\Elambda(t)$ is 
nonincreasing. Hence, for $t\geq s>t_1$, we have 
\begin{equation*}
\Phi(x(t))-\min\Phi
\leq \frac{1}{t(t-\beta)}\left(\frac{t-\beta}{t}\right)^{\alpha-2}  \left(\frac{s}{s-\beta}\right)^{\alpha-2}\Elambda(s)
\leq \frac{1}{t^2}\left(\frac{t-\beta}{t}\right)^{\alpha-3} \left(\frac{s}{s-\beta}\right)^{\alpha-2}\Elambda(s)
\leq \frac{1}{t^2}\left(\frac{s}{s-\beta}\right)^{\alpha-2}\Elambda(s),
\end{equation*}
as required.
\end{proof}

\begin{proposition} \label{P:tGrad_L2}
Let $\alpha\geq3$ and $\argmin\Phi\neq\emptyset$. Suppose 
$x:[t_0,+\infty[\rightarrow\Hb$ is a solution of (DIN-AVD). Then
$$\int_{t_0}^\infty t^2\|\nabla\Phi(x(t))\|^2dt<+\infty$$
and
$$\|\dot x(t)+\beta\nabla\Phi(x(t))\|=\mathcal O\left(\frac{1}{t}\right).$$
If, moreover, $\nabla\Phi$ is Lipschitz continuous on bounded sets then 
	$$\|\ddot x(t)\|=\mathcal O\left(\frac{1}{\sqrt{\ln t}}\right).$$
\end{proposition}

\begin{proof}
Let $\lambda\in[2,\alpha-1]$ and $t_1=\max\{t_0,\beta\}$. From \eqref{rfastoche} we deduce 
\begin{equation*}
\frac{d}{dt}\Elambda(t)\leq
  \beta(\alpha-2)(\Phi(x(t))-\min\Phi)
  -\beta t(t-\beta)\|\nabla\Phi(x(t))\|^2. 
\end{equation*}
Integrating from $t_1$ to $t>t_1$ we derive
\begin{equation*}
\beta\int_{t_1}^t t(t-\beta)\|\nabla\Phi(x(t))\|^2dt\leq
  \Elambda(t_1)-\Elambda(t)
  +\beta(\alpha-2)\int_{t_1}^t(\Phi(x(t))-\min\Phi)dt.
\end{equation*}
In view of Lemma \ref{L:Elambda_decrease} and Theorem \ref{T:Phi_t-2}, the right-hand side 
has a limit, which settles the first claim. 

For the second one, from \eqref{MM}, we deduce that
\begin{equation*}
t\|\dot x(t)+\beta\nabla\Phi(x(t))\|\leq
  \sqrt{2\Elambda(t)}+\lambda\|x(t)-x^\ast\|.
\end{equation*}
Since $\Elambda$ and $x$ are bounded, we conclude that $\|\dot x(t)+\beta\nabla\Phi(x(t))\|=\mathcal O(t^{-1})$ . 

Assume now that $\nabla\Phi$ is Lipschitz continuous on bounded sets. Since $x$ is bounded, so is 
$\nabla^2\Phi(x)$. By Proposition \ref{P:log_estimates}, we have $\|\frac{\alpha}{t}\dot x(t)\|=\mathcal O(1/t\sqrt{\ln t})$,
$\|\nabla^2\Phi(x(t))\dot x(t)\|=\mathcal O(1/\sqrt{\ln t})$ and
$\|\nabla\Phi(x(t))\|=\mathcal O(1/\sqrt{\ln t})$. Using this information in (DIN-AVD), we obtain 
$\|\ddot x(t)\|=\mathcal O(1/\sqrt{\ln t})$. 
\end{proof}

\if{
\begin{remark}
Except for item \eqref{fast-conv2-6}, the proof of Theorem 
\ref{Thm-fast-conv2} does not resort to the differentiability of $\dot x$ or 
$\nabla\Phi(x)$, but only to the differentiability of 
$\dot u=\dot x+\beta\nabla\Phi(x)$.
\end{remark}
}\fi

\begin{remark} \label{R:initial_condition}
Suppose $t_0>\beta$, so that $t_1=t_0>\beta$ in Theorem \ref{T:Phi_t-2}. Letting $s\downarrow t_0$, the estimation becomes
$$\Phi(x(t))-\min\Phi\leq   \frac{1}{t^2}\left(\frac{t_0}{t_0-\beta}\right)^{\alpha-2}\Elambda(t_0),$$
where
$$\Elambda(t_0)= t_0(t_0-\beta(\lambda+2-\alpha))(\Phi(x_0)-\min\Phi)+
  \frac{1}{2}\|\lambda(x_0-x^\ast)+t_0(\dot x_0+\beta\nabla\Phi(x_0))\|^2+
  \lambda(\alpha-\lambda-1)\frac{1}{2}\|x_0-x^\ast\|^2.$$ 
This value is important numerically. A judicious choice for the Cauchy data would consist in taking $\dot x_0=-\beta\nabla\Phi(x_0)$, $x_0$ as close 
as possible to the optimal set, and $\Phi(x_0)$ as  small as possible. For $\beta=0$, we recover the same constant $C$ as for the (AVD) system. The 
comparison of the value of $C$ for these two systems is an interesting question that requires further study.
\end{remark}

\subsection{Weak convergence of the trajectories and faster convergence of the values for $\alpha>3$}

We are now in a position to prove the weak convergence of the trajectories of (DIN-AVD), which is the main result of this section. In order to analyze the  convergence properties of the trajectories of system {\rm(\ref{edo01})}, we will use Opial's lemma \cite{Op}, that we recall in its continuous form in the Appendix (see also \cite{Bruck}, who initiated the use of this argument to analyze the asymptotic convergence of nonlinear contraction semigroups in Hilbert spaces).

We begin by establishing the following technical result, which is interesting in its own right:

\begin{lemma} \label{superconv}
Let $\alpha>3$ and $x^*\in\argmin\Phi$. Suppose 
$x:[t_0,+\infty[\rightarrow\Hb$ is a solution of {\rm (DIN-AVD)}. Then
\begin{itemize}
\item [(i)] \label{superconv-1} 
  $\displaystyle \int_{t_0}^\infty t(\Phi(x(t))-\min\Phi)dt<\infty$ and 
  $\displaystyle\int_{t_0}^\infty t\|\dot x(t)\|^2dt<\infty$;\medskip 
\item [(ii)] \label{superconv-2} 
  $\displaystyle\int_{t_0}^\infty t\langle x(t)-x^*,\nabla\Phi(x(t))\rangle dt<\infty$; and\medskip 
\item [(iii)] \label{superconv-3} 
  $\lim_{t\rightarrow+\infty}\|x(t)-x^\ast\|$ and 
  $\lim_{t\rightarrow+\infty}
    t\langle x(t)-x^\ast,\dot x(t)+\beta\nabla\Phi(x(t))\rangle$ 
exist. 
\end{itemize}
\end{lemma}
\begin{proof}
For (i), use \eqref{rfastoche} with $\lambda\in]2,\alpha-1[$ and $t\ge t_1=\max\{t_0,\beta,\frac{\beta(\alpha-2)}{\lambda-2}\}$ to deduce that
\begin{equation*} 
((\lambda-2)t-\beta(\alpha-2))(\Phi(x(t))-\min\Phi)
  +t(\alpha-\lambda-1)\|\dot x(t)\|^2
\leq
-\frac{d}{dt}\Elambda(t).
\end{equation*}
Integrate between $t_1$ and $t\geq t_1$ to obtain
\begin{equation*} 
\int_{t_1}^t((\lambda-2)s-\beta(\alpha-2))(\Phi(x(s))-\min\Phi)ds
  +\int_{t_1}^t s(\alpha-\lambda-1)\|\dot x(s)\|^2ds
\leq
\Elambda(t_1)-\Elambda(t).
\end{equation*}
It suffices to observe that the integrands are nonnegative (see Remark \ref{R:Elambda_dec}) and the right-hand side has a limit as $t\to+\infty$ by Lemma \ref{L:Elambda_decrease}.

To prove (ii), observe that, from \eqref{drfast2}, we have 
$$
\frac{d}{dt}\Elambda(t)\leq
  (2t-\beta(\lambda+2-\alpha))(\Phi(x(t))-\min\Phi)
  -\lambda(t-\beta)\langle x(t)-x^\ast,\nabla\Phi(x(t))\rangle
$$
for $t\geq t_1$. Integrating from $t_1$ to $t\geq t_1$, we obtain 
$$
\int_{t_1}^t\lambda(s-\beta)\langle x(s)-x^\ast,\nabla\Phi(x(s))\rangle ds\leq 
  \Elambda(t_1)-\Elambda(t)+
    \int_{t_1}^t(2s-\beta(\lambda+2-\alpha))(\Phi(x(s))-\min\Phi)ds.
$$
The claim follows from part (i) and Lemma \ref{L:Elambda_decrease} since the integrand on the left-hand side is nonnegative. 

Finally, for (iii), take two distinct values $\lambda$ and $\lambda'$ in $[2,\alpha-1]$. We have
\begin{equation}\label{diffElambda}
\mathcal E_{\lambda'}(t)-\Elambda(t)=(\lambda'-\lambda)
  \left(
  -\beta t(\Phi(x(t))-\min\Phi)
  +t\langle x(t)-x^\ast,\dot x(t)+\beta\nabla\Phi(x(t))\rangle
  +\frac{\alpha-1}{2}\|x(t)-x^\ast\|^2
  \right).
\end{equation}
Using part (i) above, along with Lemma \ref{L:Elambda_decrease}, we deduce that the quantity $k(t)$, defined as
$$k(t):=t\langle x(t)-x^\ast,\dot x(t)+\beta\nabla\Phi(x(t))\rangle+\frac{\alpha-1}{2}\|x(t)-x^\ast\|^2,$$
has a limit as $t\to+\infty$. Our goal, then, is to show that each term has a limit. By setting
$$q(t):=\frac{1}{2}\|x(t)-x^\ast\|^2+\inttz\langle x(s)-x^\ast,\beta\nabla\Phi(x(s))\rangle ds,$$
we may write $k(t)$ as
$$k(t)=t\dot q(t)+(\alpha-1)q(t) -\beta(\alpha-1)\inttz\langle x(s)-x^\ast,\nabla\Phi(x(s))\rangle ds.$$
Using (ii) and the fact that the integrand is nonnegative, we deduce that the last term has a limit as $t\to+\infty$. It ensues that $t\dot q(t)+(\alpha-1)q(t)$ has a limit, and, by Lemma \ref{elemutil}, so does $q(t)$. As a consequence,  $\lim_{t\rightarrow+\infty}\|x(t)-x^\ast\|$ exists, and then $\lim_{t\rightarrow+\infty} t\langle x(t)-x^\ast,\dot x(t)+\beta\nabla\Phi(x(t))\rangle$ exists as well.
\end{proof}

\begin{theorem} \label{T:weak_convergence}
Let $\alpha>3$ and $\argmin\Phi\neq\emptyset$. Suppose 
$x:[t_0,+\infty[\rightarrow\Hb$ is a solution of {\rm (DIN-AVD)}. Then $x(t)$ 
converges weakly in $\Hb$, as $t\to+\infty$, to a point in $\argmin\Phi$.
\end{theorem}
\begin{proof}
By part (iii) in Lemma \ref{superconv}, $\lim_{t\rightarrow+\infty}\|x(t)-x^\ast\|$ for every $x^\ast\in\argmin\Phi$. Next, by Theorem \ref{T:Phi_t-2} and the weak lower-semicontinuity of $\Phi$, every sequential weak cluster point of $x(t)$ as $t\to+\infty$, belongs to $\argmin\Phi$. The convergence is thus a consequence of Opial's Lemma.
\end{proof}

We now prove that the convergence of the values is actually faster than the one predicted by Theorem \ref{T:Phi_t-2}:

\begin{theorem} \label{T:superconv}
Let $\alpha>3$ and $\argmin\Phi\neq\emptyset$. Suppose 
$x:[t_0,+\infty[\rightarrow\Hb$ is a solution of {\rm (DIN-AVD)}. Then
\begin{eqnarray*}
\Phi(x(t))-\min\Phi & = & o\left(t^{-2}\right) \\
\|\dot x(t)+\beta\nabla\Phi(x(t))\| & = & o\left(t^{-1}\right). 
\end{eqnarray*}
\end{theorem}
\begin{proof}
Let $\lambda\in[2,\alpha-1]$. Function $\Elambda$ can also be written
\begin{eqnarray*}
\Elambda(t) 
& = &
  t(t-\beta(\lambda+2-\alpha))(\Phi(x(t))-\min\Phi)
  +t^2\frac{1}{2}\|\dot x(t)+\beta\nabla\Phi(x(t))\|^2 \\
& &
  +\lambda t\langle x(t)-x^\ast,\dot x(t)+\beta\nabla\Phi(x(t))\rangle 
  +\lambda(\alpha-1)\frac{1}{2}\|x(t)-x^\ast\|^2.
\end{eqnarray*}
In view of Lemma \ref{L:Elambda_decrease} and part (iii) of Lemma \ref{superconv}, the function 
$$
g(t)=t(t-\beta(\lambda+2-\alpha))(\Phi(x(t))-\min\Phi)
  +t^2\frac{1}{2}\|\dot x(t)+\beta\nabla\Phi(x(t))\|^2
$$
has a limit as $t\to+\infty$. Moreover, for $t\geq\max\{t_0,\beta\}$ we have 
$$
0\leq t^{-1}g(t)\leq
(t-\beta(\lambda+2-\alpha))(\Phi(x(t))-\min\Phi)
  +t\|\dot x(t)\|^2+\beta^2t\|\nabla\Phi(x(t))\|^2, 
$$
where the right-hand side is integrable on $[t_0,+\infty[$ by Proposition \ref{P:tGrad_L2} and part (i) of Lemma \ref{superconv}. Hence 
$\int_{t_0}^\infty t^{-1}g(t)dt<+\infty$. This forces 
$\lim_{t\to+\infty}g(t)=0$ and proves the claim, since $g$ is the sum of two 
nonnegative terms.
\end{proof}

\subsection{Some remarks concerning the Hessian-driven damping term}

\subsubsection{Second-order differentiability of $\Phi$}

In order to simplify the presentation, we have assumed, from the beginning, that $\Phi$ is twice continuously differentiable. However, the Hessian of the function $\Phi$ appears explicitly in just a few of the results that have been established so far:
\begin{itemize}
	\item It is used in Proposition \ref{P:Lyapunov} and Theorem \ref{T:limit_W_Phi}, but only for the parts concerning $W_\theta$ for $\theta\neq\beta$. In particular, it plays a role in the asymptotic properties of $W_0$ but not for those of $W_\beta$.
	\item Next, in Proposition \ref{P:speed_to_0}, we combine the asymptotic properties of $W_0$ and $W_\beta$ in order to ensure that $$\lim_{t\to+\infty}\|\dot x(t)\|=\lim_{t\to+\infty}\|\nabla\Phi(x(t))\|=0.$$
	\item This argument also appears in the proof of Proposition \ref{P:log_estimates}, which is interesting, but not central to this study. In turn, the estimates in Proposition \ref{P:log_estimates} are then used in Proposition \ref{P:tGrad_L2} to prove that the acceleration $\ddot x$ vanishes as $t\to+\infty$ when $\nabla\Phi$ is Lipschitz-continuous.
\end{itemize}
In Section \ref{section-existence}, we analyze the system (DIN-AVD) in the case of a nonsmooth potential. According to the preceding discussion, one may reasonably conjecture that most properties will possibly remain valid in a less regular context, except, perhaps, for those where the Hessian plays an active role.

\subsubsection{The case $\beta =0$}

In the case $\beta =0$, (DIN-AVD) becomes
\begin{equation}\label{edo01-0}
 \mbox{(AVD)} \quad \quad \ddot{x}(t) + \frac{\alpha}{t} \dot{x}(t)  + \nabla \Phi (x(t)) = 0.
\end{equation}
The following facts concerning this system have been established in \cite{SBC}, \cite{APR1} and \cite{May}, and can be recovered as special cases by setting $\beta=0$ in the corresponding results presented here, namely:
\begin{itemize}
	\item If $\alpha>0$ then $\lim_{t\to+\infty}\Phi(x(t))=\inf\Phi$ and every sequential weak cluster point of $x(t)$ as $t\to+\infty$, belongs to $\argmin\Phi$ (Theorem \ref{T:limit_W_Phi} and Corollary \ref{C:limitpoints}). If, moreover, $\inf\Phi>-\infty$, then $\lim_{t\to+\infty}\|\dot x(t)\|=0$ (Proposition \ref{P:speed_to_0}). \medskip
	\item If $\alpha\ge 3$ and $\argmin\Phi\neq\emptyset$, then $x$ is bounded and $\Phi(x(t))-\min\Phi=\mathcal O(t^{-2})$ (Theorem \ref{T:Phi_t-2}). \medskip
	\item If $\alpha>3$ and $\argmin\Phi\neq\emptyset$, then $\Phi(x(t))-\min\Phi = o\left(t^{-2}\right)$, $\|\dot x(t)\|=o(t^{-1})$, and $x(t)$ converges weakly, as $t\to+\infty$, to a point in $\argmin\Phi$ (Theorems \ref{T:weak_convergence} and \ref{T:superconv}). Strong convergence holds if $\Phi$ is even, uniformly convex, boundedly inf-compact, or if $\hbox{int}(\argmin\Phi)\neq\emptyset$ (see Theorems \ref{Thm-conv-forte}, \ref{Thm-strong-int}, \ref{Thm-infc-basic} and \ref{Thm-strong-convex} in Section \ref{S-conv-forte} below).\bigskip 
\end{itemize}

\subsubsection{The transition from $\beta>0$ to $\beta=0$}\label{Hdd-coeff}

Recall that (DIN-AVD) is given by
\begin{equation}\label{beta-variant}
 \ddot{x}(t) + \frac{\alpha}{t} \dot{x}(t) + \beta \nabla^2 \Phi (x(t))\dot{x} (t) + \nabla \Phi (x(t)) = 0.
\end{equation} 
It turns out that the qualitative behavior of this system does not depend on the value of $\beta>0$. To see this, set
$y(s)= x(\beta s)$ and $\Psi (y)=\beta^2 \Phi (y)$, and take $t=\beta s$ in \eqref{beta-variant} to obtain
$$
 \ddot{x}(\beta s) + \frac{\alpha}{\beta s} \dot{x}(\beta s) + \beta \nabla^2 \Phi (x(\beta s))\dot{x} (\beta s) + \nabla \Phi (x(\beta s)) = 0.
$$
 Since $\dot{y}(s)= \beta\dot{x}(\beta s)$, and $\ddot{y}(s)= \beta^2\ddot{x}(\beta s)$, we obtain
$$
\frac{1}{\beta^2} \ddot{y}(s) +  \frac{\alpha}{\beta^2 s} \dot{y}(s) + \nabla^2 \Phi (y(s))\dot{y} (s) + \nabla \Phi (y(s)) = 0. 
$$
Equivalently,
$$\ddot{y}(s) +  \frac{\alpha}{s} \dot{y}(s) +  \nabla^2 \Psi (y(s))\dot{y} (s) + \nabla \Psi (y(s))= 0,$$
which corresponds to (DIN-AVD) with $\beta=1$.\\

On the other hand, a closer look at the proof of Propositions \ref{P:speed_to_0} and \ref{P:tGrad_L2} reveals that the estimations concerning $\|\nabla\Phi(x(t))\|$ degenerate and become meaningless as $\beta\to 0$. In this sense, the transition between the (essentially constant) case $\beta>0$ and the singular case $\beta=0$ is abrupt.

\subsubsection{Advantages of the case $\beta>0$}

The system (DIN-AVD) presents several advantages with respect to (AVD). We shall briefly comment some of them:\\

\noindent\underline{Estimations for $\nabla\Phi$ on the trajectory.} The quantity $\nabla\Phi(x(t))\|$ has the following additional properties:
\begin{itemize}
	\item If $\alpha>0$ and $\inf\Phi>-\infty$, then $\lim_{t\to+\infty}\|\nabla\Phi(x(t))\|=0$. This property is not known for (AVD).
	\item If $\alpha\geq3$ and $\argmin\Phi\neq\emptyset$, then $\displaystyle\int_{t_0}^{\infty}t^2\|\nabla\Phi(x(t))\|^2\,dt<+\infty$. Observe that, when $\ddot{x}$ is bounded, it is roughly equivalent to saying that $\|\nabla\Phi(x(t))\|\to 0$ strictly faster than $t^{-\frac{3}{2}}$. This is a striking result, when compared with the $t^{-\frac{1}{2}}$ rate of convergence in the case of the continuous steepest descent.	\bigskip 
\end{itemize}

\noindent\underline{Acceleration decay.} Assume $\alpha\geq3$ and $\argmin\Phi\neq\emptyset$. If, moreover, $\nabla\Phi$ is Lipschitz-continuous on bounded sets, then $\lim_{t\to +\infty}\|\ddot{x}(t)\|=0$.\bigskip

\noindent\underline{Apparent vs. actual complexity. Extension to the nonsmooth setting.} At a first glance, one may believe that the introduction of the Hessian-driven damping term brings an inherent additional complexity to the system, either in terms of the regularity required to establish existence and uniqueness of solutions, or in their actual computation. However, this turns out to be a misconception. Indeed, the presence of this additional term allows us to reformulate (DIN-AVD) as a first-order system both in time and space (see Subsection \ref{section-first-order}), called (g-DIN-AVD). This fact has two remarkable consequences:\medskip 
\begin{itemize}
	\item As we show in Subsection \ref{section-gen- existence}, existence and uniqueness of solution can be established by means of the perturbation theory developed in \cite{Bre1}, even when the potential function is just proper, convex and lower-semicontinuous. By contrast, it is  difficult to handle (AVD) with a nonsmooth $\Phi$, because the trajectories may exhibit shocks, and uniqueness is not guaranteed, see \cite{ACR}.\medskip 
	\item By considering structured potentials $\Phi+\Psi$, with $\Phi$ is smooth and $\Psi$ is not, an explicit-implicit discretization of (g-DIN-AVD) gives rise to new inertial forward-backward algorithms. Recall (from \cite{SBC} and \cite{APR1}) that a similar argument provides a connection between (AVD) and forward-backward algorithms accelerated by means of Nesterov's scheme (such as FISTA). If the asymptotic properties of (DIN-AVD) are preserved by this discretization, one may reasonably expect the resulting algorithms to outperform FISTA-like methods. This issue goes beyond the scope of the present paper and will be addressed in the future.
\end{itemize}

\subsubsection{A simple example to compare (AVD) an (DIN-AVD)}

We know, from \cite{SBC} and \cite{APR1}, that (AVD) is linked with accelerated forward-backward methods (by means of Nesterov's scheme). Let us compare the behavior of (AVD) and (DIN-AVD) in a simple example. Let $\Phi:\R\to\R$ be defined by $\Phi(x)=\frac{1}{2}x^2$, and take $\alpha>3$ and $\beta>0$. For simplicity, we shall also suppose that $\alpha\in\N$. Observe that $\Phi$ is strongly convex and $\argmin\Phi=\{0\}$. In this context, we can use a symbolic differential computation software to determine explicit solutions for (AVD) and (DIN-AVD) in terms of special functions. We used WolframAlpha$^{\tiny \textregistered}$\ Computational Knowledge Engine{\tiny \texttrademark}, available at {\tt http://www.wolframalpha.com/}.\\ 

\noindent\underline{AVD:} In this case, (AVD) becomes
$$\ddot x(t)+\frac{\alpha}{t}\dot x(t)+x(t)=0,$$
whose solutions are of the form
	$$x(t)=t^{\frac{1-\alpha}{2}}\left[c_1J_{\frac{\alpha-1}{2}}(t)+c_2Y_{\frac{\alpha-1}{2}}(t)\right],$$
where $c_1,c_2\in\R$ are constants depending on the initial conditions, and $J_\gamma$ and $Y_\gamma$ are the Bessel functions of the first and the second kind, respectively, with parameter $\gamma$. Since $|J_\gamma(t)|=\mathcal O(t^{-\frac{1}{2}})$ and $|Y_\gamma(t)|=\mathcal O(t^{-\frac{1}{2}})$ (see \cite[Section 5.11]{Lebedev}), we deduce that 
$$|x(t)|=\mathcal O(t^{-\frac{\alpha}{2}}).$$ 
This speed of convergence is faster than the one predicted in \cite[Theorem 3.4]{APR1}, namely $|x(t)|=\mathcal O(t^{-\frac{\alpha}{3}})$, but it is still a power of $t$.\\

\noindent\underline{DIN-AVD:} In turn, (DIN-AVD) is written as
$$\ddot x(t)+\left(\beta+\frac{\alpha}{t}\right)\dot x(t)+x(t)=0,$$
and its solutions are of the form
$$x(t)=e^{-\frac{t}{2}\left(\sqrt{\beta^2-4}+\beta\right)}\left[d_1 U\left(\delta,\alpha,\sqrt{\beta^2-4}t\right)+d_2L_{-\delta}^{\alpha-1}\left(\sqrt{\beta^2-4}t\right)\right],$$
where $\delta:=\frac{\alpha}{2}\left(\frac{\beta}{\sqrt{\beta^2-4}}+1\right)$, $U(a,b;z)$ is the confluent hypergeometric function of the second kind with parameter $(a,b)$ at the point $z$, $L_{-\delta}^{\alpha-1}$ is the associated Laguerre Polynomial of degree $\alpha-1$ and parameter $-\delta$, and $d_1,d_2\in\R$. But $|U(a,b;z)|=\mathcal O(|z|^{-a})$ as $|z|\to+\infty$ (see \cite[Section 9.12]{Lebedev}). Therefore, the worst-case speed of convergence is
$$|x(t)|=\mathcal O\left(t^{\alpha-1}e^{-\frac{\beta t}{2}}\right).$$
Observe that, if $d_2=0$ (which will depend on the initial conditions), then
$$|x(t)|=\mathcal O\left(t^{-\frac{\alpha}{2}}e^{-\frac{\beta t}{2}}\right).$$

\if{
If $\beta\ge 2$, then $\displaystyle |x(t)|=\mathcal O\left(t^{\alpha-1}e^{-\frac{t}{2}\left(\sqrt{\beta^2-4}+\beta\right)}\right)$
in general, and $\displaystyle |x(t)|=\mathcal O\left(t^{-\frac{\alpha\left(\sqrt{\beta^2-4}+\beta\right)}{2\sqrt{\beta^2-4}}}e^{-\frac{t}{2}\left(\sqrt{\beta^2-4}+\beta\right)}\right)$
if $d_2=0$.
}\fi

\begin{illustration} \label{I1}
We consider the function $\Phi:\R\to\R$ defined by $\Phi(x)=\frac{1}{2}x^2$. In Figure \ref{1D}, we show $(t,x(t))$ and $(t,\Phi(x(t)))$ for $t\in[1,20]$ with initial conditions $x(1)=1$ and $\dot x(1)=-3$. The parameters taken were $\alpha=3.1$ and, for (DIN-AVD), $\beta=1$. In both cases, the trajectories and the function values converge to the global minimum $0$ and the optimal value $0$, respectively.
\end{illustration}
\begin{figure}[h]
\begin{center}
\begin{tabular}{ccc}
\includegraphics[width=40mm]{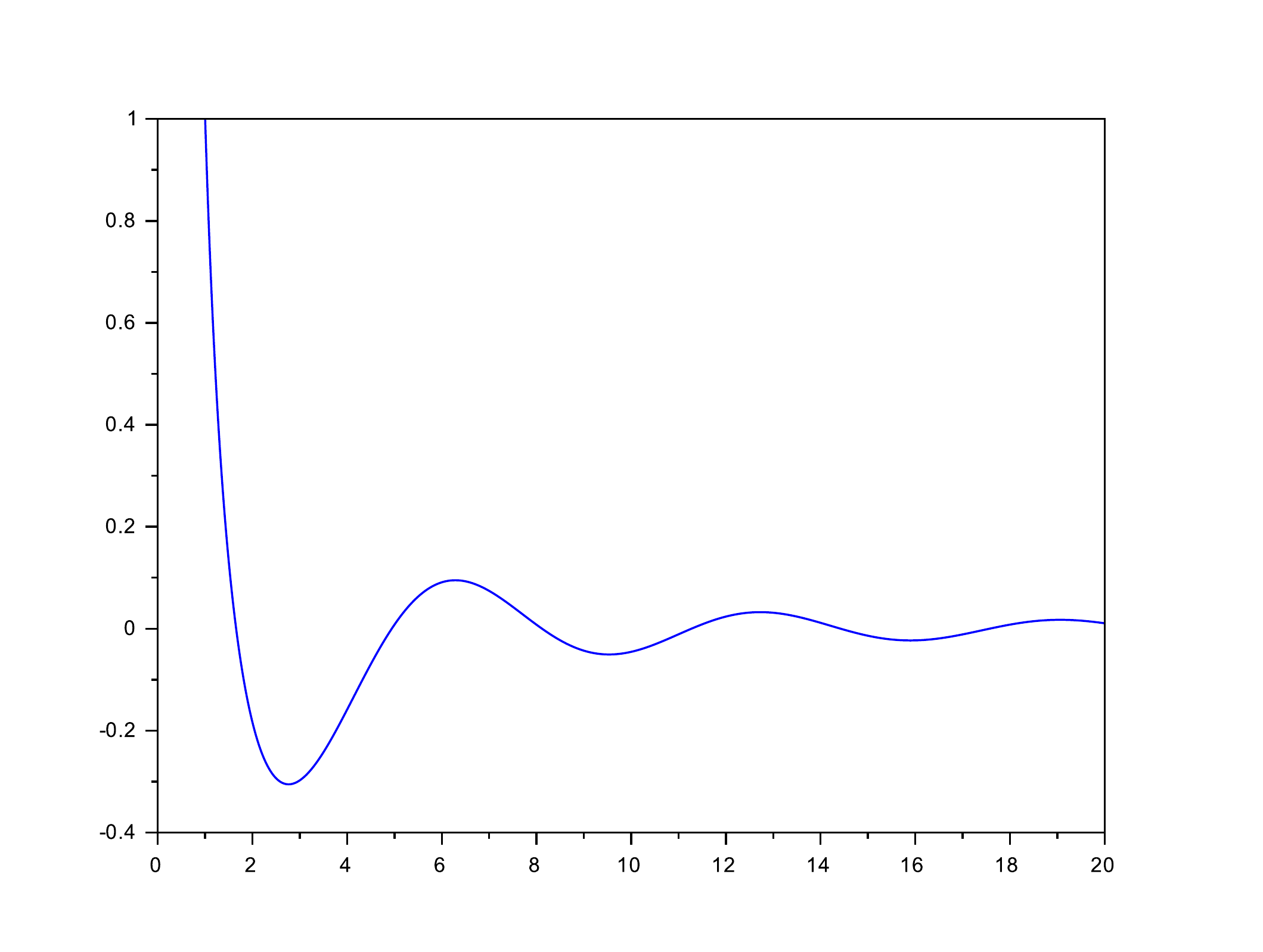} &
\includegraphics[width=40mm]{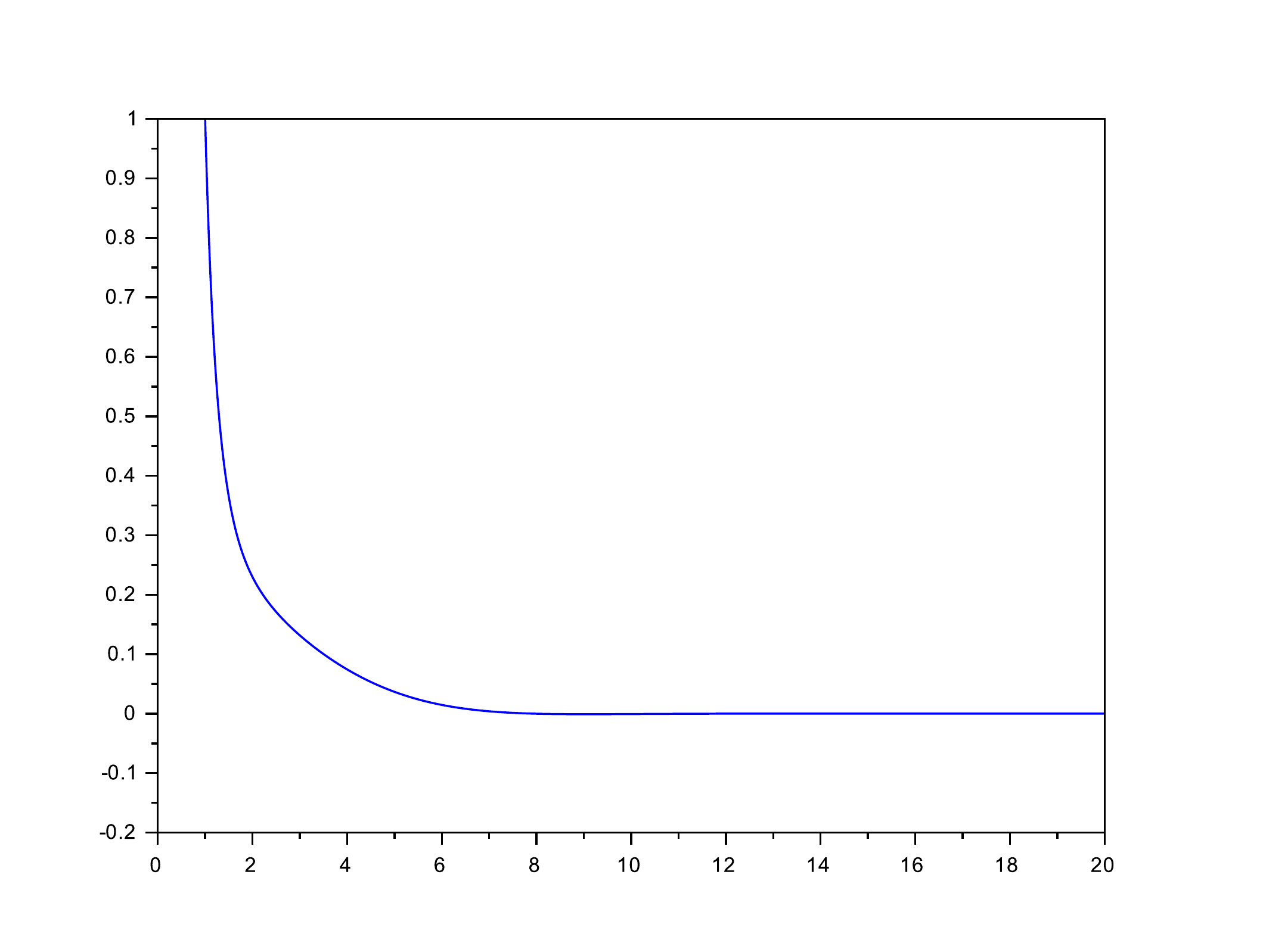} &
\includegraphics[width=40mm]{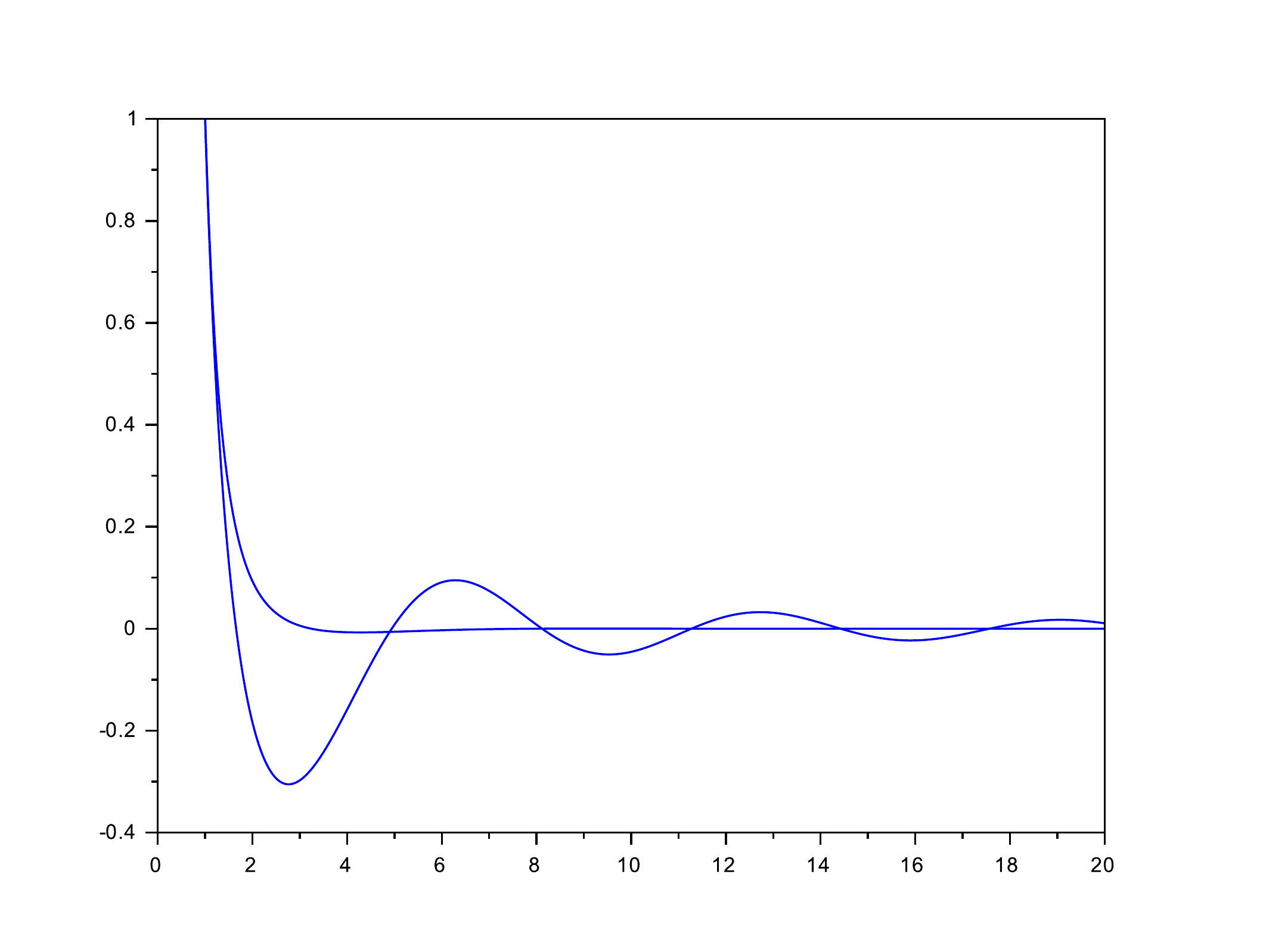}\\
\includegraphics[width=40mm]{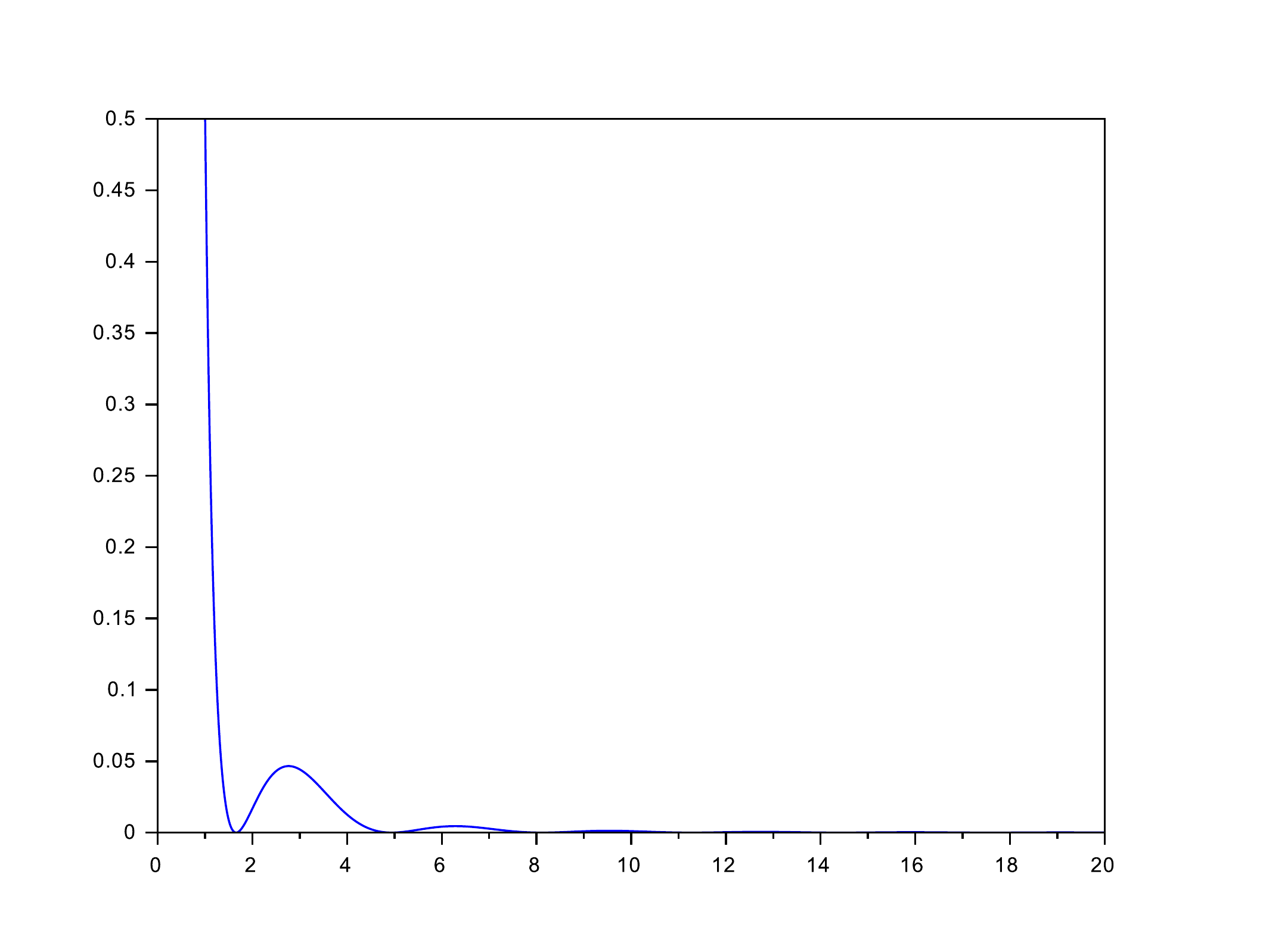} &
\includegraphics[width=40mm]{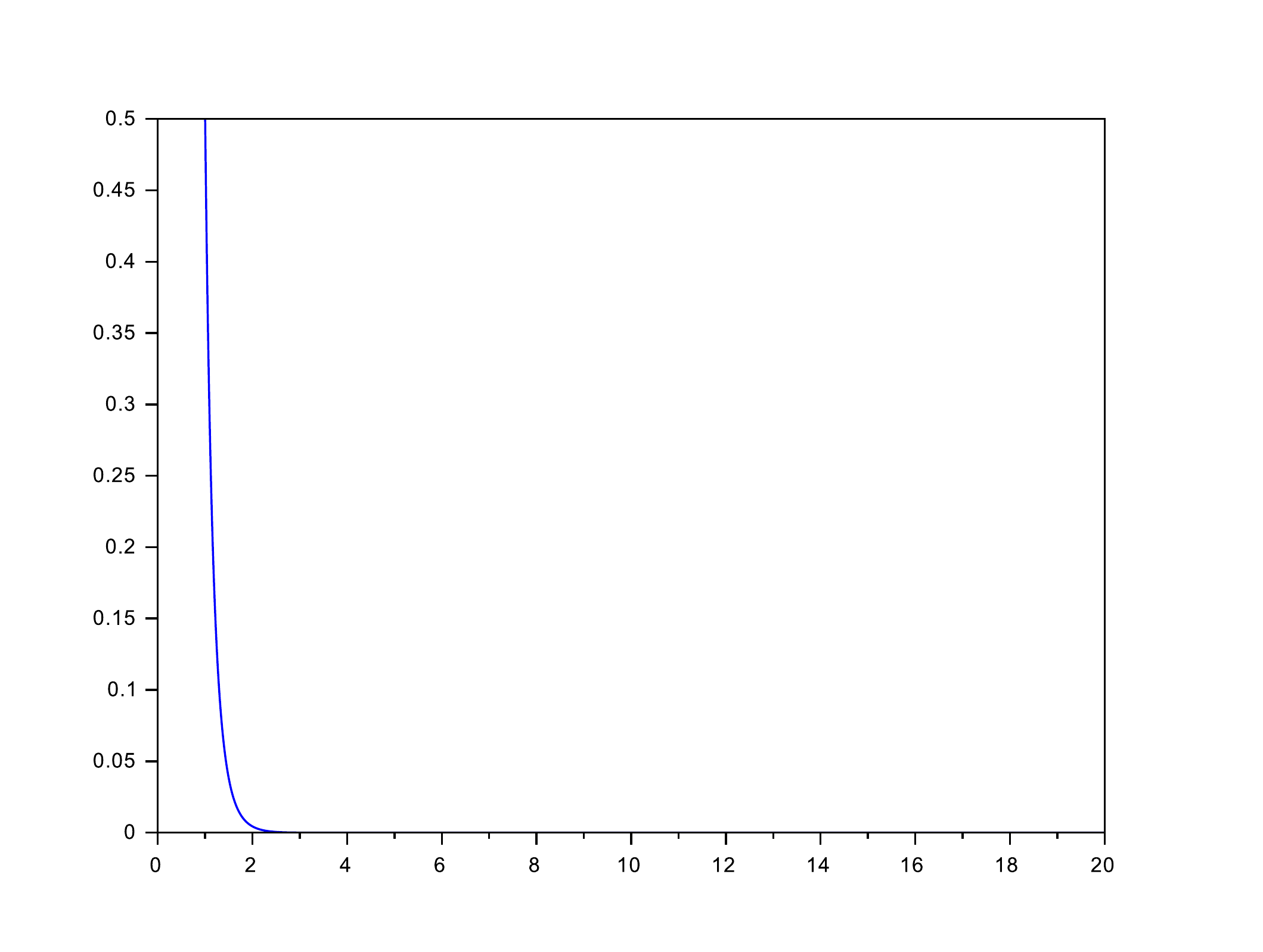} &
\includegraphics[width=40mm]{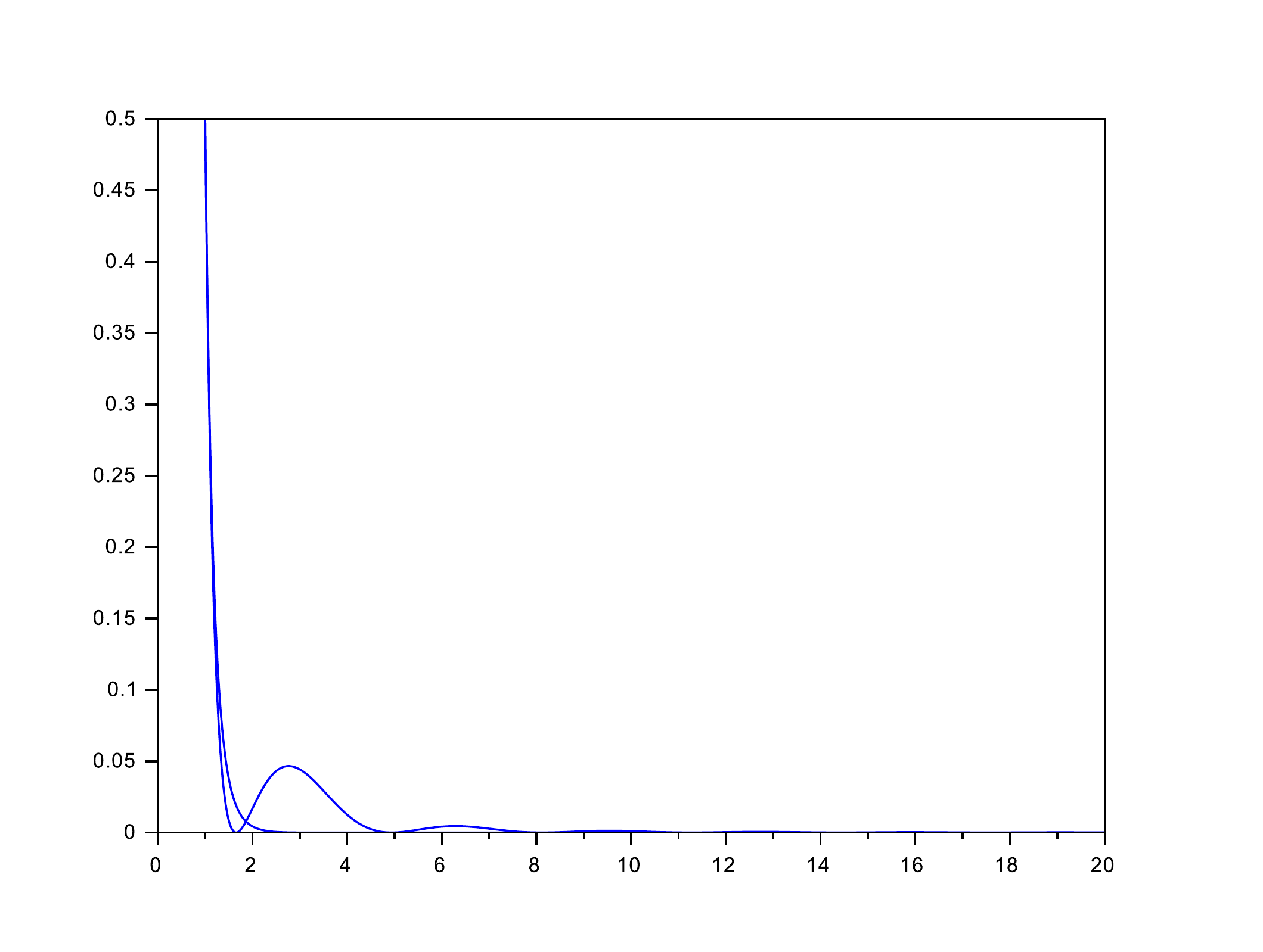}
\end{tabular}
\end{center}
\caption{Up: $(t,x(t))$ for AVD (left), DIN-AVD (middle), both (right). Down: $(t,\Phi(x(t)))$ for AVD (left), DIN-AVD (middle), both (right).}
\label{1D}
\end{figure}

\begin{illustration} \label{I2}
Now, we consider the function $\Phi(x,y)=\frac{1}{2}(x^2+1000y^2)$, which is still quadratic but not well conditioned. Figure \ref{2D} shows the curves $(x(t),y(t))$ and $(t,\Phi(x(t),y(t)))$. As before, we show the behavior on the interval $[1,20]$ with $\alpha=3.1$ and, for (DIN-AVD), $\beta=1$. The initial conditions were $(x(1),y(1))=(1,1)$ and $(\dot x(1),\dot y(1))=(0,0)$. In both cases, the trajectories and the function values converge to the global minimum $(0,0)$ and the optimal value $0$, respectively. However, the wild transversal oscillation exhibited by the solution of (AVD) are neutralized by (DIN-AVD).
\end{illustration}
 
\begin{figure}[h]
\begin{center}
\begin{tabular}{ccc}
\includegraphics[width=40mm]{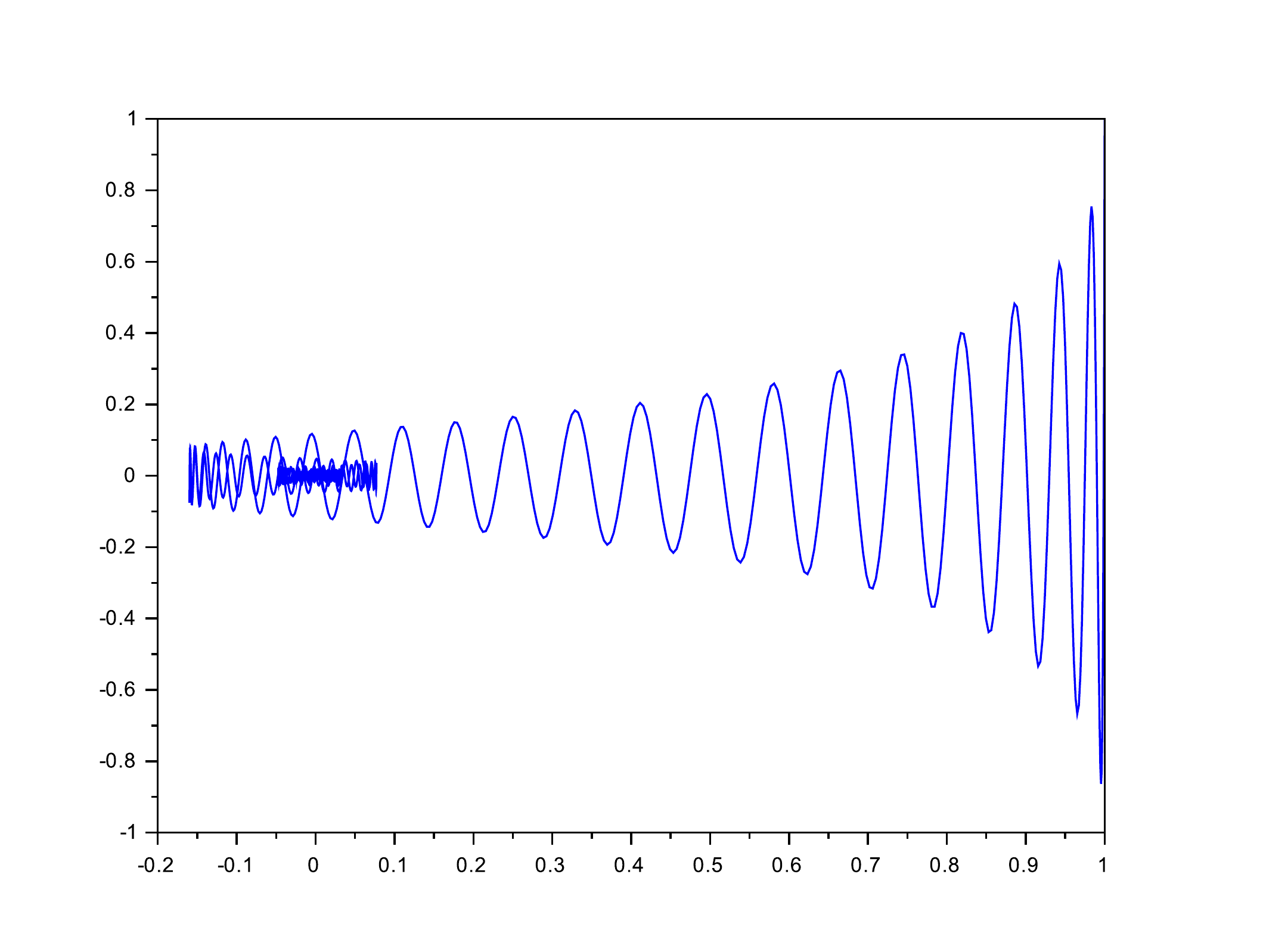} &
\includegraphics[width=40mm]{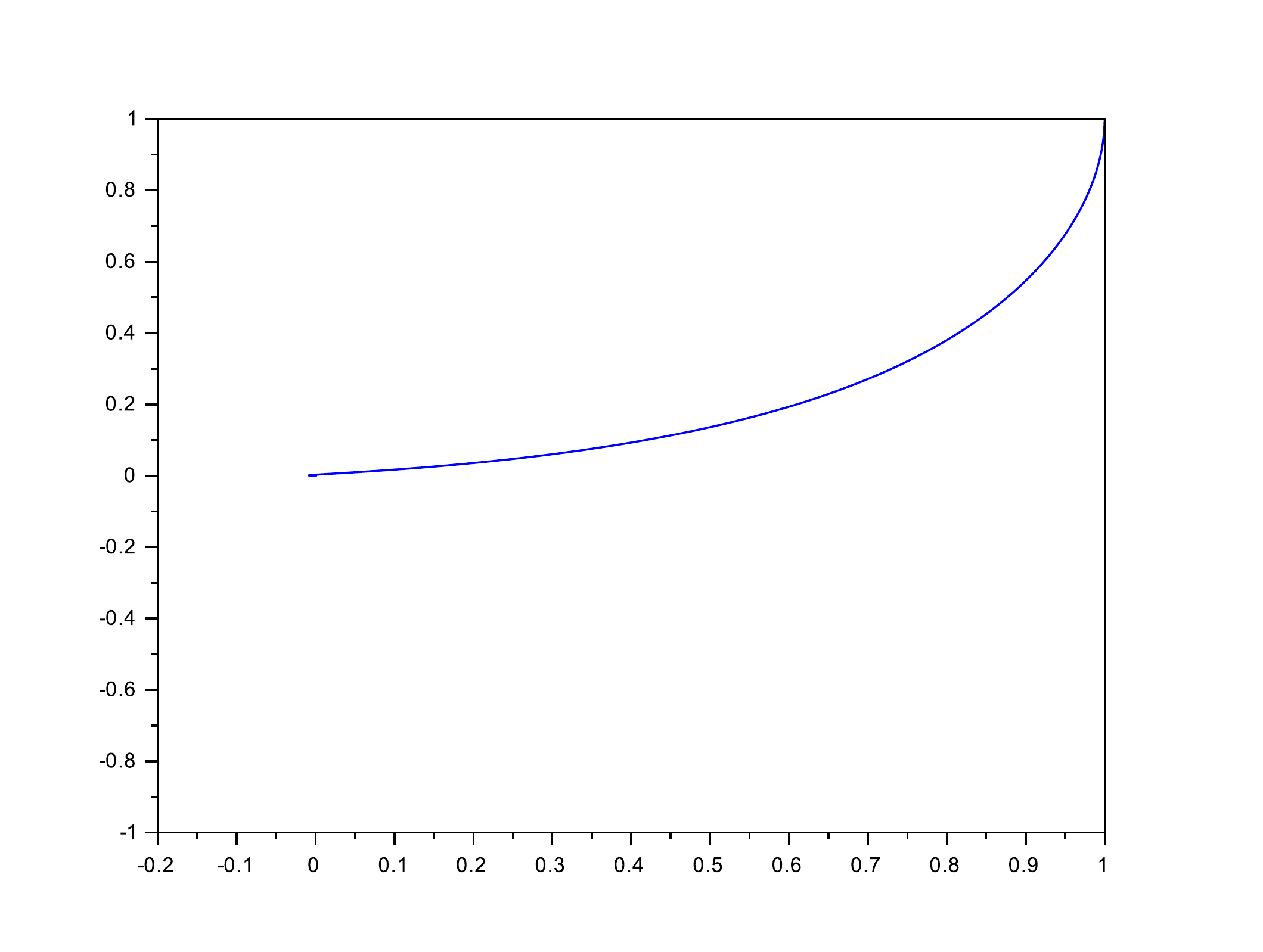} &
\includegraphics[width=40mm]{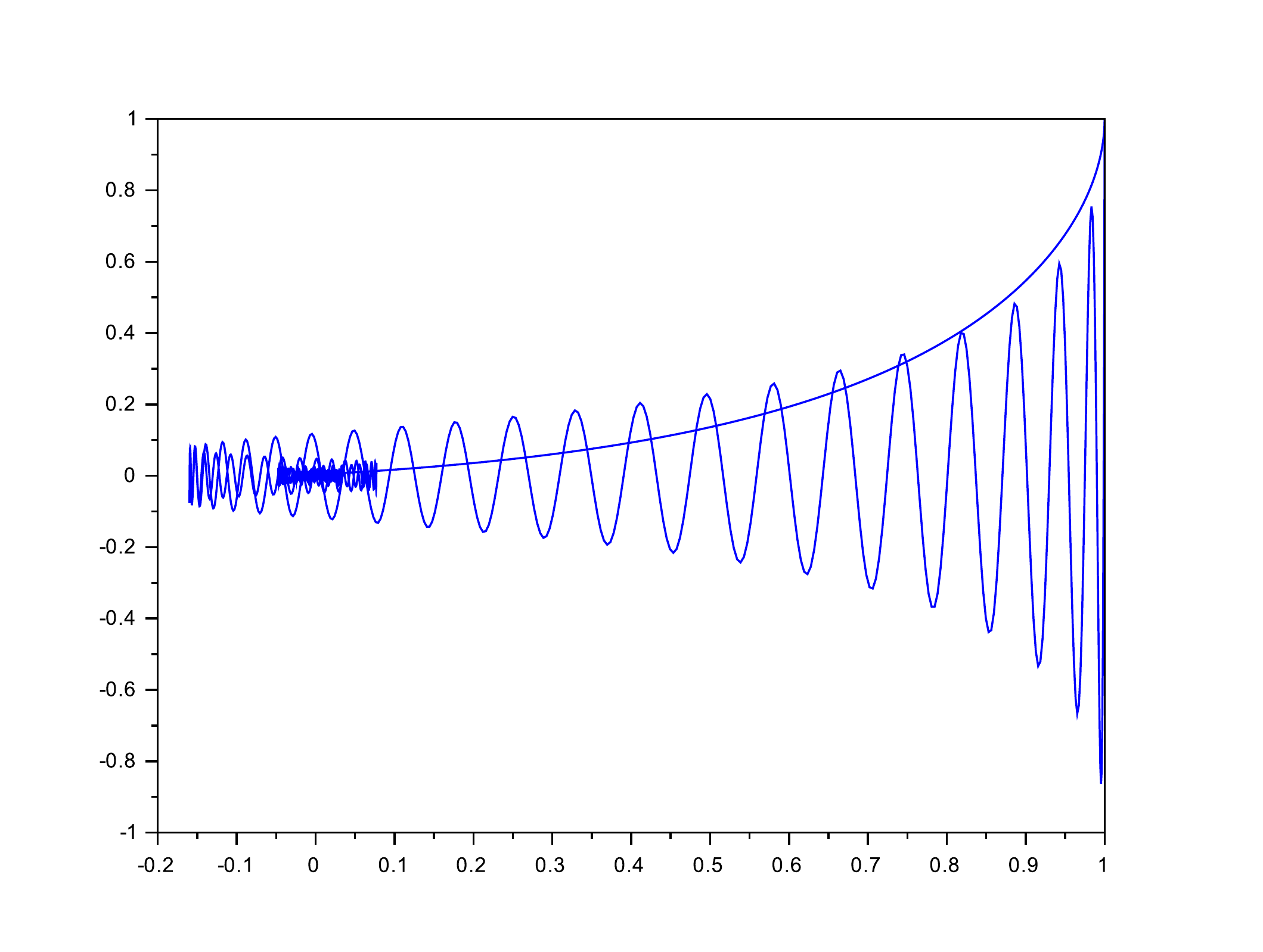}\\
\includegraphics[width=40mm]{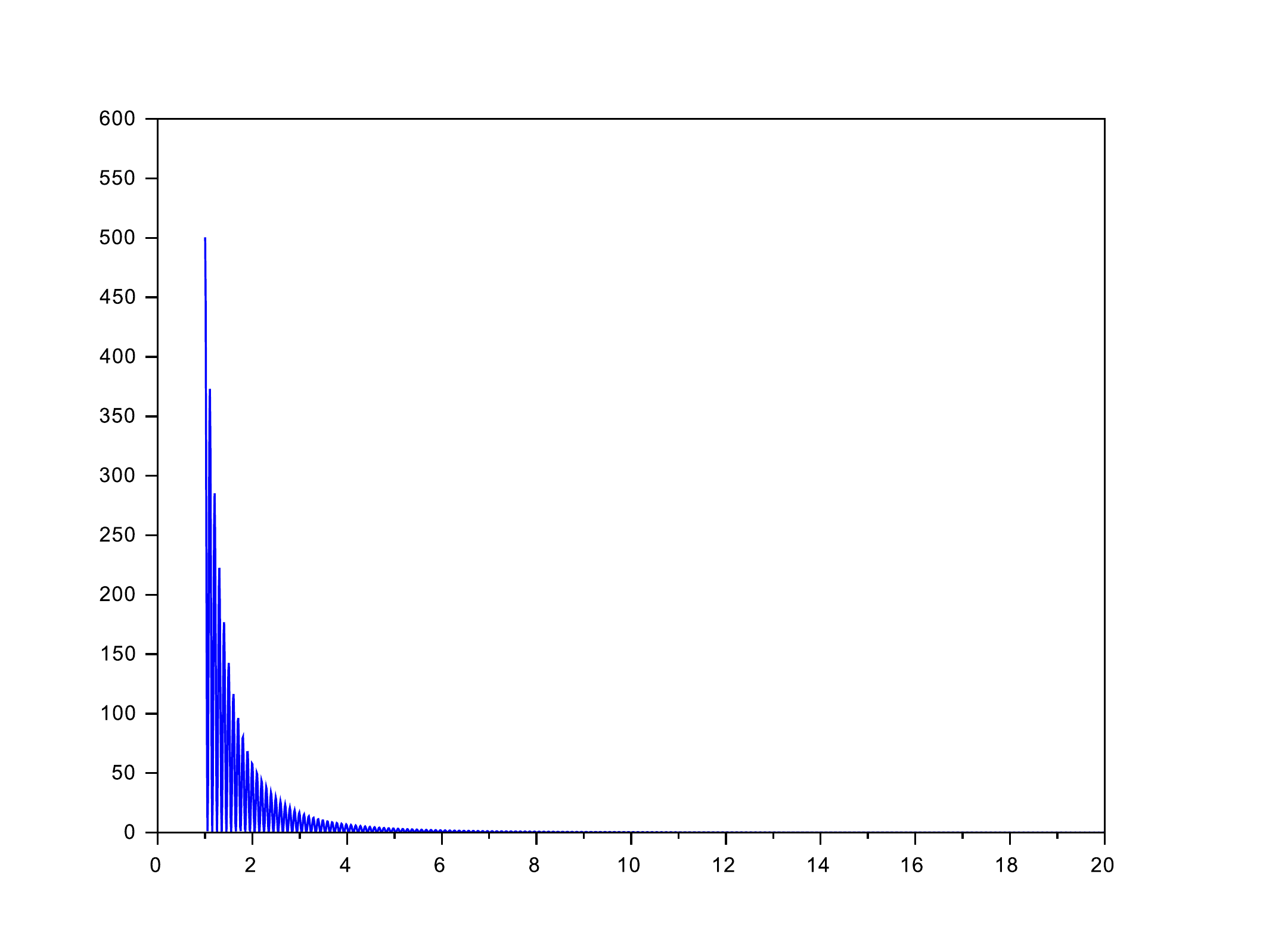} &
\includegraphics[width=40mm]{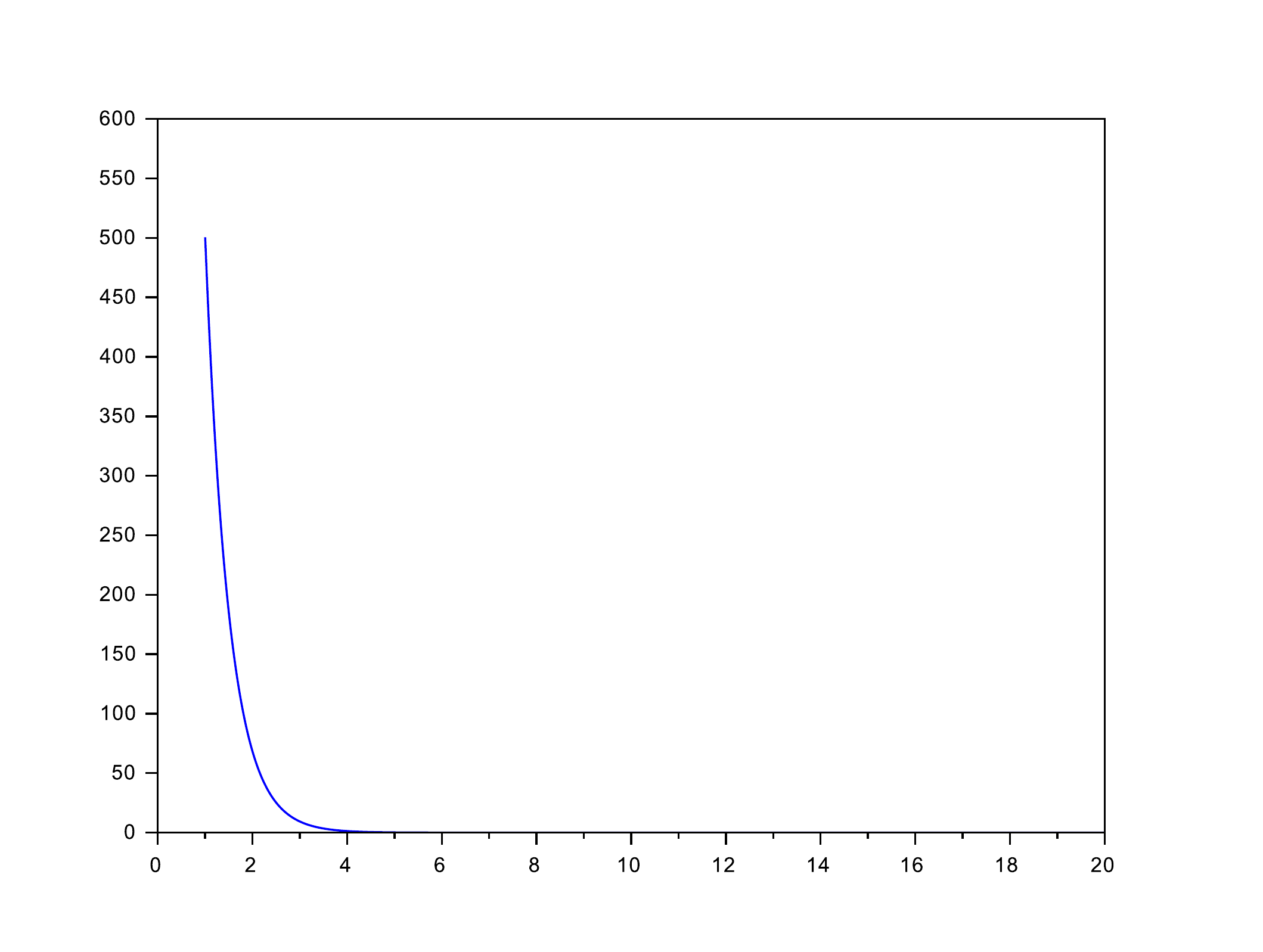} &
\includegraphics[width=40mm]{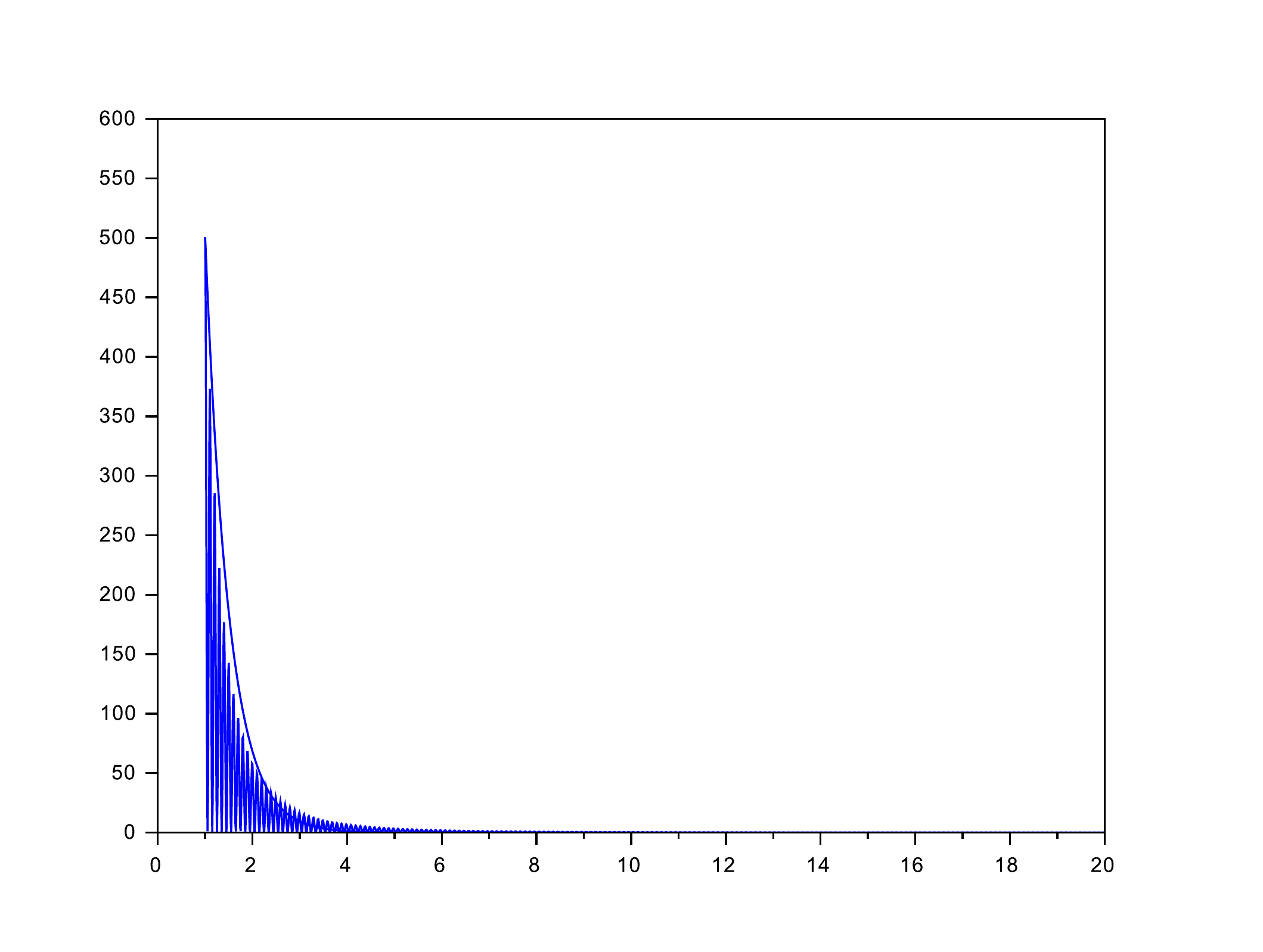}
\end{tabular}
\end{center}
\caption{Up: $(x(t),y(t))$ for AVD (left), DIN-AVD (middle), both (right). Down: $(t,\Phi(x(t),y(t)))$ for AVD (left), DIN-AVD (middle), both (right).}
\label{2D}
\end{figure}

\section{Strong convergence results} \label{S-conv-forte}
In this section, we establish strong convergence of the trajectories in several relevant cases, namely: when $\Phi$ is even, uniformly convex, boundedly inf-compact, or if $\hbox{int}(\argmin\Phi)\neq\emptyset$.

\subsection{Even objective function}

Recall that $\Phi$ is even if $\Phi(x)=\Phi(-x)$ for all $x\in \mathcal H$.

\if{
\tcr{Je propose de supprimer le th. qui suit pour le remplacer par le th. 
d'apr\`es qui passe au cas irr\'egulier.}
\begin{theorem} \label{Thm-strong-conv}
Suppose $\alpha>3$, $\beta>0$, and $\Phi$ is even. Let $x: [t_0, +\infty[ \rightarrow \mathcal H$ be a  classical global solution  of (DIN-AVD). 
Then, there exists some  $x^{*}  \in \argmin \Phi$ such that
\begin{equation}\label{strong-conv-basic}
  x(t) \rightarrow x^{*} \ \mbox{strongly in }  \mathcal H   \mbox{  as} \ t \rightarrow + \infty.
\end{equation}
\end{theorem}
\begin{proof} 
First notice $\Phi(0)=\min\Phi$. Indeed 
$\Phi(0)\leq(\Phi(x)+\Phi(-x))/2=\Phi(x)$ for all $x\in\Hb$.

Let $t_1>t_0$ and for $t_0\leq t\leq t_1$ define following function of $t$
\begin{equation*}
g(t)=\|x(t)\|^2-\|x(t_1) \|^2-\frac{1}{2}\|x(t)-x(t_1)\|^2.
\end{equation*}
We have
\begin{align*}
\dot g(t) & =\langle\dot x(t),x(t)+ x(t_1)\rangle,\\
\ddot g(t) &=\|\dot x(t)\|^2+\langle\ddot x(t),x(t)+x(t_1)\rangle.
\end{align*}
Combining the two equations above,  and using (DIN-AVD) we obtain
\begin{align}
\ddot g(t)+\frac{\alpha}{t}\dot g(t) 
&=\|\dot x(t)\|^2+\langle\ddot x(t)+\frac{\alpha}{t}\dot x(t),
  x(t)+x(t_1)\rangle 
  \nonumber\\
&=\|\dot x(t)\|^2-\langle\beta\nabla^2\Phi(x(t))\dot x(t)+\nabla\Phi(x(t)),
  x(t)+x(t_1)\rangle. 
  \label{even3}
\end{align}
By the classical chain derivation rule
\begin{align}
\frac{d}{dt}\langle\nabla\Phi(x(t)),x(t)+x(t_1)\rangle 
&=\langle\nabla^2\Phi(x(t))\dot x(t),x(t)+x(t_1)\rangle 
  +\langle\nabla\Phi(x(t)),\dot x(t)\rangle
  \nonumber \\
&=\langle\nabla^2\Phi(x(t))\dot x(t),x(t)+x(t_1)\rangle 
  +\frac{d}{dt}\left(\Phi(x(t))-\min\Phi\right). 
  \label{even30}
\end{align}
Define 
\begin{eqnarray*}
\omega(t) & = & -\langle\nabla\Phi(x(t)),x(t)+x(t_1)\rangle \\
\varphi(t) & = & \Phi(x(t))-\min\Phi.
\end{eqnarray*}
Combining (\ref{even3}) with (\ref{even30}) we obtain
\begin{equation}\label{even31}
\ddot g(t)+\frac{\alpha}{t}\dot g(t)= 
\|\dot x(t)\|^2+\beta\dot\varphi(t)+\beta\dot\omega(t)+\omega(t).
\end{equation}
Recall that the energy function $W_0(t)=\frac{1}{2}\|\dot x(t)\|^2+\Phi(x(t))$ 
is nonincreasing (Theorem \ref{Thm-weak-conv20}\eqref{weak-conv20-1}). 
Comparing the values of $W$ at $t$ and $t_1$, and  successively using the 
evenness property of $\Phi$, and the convex differential inequality, we obtain
\begin{align*}
\frac{1}{2}\|\dot x(t)\|^2+\Phi(x(t)) 
&\geq\frac{1}{2}\|\dot x(t_1)\|^2+\Phi(x(t_1)) \\
&\geq\frac{1}{2}\|\dot x(t_1)\|^2+\Phi(-x(t_1)) \\
&\geq\frac{1}{2}\|\dot x(t_1)\|^2+\Phi(x(t))
  -\langle\nabla\Phi(x(t)),x(t)+x(t_1)\rangle. 
\end{align*}
After simplification, and using $\frac{1}{2}\|\dot x(t_1)\|^2\leq0$, we obtain
\begin{equation}\label{even6}
\omega(t)\leq\frac{1}{2}\|\dot x(t)\|^2.
\end{equation}
Multiplying (\ref{even31}) by $t$ and taking (\ref{even6}) into account, we 
obtain
\begin{equation*}
t\ddot g(t)+\alpha\dot g(t)\leq
\frac{3}{2}t\|\dot x(t)\|^2+\beta t\dot\varphi(t)+\beta t\dot\omega(t).
\end{equation*}
Integrate this inequality from $t$ to $t_1\geq t$
\begin{equation}\label{even12}
t_1\dot g(t_1)-t\dot g(t)-(\alpha-1)g(t)\leq
\frac{3}{2}\int_t^{t_1}s\|\dot x(s)\|^2ds+\beta\int_t^{t_1}s\dot\varphi(s)ds
  +\beta\int_t^{t_1}s\dot\omega(s)ds.
\end{equation}
Evaluate the last two integrals
\begin{align*}
& \int_t^{t_1}s\dot\varphi(s)ds
  =t_1\varphi(t_1)-t\varphi(t)-\int_t^{t_1}\varphi(s)ds 
  \leq t_1\varphi(t_1) 
& \\
& \int_t^{t_1}s\dot\omega(s)ds
  =t_1\omega(t_1)-t\omega(t)-\int_t^{t_1}\omega(s)ds. 
\end{align*}
In view of which we derive from \eqref{even12} (recalling the definition of 
$g$)
\begin{eqnarray*}
(\alpha-1)\frac{1}{2}\|x(t_1)-x(t)\|^2 
& \leq &
(\alpha-1)(\|x(t)\|^2-\|x(t_1)\|^2)
  -t_1(\dot g(t_1)-\beta\omega(t_1))+t(\dot g(t)-\beta\omega(t)) \\
& & 
  -\beta\int_t^{t_1}\omega(s)ds
  +\frac{3}{2}\int_t^{t_1}s\|\dot x(s)\|^2ds+\beta t_1\varphi(t_1).
\end{eqnarray*}
Now, we have the following convergence results, as $t$ and $t_1$ tend to 
$+\infty$ with $t\leq t_1$ \\
\-- with Lemma \ref{superconv}\eqref{superconv-3}, $\|x(t)\|$ has a limit,  
since $0\in\argmin\Phi$;
\\
\-- 
since 
$\|t(\dot x(t)+\beta\nabla\Phi(x(t)))\|\,\|x(t)+x(t_1)\|\geq
  |t\langle\dot x(t)+\beta\nabla\Phi(x(t)),x(t)+x(t_1)\rangle|=
  |t(\dot g(t)-\beta\omega(t))|$, 
this last quantity vanishes in view of Theorem \ref{T:superconv} and of the 
boundedness of $x$ (Theorem \ref{Thm-fast-conv2}\eqref{fast-conv2-3});
\\
\-- $\int_t^{t_1}\omega(s)ds$ vanishes because $\omega$ belongs to 
$L^1([t_0,+\infty[)$; indeed we have 
$|\omega(t)|\leq\|\nabla\Phi(x(t))\|\,\|x(t)+x(t_1)\|$ where 
$\|\nabla\Phi(x)\|$ belongs to $L^1([t_0,+\infty[)$ by Theorem 
\ref{Thm-fast-conv2}\eqref{fast-conv2-4}; \\
\-- $\int_t^{t_1}s\|\dot x(s)\|^2ds$ vanishes in view of lemma 
\ref{superconv}\eqref{superconv-2}; \\
\-- $t_1\varphi(t_1)$ vanishes since $\varphi(t_1)=\mathcal O(1/t_1^2)$ 
(Theorem \ref{Thm-fast-conv2}\eqref{fast-conv2-2}).

As a consequence, as $t \to + \infty$,  $x(t)$ satisfies the Cauchy criterion in the Hilbert space  $\mathcal H$.
\end{proof}
}\fi

\begin{theorem} \label{Thm-conv-forte}
Suppose $\alpha>3$, $\beta>0$ and let $\Phi$ be twice differentiable, convex and even. Let $x: [t_0, +\infty[ \rightarrow \mathcal H$ be the global 
classical solution to (DIN-AVD) with Cauchy data $x(t_0)=x_0$ and $\dot x(t_0)=\dot x_0$. Then, $x(t)$ converges strongly, as $t\to+\infty$, to some $ x^{*}\in \argmin \Phi$.
\end{theorem}
\begin{proof} 
Let $t_1>t_0$ and for $t_0\leq t\leq t_1$ define the following function of $t$
\begin{equation} \label{gpair}
g(t)=
  \|x(t)\|^2-\|x(t_1) \|^2-\frac{1}{2}\|x(t)-x(t_1)\|^2
  -\int_t^{t_1}\langle\beta\nabla\Phi(x(s)),x(s)+x(t_1)\rangle ds.
\end{equation}
We have
\begin{align*}
\dot g(t) & =\langle\dot u_\beta(t),x(t)+x(t_1)\rangle \\
\ddot g(t) &=
  \langle\ddot u_\beta(t),x(t)+x(t_1)\rangle+\langle\dot u_\beta(t),\dot x(t)\rangle,
\end{align*}
where, we recall, $u_\beta$ is defined by \eqref{u0reg} with $\theta=\beta$. Combining the two equations 
above, and using \eqref{uter} we obtain
\begin{equation} \label{pair3}
\ddot g(t)+\frac{\alpha}{t}\dot g(t)= 
  \langle\dot u_\beta(t),\dot x(t)\rangle
    -\left(1-\frac{\alpha\beta}{t}\right)
      \langle\nabla\Phi(x(t)),x(t)+x(t_1)\rangle. 
\end{equation}
The energy function $W_\beta(t)=\frac{1}{2}\|\dot u_\beta(t)\|^2+\Phi(x(t))$ is 
nonincreasing by Proposition \ref{P:Lyapunov}. Comparing the values of $W_\beta$ at $t$ and $t_1$, and successively using the fact that
$\Phi$ is even along with the convex differential inequality, we obtain
\begin{align*}
\frac{1}{2}\|\dot u_\beta(t)\|^2+\Phi(x(t)) 
&\geq\frac{1}{2}\|\dot u_\beta(t_1)\|^2+\Phi(x(t_1)) \\
& = \frac{1}{2}\|\dot u_\beta(t_1)\|^2+\Phi(-x(t_1)) \\
&\geq\frac{1}{2}\|\dot u_\beta(t_1)\|^2+\Phi(x(t))
  -\langle\nabla\Phi(x(t)),x(t)+x(t_1)\rangle. 
\end{align*}
After simplification, we obtain
$$
-\langle\nabla\Phi(x(t)),x(t)+x(t_1)\rangle\leq\frac{1}{2}\|\dot u_\beta(t)\|^2.
$$
Since we are interested in the asymptotic behaviour of $x$, there is no harm 
in supposing $t\geq\alpha\beta$; so we deduce from the inequality above that
$$
-\left(1-\frac{\alpha\beta}{t}\right)
  \langle\nabla\Phi(x(t)),x(t)+x(t_1)\rangle
\leq
\left(1-\frac{\alpha\beta}{t}\right)\frac{1}{2}\|\dot u_\beta(t)\|^2
\leq
\frac{1}{2}\|\dot u_\beta(t)\|^2.
$$
Then, equality \eqref{pair3} yields
$$
\ddot g(t)+\frac{\alpha}{t}\dot g(t)\leq
  \langle\dot u_\beta(t),\dot x(t)\rangle+\frac{1}{2}\|\dot u_\beta(t)\|^2.
$$
Using \eqref{uniter}, it ensues that 
\begin{eqnarray*}
\ddot g(t)+\frac{\alpha}{t}\dot g(t)
& \leq &
  \frac{3}{2}\|\dot x(t)\|^2+2\langle\beta\nabla\Phi(x(t)),\dot x(t)\rangle
  +\frac{1}{2}\|\beta\nabla\Phi(x(t))\|^2
\\
& \leq &
  \frac{5}{2}\|\dot x(t)\|^2+\frac{3}{2}\|\beta\nabla\Phi(x(t))\|^2.
\end{eqnarray*}
Multiply the latter inequality by t and integrate from $t$ to $t_1\geq t$ to obtain
$$
t_1\dot g(t_1)-t\dot g(t)-(\alpha-1)g(t)\leq
  \frac{5}{2}\int_t^{t_1}s\|\dot x(s)\|^2ds
  +\frac{3}{2}\int_t^{t_1}s\|\beta\nabla\Phi(x(s))\|^2ds.
$$
Let us take the definition of $g$ \eqref{gpair} into account to obtain
\begin{eqnarray*}
(\alpha-1)\frac{1}{2}\|x(t)-x(t_1)\|^2 
& \leq & 
(\alpha-1)(\|x(t)\|^2-\|x(t_1)\|^2)
-(\alpha-1)\int_t^{t_1}\langle\beta\nabla\Phi(x(s)),x(s)+x(t_1)\rangle ds
\\
& & 
  +(t\dot g(t)-t_1\dot g(t_1))
  +\frac{5}{2}\int_t^{t_1}s\|\dot x(s)\|^2ds
  +\frac{3}{2}\int_t^{t_1}s\|\beta\nabla\Phi(x(s))\|^2ds.
\end{eqnarray*}
Now, we have the following convergence results, as $t$ and $t_1$ tend to 
$+\infty$ with $t\leq t_1$:
\begin{itemize}
	\item Since $0\in\argmin\Phi$, part (iii) of Lemma \ref{superconv} implies $\|x(t)\|$ has a limit as $t\to+\infty$. Therefore, 
	$(\|x(t)\|^2-\|x(t_1)\|^2)$ vanishes;
	\item Proposition \ref{P:tGrad_L2} implies $t\mapsto\|\nabla\Phi(x(t))\|=\frac{1}{t}\times t\|\nabla\Phi(x(t))\|$ belongs to
	$L^1(t_0,+\infty)$, as a product of functions in $L^2(t_0,+\infty)$. Since $x$ is bounded, $s\mapsto \langle\beta\nabla\Phi(x(s)),x(s)+x(t_1)$ is in $L^1(t_0,+\infty)$, and so
	$\int_t^{t_1}\langle\beta\nabla\Phi(x(s)),x(s)+x(t_1)\rangle ds$ vanishes; 
	\item $|t\dot g(t)|=|t\langle\dot x(t)+\beta\nabla\Phi(x(t)),x(t)+x(t_1)\rangle|$, 
	this last quantity vanishes in view of Theorem \ref{T:superconv} and the boundedness of $x$;
	\item $\int_t^{t_1}s\|\dot x(s)\|^2ds$ vanishes in view of part (i) of Lemma \ref{superconv}; 
	\item $\int_t^{t_1}s\|\beta\nabla\Phi(x(s))\|^2ds$ vanishes in view of Proposition \ref{P:tGrad_L2}.
\end{itemize}

As a consequence, as $t \to + \infty$,  $x(t)$ satisfies the Cauchy criterion 
in the Hilbert space  $\mathcal H$. The limit obviously is a minimum point by Corollary \ref{C:limitpoints}.
\end{proof}

\subsection{Solution set with nonempty interior} In this subsection, we examine the case where $\hbox{int}(\argmin \Phi) \neq \emptyset $.
\begin{theorem} \label{Thm-strong-int}
Suppose $\alpha >3$, $\beta >0$ and let $\Phi$ satisfy 
{\rm int}$(\argmin \Phi) \neq \emptyset $.
Let $x$ be a  classical global solution  of (DIN-AVD).  Then, $x(t)$ converges strongly, as $t\to+\infty$, to some $ x^{*}\in \argmin \Phi$. Moreover,
$$\int_{t_0}^{\infty}  t\| \nabla \Phi (x(t))  \|dt < +\infty.$$
\end{theorem}
\begin{proof}
Since int$(\argmin\Phi)\neq\emptyset$, there exist $\xi^*\in\argmin\Phi$ and 
$\rho>0$ such that$\{z\in\Hb,\;\|z-\xi^*\|<\rho\}\subseteq\argmin\Phi$. 
According to the monotonicity of $\nabla \Phi$, for any $y\in\Hb$ and any 
$z\in\Hb$ such that $\|z-\xi^*\|<\rho$, we have
\begin{equation*}
\langle \nabla \Phi (y), y- z \rangle \geq 0.
\end{equation*}
Hence
\begin{equation*}
\langle \nabla \Phi (y), y- \xi^* \rangle \geq \langle \nabla \Phi (y), z- \xi^* \rangle .
\end{equation*}
Taking the supremum with respect to $z \in \mathcal H$ such that $\|z- \xi^*\| < \rho$, we infer that for any $ y \in \mathcal H $, we have
\begin{equation}\label{strong-conv-int6}
\langle \nabla \Phi (y), y- \xi^* \rangle \geq \rho \|\nabla \Phi (y)\|.
\end{equation}
In particular taking $y= x(t)$, we obtain
\begin{equation}\label{strong-conv-int7}
\langle \nabla \Phi ( x(t)),  x(t)- \xi^* \rangle \geq \rho \|\nabla \Phi ( x(t))\|.
\end{equation}
From part (ii) of Lemma \ref{superconv}, we deduce that
\begin{equation}\label{strong-conv-int8}
\rho\int_{t_0}^\infty t\|\nabla\Phi(x(t))\|dt
  \leq\int_{t_0}^\infty t\langle x(t)-\xi^*,\nabla\Phi(x(t))dt
  <\infty.
\end{equation}
Multiply equation \eqref{uter} (with $\theta=\beta$) by $t$ to obtain
$$
t\ddot u_\beta(t)+\alpha\dot u_\beta(t)=(\alpha\beta-t)\nabla\Phi(x(t)),
$$
and integrate between $t_0$ and $t\geq t_0$ to conclude that
$$
t\dot u_\beta(t)+(\alpha-1)u_\beta(t)=
  t_0\dot u_\beta(t_0)+(\alpha-1)u_\beta(t_0)
  +\int_{t_0}^t(\alpha\beta-s)\nabla\Phi(x(s))ds.
$$
In view of \eqref{strong-conv-int8} the right-hand side has a limit as 
$t\to+\infty$. |With Lemma \ref{elemutil}, $u_\beta(t)=x(t)+\int_{t_0}^t\beta\nabla\Phi(x(s))ds$ has a limit as well. Hence $x(t)$ has 
a limit, which is a minimum point of $\Phi$.
\end{proof}

Let us notice that, thanks to the assumption $\mbox{int}(\argmin\Phi)\neq\emptyset $, we have been able to pass from the 
$L^2$ estimation of Proposition \ref{P:tGrad_L2} to an $L^1$ estimate for $t \mapsto  t\| \nabla \Phi (x(t))\|$.

\subsection{Bounded inf-compactness}\label{infc}
A function $\Phi:\mathcal H\rightarrow\mathbb R$ is {\it boundedly inf-compact} if, for any $\lambda\in\mathbb R$, and $R>0$, the set \
$
\left\lbrace x \in \mathcal H :\ \Phi(x) \leq \lambda , \ \|x\| \leq R \right\rbrace 
$
\ is relatively compact in $\mathcal H$.

\begin{theorem} \label{Thm-infc-basic}
Suppose $\alpha\geq3$ and that  $\Phi : \mathcal H \rightarrow \mathbb R$ is a boundedly inf-compact function with $\argmin \Phi\neq\emptyset$. Then,
every trajectory $x$ of (DIN-AVD) converges strongly, as $t\to+\infty$, to some $ x^{*}\in \argmin \Phi$.
\end{theorem} 

\begin{proof} 
The trajectory $x$ is minimizing by Theorem \ref{T:limit_W_Phi}, and bounded by Theorem \ref{T:Phi_t-2}. Consequently, the trajectory is contained in the 
intersection of a sublevel set of $\Phi$ with some ball, which must be a compact set, since $\Phi$ is boundedly inf-compact. The trajectory converges weakly and is contained in a compact set. Hence it converges strongly
to some $ x^{*}\in \argmin \Phi$. 
\end{proof}

\subsection{Uniform convexity/monotonicity}\label{unifconv}
Following \cite{BC}, we say that $\nabla \Phi $ is uniformly monotone on bounded sets, if, for each $r>0$, there exists a nondecreasing function $\omega_r:[0, + \infty[\to[0, + \infty[$ vanishing only at 0, and such that 
$$  
\langle\nabla\Phi(x)-\nabla\Phi(y), x-y  \rangle \geq \omega_r (\|x-y  \|)
$$
for all $x, y \in \mathcal H$ with $\| x \|\leq r, \| y \|\leq r$.

\begin{theorem} \label{Thm-strong-convex}
Suppose $\alpha >3$, $\beta>0$. Suppose that $\argmin \Phi\neq\emptyset$, and  $\nabla \Phi $ is uniformly monotone on bounded sets.
Let $x$ be a  classical global solution  of (DIN-AVD). Then, $\argmin \Phi$ is reduced to a singleton $x^{*}$, and $x(t)$ converges strongly to $x^*$, as $t\to+\infty$.
\end{theorem}

\begin{proof}
If $x^*$ and $x^{**}$ are two minimum points, then for 
$r=\max\{\|x^*\|,\|x^{**}\|\}$, we must have $\omega_r(\|x^*-x^{**}\|)=0$; 
hence $\argmin\Phi$, which is nonempty, reduces to one point, say, $x^*$.

By Theorem \ref{T:Phi_t-2}, the trajectory $x$ is bounded. Let $r>0$ be such that it is contained in the ball centered at the origin with radius $r$. Since $\nabla \Phi $ is uniformly monotone on bounded sets, and $\nabla \Phi (x^{*} )=0$, we have 
$$
\langle   \nabla \Phi (x(t)) , x(t)-x^{*}  \rangle \geq \omega_r (\|x(t)-x^{*} \|).
$$
Hence
$$
 \omega_r (\|x(t)-x^{*} \|) \leq 2r \| \nabla \Phi (x(t)) \| . 
$$
By Proposition \ref{P:speed_to_0}, 
$$
\lim_{t\to\infty}  \| \nabla \Phi (x(t))  \| = 0.
$$
Hence $\omega_r (\|x(t)-x^{*} \|) \to 0$, which implies $\|x(t)-x^{*} \| \rightarrow 0$.
\end{proof}

\section{Further results in the strongly convex case}\label{sc}
Let us recall that a function 
$\Phi : \mathcal H \rightarrow \mathbb R$ is  strongly convex if there exists 
some $\mu>0$  such that the function $\Phi-(\mu/2)\|\cdot\|^2$ is convex. If 
$\Phi$ is differentiable, the convexity inequality yields, for all 
$x, y \in \mathcal H$ 
$$
\Phi(y)-\frac{\mu}{2}\|y\|^2\geq
  \Phi(x)-\frac{\mu}{2}\|x\|^2+\langle\nabla\Phi(x)-\mu x,y-x\rangle.
$$  
Whence, for all $x,y \in \mathcal H$
\begin{equation}\label{sc1}
\Phi(y) \geq \Phi (x) + \langle  \nabla \Phi (x)  , y-x \rangle + \frac{\mu}{2} \|x-y  \|^2.
\end{equation}
The gradient of a strongly convex function is uniformly monotone on bounded sets, so that Theorem \ref{Thm-strong-convex} holds. However, in the case of a strongly convex function, we can obtain a rate of convergence for $\Phi(x(t))$ to the infimal value better than that of Theorem \ref{T:Phi_t-2} and more precise than that of Theorem \ref{T:superconv}.

\begin{theorem} \label{Thm-sc-basic}
Suppose $\alpha\geq3$ and that $\Phi:\Hb\to\R$ is strongly convex. Then, 
$\argmin\Phi$ is reduced to a singleton $x^*$, and for any trajectory $x$  of 
(DIN-AVD) the following properties hold: 
\begin{eqnarray*}
\Phi(x(t))-\min_{\mathcal H}\Phi
& = & 
\mathcal O\left(t^{-\frac{2\alpha}{3}}\right) \\
\|x(t)-x^*\|
& = & 
O\left(t^{-\frac{\alpha}{3}}\right).
\end{eqnarray*}
\end{theorem}
\begin{proof}
The first part follows from Theorem \ref{Thm-strong-convex}. For the convergence rates, the proof goes along the lines of that of Theorem \ref{T:Phi_t-2}: we shall
show that a surrogate Lyapunov function (namely, the function $\mathcal L$ defined below) is bounded. Let $p\in\R$, $\lambda>0$ and let $Q$ be a quadratic polynomial. Precise values, 
depending on $\alpha$ and $\beta$, will be given further to $\lambda$, $p$, $Q$. Let us briefly write $u$ for $u_\beta$. Set $P(t)=t^p$ and for $x^*\in\argmin\Phi$ define 
\begin{eqnarray*}
\mathcal L_0(t)
& = &
Q(t)(\Phi(x(t))-\Phi(x^*))
  +\frac{1}{2}\|\lambda(x(t)-x^*)+t\dot u(t)\|^2 \\
\mathcal L(t)
& = &
P(t)\mathcal L_0(t).
\end{eqnarray*}
Our first task is to differentiate function $\mathcal L$. To simplify the 
wording, we write $\Phi^*$ for $\Phi(x^*)$ and the dependence of $x$, $P$, 
$Q$, $\mathcal L_0$ and $\mathcal L$ on $t$ is not made explicit. To compute 
$d\mathcal L_0/dt$, we make use of \eqref{uniter} and \eqref{ubis}.
\begin{align*}
\frac{d\mathcal L_0}{dt}=\, &
  \dot Q(\Phi(x)-\Phi^*)+Q\langle\dot x,\nabla\Phi(x)\rangle
  +\langle \lambda(x-x^*)+t\dot u,\lambda\dot x+\dot u+t\ddot u\rangle
\\ 
=\, &
  \dot Q(\Phi(x)-\Phi^*)+Q\langle\dot x,\nabla\Phi(x)\rangle
  +\langle
    \lambda(x-x^*)+t\dot x+t\beta\nabla\Phi(x)),
    (\lambda+1-\alpha)\dot x+(\beta-t)\nabla\Phi(x)
  \rangle
\\ 
=\, &
  \dot Q(\Phi(x)-\Phi^*)
  +(Q-t^2+\beta(\lambda+2-\alpha)t)\langle\dot x,\nabla\Phi(x)\rangle
  +\lambda(\lambda+1-\alpha)\langle x-x^*,\dot x\rangle
  +\lambda(\beta-t)\langle x-x^*,\nabla\Phi(x)\rangle \\
\, &
  +(\lambda+\alpha-1)t\|\dot x\|^2
  +\beta t(\beta-t)\|\nabla\Phi(x)\|^2.
\end{align*} 
Expand $\mathcal L_0$ as
\begin{align*}
\mathcal L_0=\, & 
  Q(\Phi(x)-\Phi^*)+\frac{\lambda^2}{2}\|x-x^*\|^2
  +\frac{t^2}{2}\|\dot x\|^2+\frac{\beta^2t^2}{2}\|\nabla\Phi(x)\|^2
  +\lambda t\langle x-x^*,\dot x\rangle
  +\lambda\beta t\langle x-x^*,\nabla\Phi(x)\rangle
  +\beta t^2\langle\dot x,\nabla\Phi(x)\rangle, 
\end{align*} 
and compute 
\begin{align}
\frac{d\mathcal L}{dt}=\, &
  \frac{d\mathcal L_0}{dt}P+\mathcal L_0\dot P
\nonumber \\
=\, &
  \frac{d(PQ)}{dt}(\Phi(x)-\Phi^*)
  +[(Q-t^2+\beta(\lambda+2-\alpha)t)P+\beta t^2\dot P]
    \langle\dot x,\nabla\Phi(x)\rangle \label{dLdt} \\
\, &
  +[\lambda(\lambda+1-\alpha)P+\lambda t\dot P]\langle x-x^*,\dot x\rangle
  +\lambda[(\beta-t)P+\beta t\dot P]\langle x-x^*,\nabla\Phi(x)\rangle
  +[2(\lambda+1-\alpha)tP+t^2\dot P]\frac{1}{2}\|\dot x\|^2 \nonumber \\
\, &
  +[2\beta(\beta-t)tP+\beta^2t^2\dot P]\frac{1}{2}\|\nabla\Phi(x)\|^2
  +\frac{\lambda^2\dot P}{2}\|x-x^*\|^2. \nonumber
\end{align}
For $t$ large enough (namely $t>(p+1)\beta$) the coefficient  
$\lambda[(\beta-t)P+\beta t\dot P]$ of $\langle x-x^*,\nabla\Phi(x)\rangle$ is 
negative. Applying the strong convexity inequality \eqref{sc1} (with $y=x^*$), 
we have
\begin{equation*}
\lambda[(\beta-t)P+\beta t\dot P]\langle x-x^*,\nabla\Phi(x)\rangle\leq
\lambda[(\beta-t)P+\beta t\dot P](\Phi(x)-\Phi^*)
  +\lambda[(\beta-t)P+\beta t\dot P]\frac{\mu}{2}\|x-x^*\|^2.
\end{equation*} 
So, we can dispose of $\langle x-x^*,\nabla\Phi(x)\rangle$ in \eqref{dLdt}, and 
obtain the inequality
\begin{align}
\frac{d\mathcal L}{dt}\leq\, &
  \left[\frac{d(PQ)}{dt}+\lambda((\beta-t)P+\beta t\dot P)\right]
    (\Phi(x)-\Phi^*)
  +[(Q-t^2+\beta(\lambda+2-\alpha)t)P+\beta t^2\dot P]
    \langle\dot x,\nabla\Phi(x)\rangle \label{dLdtbis} \\
\, &
  +[\lambda(\lambda+1-\alpha)P+\lambda t\dot P]\langle x-x^*,\dot x\rangle
  +[2(\lambda+1-\alpha)tP+t^2\dot P]\frac{1}{2}\|\dot x\|^2 \nonumber \\
\, &
  +[2\beta(\beta-t)tP+\beta^2t^2\dot P]\frac{1}{2}\|\nabla\Phi(x)\|^2
  +[\lambda^2\dot P+\mu\lambda((\beta-t)P+\beta t\dot P)]
    \frac{1}{2}\|x-x^*\|^2. 
\nonumber
\end{align}
If we choose $p=2(\alpha-1-\lambda)$ and 
$Q(t)=t^2-\beta(\lambda+2-\alpha+p)t=t^2-\beta(\alpha-\lambda)t$, the 
coefficients of $\|\dot x\|$ and $\langle\dot x,\nabla\Phi(x)\rangle$ vanish 
(recall $P(t)=t^p$). Taking these facts into account, we deduce from 
\eqref{dLdtbis} that
\begin{align*}
\frac{d\mathcal L}{dt}\leq\, &
  -t^p[(\lambda-2-p)t+\beta(\alpha-2\lambda)(p+1)](\Phi(x)-\Phi^*)
  +\lambda(\lambda+1-\alpha+p)t^p\langle x-x^*,\dot x\rangle \\
\, &
  -\beta t^{p+1}[2t-\beta(p+2)]\frac{1}{2}\|\nabla\Phi(x)\|^2 
  -\lambda t^{p-1}[\mu t^2-\mu\beta(p+1)t-p\lambda]\frac{1}{2}\|x-x^*\|^2. 
\end{align*}
For $t$ sufficiently large, the coefficients of $\|\nabla\Phi(x)\|^2$ and 
$\|x-x^*\|^2$ are negative; hence 
\begin{align}\label{dLdtter}
\frac{d\mathcal L}{dt}\leq\, &
  -t^p[(\lambda-2-p)t+\beta(\alpha-2\lambda)(p+1)](\Phi(x)-\Phi^*)
  +\lambda(\lambda+1-\alpha+p)t^p\langle x-x^*,\dot x\rangle. 
\end{align}
Choose $\lambda=\frac{2}{3}\alpha$, whence $p=\frac{2}{3}\alpha-2$, 
$\lambda-2-p=0$. Moreover, define $h(t)=\frac{1}{2}\|x(t)-x^*\|^2$. Then 
inequality \eqref{dLdtter} becomes
\begin{align}\label{dLdtquater}
\frac{d\mathcal L}{dt}\leq\, &
  \frac{\alpha\beta}{9}(2\alpha-3)t^p(\Phi(x)-\Phi^*)
  +\frac{2\alpha}{9}(\alpha-3)t^p\dot h.
\end{align}
To simplify the notations, set $\gamma=\frac{\alpha\beta}{9}(2\alpha-3)$ and 
$\delta=\frac{\alpha\beta}{3}$, and notice that $Q(t)=t(t-\delta)$ and 
$\frac{2\alpha}{9}(\alpha-3)=\frac{p\alpha}{3}$. Inequality \eqref{dLdtquater} 
reads 
\begin{align*}
\frac{d\mathcal L}{dt}\leq\, &
  \gamma t^p(\Phi(x)-\Phi^*)+\frac{p\alpha}{3}t^p\dot h.
\end{align*}
With $t^pQ(\Phi(x)-\Phi^*)\leq\mathcal L$ (for $t>\delta$), the inequality 
above becomes 
\begin{align*}
\frac{d\mathcal L}{dt}\leq\, &
  \gamma\frac{\mathcal L}{Q}+\frac{p\alpha}{3}t^p\dot h.
\end{align*}
If we multiply this inequality by $(\frac{t}{t-\delta})^{p+1}$ we obtain
\begin{equation*}
\left(\frac{t}{t-\delta}\right)^{p+1}
  \left[\frac{d\mathcal L}{dt}-\gamma\frac{\mathcal L}{Q}\right]
=\frac{d}{dt}
  \left[\left(\frac{t}{t-\delta}\right)^{p+1}\mathcal L\right]
\leq\frac{p\alpha}{3}t^p\left(\frac{t}{t-\delta}\right)^{p+1}\dot h.
\end{equation*}
Integrating between $t_1$, sufficiently large, and $t\geq t_1$, we obtain
\begin{eqnarray*}
\lefteqn{\left(\frac{t}{t-\delta}\right)^{p+1}\mathcal L(t)) \leq } \\
& & 
\left(\frac{t_1}{t_1-\delta}\right)^{p+1}\mathcal L(t_1)
  +\frac{p\alpha}{3}
    \left[
    t^p\left(\frac{t}{t-\delta}\right)^{p+1}h(t)
    -t_1^p\left(\frac{t_1}{t_1-\delta}\right)^{p+1}h(t_1)
    -\int_{t_1}^t\frac{d}{ds}
      \left[s^p\left(\frac{s}{s-\delta}\right)^{p+1}\right]h(s)ds
    \right].
\end{eqnarray*}
But 
$\frac{d}{ds}\ln(s^p(\frac{s}{s-\delta})^{p+1})=
  \frac{ps-\gamma}{s(s-\delta)}$, 
which shows that the integrand is positive for $t$ large enough. Hence 
$$
\left(\frac{t}{t-\delta}\right)^{p+1}\mathcal L(t))
\leq  
\left(\frac{t_1}{t_1-\delta}\right)^{p+1}\mathcal L(t_1)
  +\frac{p\alpha}{3}
    t^p\left(\frac{t}{t-\delta}\right)^{p+1}h(t).
$$
Whence, we deduce that
\begin{equation}\label{inegL}
\mathcal L(t))
\leq  
\left(\frac{t-\delta}{t}\right)^{p+1}
\left(\frac{t_1}{t_1-\delta}\right)^p\mathcal L(t_1)
  +\frac{p\alpha}{3} t^ph(t)
\leq
\left(\frac{t_1}{t_1-\delta}\right)^p\mathcal L(t_1)
  +\frac{p\alpha}{3} t^ph(t).
\end{equation}
Now, in view of the strong convexity inequality \eqref{sc1} (with $x=x^*$ and 
$y=x(t)$) we have 
\begin{equation}\label{fc}
h(t)\leq\frac{1}{\mu}(\Phi(x(t))-\Phi^*)
  \leq\frac{\mathcal L(t)}{\mu t^pQ(t)}. 
\end{equation}
Hence, inequality \eqref{inegL} yields
$$
\mathcal L(t))
\leq
\left(\frac{t_1}{t_1-\delta}\right)^p\mathcal L(t_1)
  +\frac{p\alpha}{3\mu}\frac{\mathcal L(t)}{Q(t)}.
$$
We deduce that
$$
\mathcal L(t))
\leq 
\left(\frac{t_1}{t_1-\delta}\right)^{p+1}\mathcal L(t_1)
  \frac{Q(t)}{Q(t)-\frac{p\alpha}{3\mu}},
$$
which shows that $\mathcal L$ is bounded. 

If $L$ denotes an upper bound of $\mathcal L$, we have 
$$
\Phi(x(t))-\Phi^*\leq\frac{L}{t^pQ(t)}
  =\mathcal O\left(\frac{1}{t^{p+2}}\right)
  =\mathcal O\left(t^{-\frac{2\alpha}{3}}\right).
$$
In view of the first inequality in \eqref{fc}, we also have
$$
\|x(t)-x^*\|=\mathcal O\left(t^{-\frac{\alpha}{3}}\right),
$$
as claimed.
\end{proof}

\section{(DIN-AVD) as a first-order system. Extension to non-smooth potentials}\label{section-existence}

Let $\beta>0$. As we shall see, the presence of the Hessian damping term allows formulating (DIN-AVD) as a first-order system both in time and space (with no occurrence of the Hessian). This will allow us to extend our study to the case of a proper lower-semicontinuous convex function, by simply replacing the gradient by the subdifferential. This approach was initiated in \cite{aabr} in the case of (DIN), and further exploited for the study of damped shocks in mechanics in \cite{AMR}.

We begin by establishing the equivalence between (DIN-AVD) and a first-order system in the smooth case in Subsection \ref{section-first-order}, and then recover most results from preceding sections in the nonsmooth setting. Some of the arguments are essentially the same, so we will study in more detail the parts that are not, and leave the rest to the reader. To simplify the reading, we shall use $\Phi$ to denote a smooth potential (as in the previous sections), and $\phi$ for a proper lower-semicontinuous convex function.

\subsection{(DIN-AVD) as a first-order system}\label{section-first-order}

\begin{theorem} \label{Thm-first-order system}
Let $\Phi:\mathcal H\rightarrow\mathbb R$ be twice continuously differentiable. Suppose $\alpha\geq0$, $\beta>0$. Let $(x_0,\dot x_0)\in\Hb\times\Hb$. The 
following statements are equivalent:
\begin{enumerate}
\item
$x:[t_0,+\infty [ \to \mathcal H$ is a solution to the second-order differential equation
\begin{equation*}
{\hbox{\rm (DIN-AVD)}}  \quad \quad \ddot{x}(t) + \frac{\alpha}{t} \dot{x}(t) + \beta \nabla^2 \Phi (x(t))\dot{x} (t) + \nabla \Phi (x(t)) = 0,
\end{equation*}
with initial conditions $x(t_0)=x_0$, $\dot x(t_0)=\dot x_0$. 
\item
$(x,y):[t_0,+\infty [ \to \mathcal H \times \mathcal H$ is a solution to the first-order  system 
$$
{\hbox{\rm (DIN-AVD)}_f} \ \left\{
\begin{array}{rcl}
\dot x(t) +  \beta \nabla \Phi (x(t))- \left( \frac{1}{\beta} -  \frac{\alpha}{t} \right)  x(t)+\frac{1}{\beta}y(t) & = & 0 \\
 \dot{y}(t)-\left( \frac{1}{\beta} -  \frac{\alpha}{t} +  \frac{\alpha \beta}{t^2}  \right) x(t)
 +\frac{1}{\beta} y(t) & = & 0,
 \end{array}\right.
$$
with initial conditions $x(t_0)=x_0$, 
$y(t_0)=-(\dot x_0+\beta\nabla\Phi(x_0))/\beta+(1-\beta\alpha/t_0)x_0$.
\end{enumerate}
\end{theorem}

\begin{proof}
To simplify the notation, set $a(t)=\alpha/t$ and $z(t)=\int_{t_0}^t(a(s)\dot x(s)+\nabla\Phi(x(s)))ds-(\dot x_0+\beta\nabla\Phi(x_0))$.

Integrating (DIN-AVD) from $t_0$ to $t\geq t_0$ and differentiating $z$, we 
see that $x$ is a solution to (DIN-AVD) with initial conditions $x(t_0)=x_0$, 
$\dot x(t_0)=\dot x_0$, if and only if, $(x,z)$ is a solution to 
\begin{equation}\label{eq:spo}
\left\{
\begin{array}{rcl}
\dot x(t) +  \beta \nabla \Phi (x(t))+z(t) & = & 0 \\
\dot z(t)-a(t)\dot x(t)-\nabla\Phi(x(t)) & = & 	0,
 \end{array}\right.
\end{equation}
with initial conditions $x(t_0)=x_0$, 
$z(t_0)=-(\dot x_0+\beta\nabla\Phi(x_0))$.

Use a linear combination of the rows in \eqref{eq:spo} to eliminate the 
gradient in the second equation, and obtain the equivalent system 
\begin{equation}\label{eq:spo2}
\left\{
\begin{array}{rcl}
\dot x(t) +  \beta \nabla \Phi (x(t))+z(t) & = & 0 \\
\beta\dot z(t)+(1-\beta a(t))\dot x(t)+z(t) & = & 0.
 \end{array}\right.
\end{equation}
Now define $y(t)=\beta z(t)+(1-\beta a(t))x(t)$. We see that $(x,z)$ is a 
solution to \eqref{eq:spo2}, with initial conditions $x(t_0)=x_0$, 
$z(t_0)=-(\dot x_0+\beta\nabla\Phi(x_0))$, if and only if $(x,y)$ is a 
solution to 
$$
\left\{
\begin{array}{rcl}
\dot x(t)+\beta\nabla\Phi(x(t))-\left(\frac{1}{\beta}-a(t)\right)x(t)
  +\frac{1}{\beta}y(t) & = & 0 \\
\dot y(t)+\beta\dot a(t)x(t)-\left(\frac{1}{\beta}-a(t)\right)x(t)
  +\frac{1}{\beta}y(t) & = & 0,
 \end{array}\right.
$$
with initial conditions $x(t_0)=x_0$, 
$y(t_0)=-\beta(\dot x_0+\beta\nabla\Phi(x_0))+(1-\beta\ a(t_0))x_0$.
\end{proof}

\subsection{Existence of solutions in a nonsmooth setting}\label{section-gen- existence}

Beyond being of first-order in time, $\hbox{\rm(DIN-AVD)}_f$ does not involve the Hessian of $\Phi$. As a first consequence, the numerical solution of (DIN-AVD) is highly simplified, since it may be performed by discretization of $\hbox{\rm(DIN-AVD)}_f$ and only requires approximating the gradient of $\Phi$. Next, $\hbox{\rm(DIN-AVD)}_f$ permits to give a meaning to (DIN-AVD) even when $\Phi$ is not twice differentiable. In particular, we may consider a proper lower-semicontinuous convex potential function $\phi$. 

More precisely, we have the following:

\begin{definition}
Let $\alpha\geq0$, $\beta>0$ and $\phi:\Hb\to\R\cup\{+\infty\}$ be a proper lower-semicontinuous convex function. The generalized (DIN-AVD) system, 
(g-DIN-AVD) for short, is defined by
\begin{equation}\label{eq:fos2}
\hbox{\rm(g-DIN-AVD)}\ \left\{
\begin{array}{rcl}
\dot x(t)+\beta\partial\phi(x(t))
  -\left(\frac{1}{\beta}-\frac{\alpha}{t}\right)x(t)+\frac{1}{\beta}y(t) & \ni & 0 \\
\dot{y}(t)
  -\left(\frac{1}{\beta}-\frac{\alpha}{t}+\frac{\alpha\beta}{t^2}\right)x(t)
   +\frac{1}{\beta} y(t) & = & 0,
\end{array}\right.
\end{equation}
where $\partial \phi$ stands for the convex subdifferential of $\phi$.
\end{definition}

Setting $Z(t) = (x(t), y(t)) \in  \mathcal H \times \mathcal H$, (g-DIN-AVD) 
can be equivalently written 
\begin{equation}\label{eq:fos3}
\dot{Z}(t) + \partial \mathcal G (Z(t)) + D(t, Z(t)) \ni 0,
\end{equation}
where $\mathcal G: \mathcal H \times \mathcal H \rightarrow \mathbb R \cup \lbrace + \infty \rbrace$ is the convex function defined by
\begin{equation}\label{eq:fos4}
 \mathcal G (Z) = \mathcal G(x,y)= \beta \phi (x),
\end{equation}
and $D:\ [t_0,+\infty[ \times \mathcal H \times \mathcal H \to \mathcal H \times \mathcal H$ is given by
\begin{equation}\label{eq:fos5}
D(t,Z)=D(t,x,y)=
  \left(-\left(\frac{1}{\beta}-\frac{\alpha}{t}\right)x+\frac{1}{\beta}y,
    -\left(\frac{1}{\beta}-\frac{\alpha}{t}+\frac{\alpha\beta}{t^2}\right)x
 +\frac{1}{\beta} y \right) .
\end{equation}

The differential inclusion \eqref{eq:fos3} is governed by the sum of the 
maximal monotone operator $\partial \mathcal G$ (a convex subdifferential) and the time-dependent linear continuous operator $D (t,\cdot)$.
The existence and uniqueness of a global solution for the corresponding Cauchy problem is a consequence of the general theory of evolution equations governed by maximal monotone operators. Before giving a precise statement, let us recall the notion of strong solution 
(see \cite[Definition 3.1]{Bre1}).

\begin{definition}\label{defsolforte}
Let $\Hb$ be a Hilbert space, $t_0\in\R$ and $T>t_0$. Consider a proper lower-semicontinuous convex function $\phi:\Hb\to\R\cup\{+\infty\}$, along with a function $B:[t_0,+\infty[\times\Hb\to\Hb$. We say that $z:[t_0,T]\to\Hb$ 
is a strong solution on $[t_0,T]$ to the differential inclusion
\begin{equation}\label{eq:fosdef}
\dot z(t)+\partial\phi(z(t))+B(t,z(t))\ni0,
\end{equation}
if the following properties are satisfied:
\begin{enumerate}
\item \label{defsolforte-1}
$z\in\mathcal C([t_0,T],\Hb)$;
\item \label{defsolforte-2}
$z$ is absolutely continuous on any compact subset of $]t_0,T]$;
\item\label{defsolforte-3} 
$z(t)\in\mbox{dom}(\partial\phi)$ for almost every $t\in]t_0,T]$;
\item 
the inclusion \eqref{eq:fosdef} is verified for almost every $t\in]t_0,T]$.
\end{enumerate}
We say that $z:[t_0,+\infty[\to\Hb$ is a global strong solution to 
\eqref{eq:fosdef}, if it is a strong solution to (\ref{eq:fosdef}) on 
$[t_0,T]$ for all $T>t_0$. 
\end{definition}
With this terminology, we have the following:

\begin{theorem} \label{Thm-existence}
Let $\phi : \mathcal H \rightarrow \mathbb R \cup \lbrace + \infty \rbrace$ be a convex lower semicontinuous proper function, and let $\beta >0$.
 For any Cauchy data $(x_0, y_0) \in  \mbox{\rm{dom}}\,\phi \times \mathcal H $, there exists a unique global strong solution
 $(x,y):[t_0, +\infty[ \rightarrow \mathcal H \times \mathcal H $ to 
(g-DIN-AVD) verifying the initial condition $x(t_0)=x_0$, $y(t_0)= y_0$. This 
solution enjoys the further properties
\begin{itemize}
\item [(i)]\label{existence-1}
$y$ is continuously differentiable on $[t_0,+\infty[$, and 
$ 
\dot{y}(t)
  -\left(\frac{1}{\beta}-\frac{\alpha}{t}+\frac{\alpha\beta}{t^2}\right)x(t)
   +\frac{1}{\beta} y(t) =0,
$ 
for all $t\geq t_0$;
\item [(ii)]\label{existence-2}
$x$ is absolutely continuous on $[t_0,T]$ and $\dot x\in L^2(t_0,T;\Hb)$ for 
all $T>t_0$;
\item [(iii)]\label{existence-3}
$x(t)\in\mbox{dom}(\partial\phi)$ for all $t>t_0$;
\item [(iv)]\label{existence-4}
$x$ is Lipschitz continuous on any compact subinterval of $]t_0,+\infty[$;
\item [(v)]
the function $[t_0,+\infty[\ni t\mapsto\phi(x(t))$ is absolutely continuous on 
$[t_0,T]$ for all $T>t_0$;
\item [(vi)]\label{existence-6}
there exists a function $\xi:[t_0,+\infty[\to\Hb$ such that 
\begin{itemize}
\item [(a)]
  $\xi(t)\in\partial\phi(x(t))$ for all $t>t_0$;
\item [(b)]
  $\dot x(t)+\beta\xi(t)
  -\left(\frac{1}{\beta}-\frac{\alpha}{t}\right)x(t)+\frac{1}{\beta}y(t)=0$  
  for almost every $t>t_0$;
\item [(c)]
  $\xi\in L^2(t_0,T;\Hb)$ for all $T>t_0$;
\item [(d)]
  $\frac{d}{dt}\phi(x(t))=\langle\xi(t),\dot x(t)\rangle$ for almost every 
  $t>t_0$.
\end{itemize}
\end{itemize}
\end{theorem}

\begin{proof} It is sufficient to prove that $(x,y)$ is a strong solution of 
(g-DIN-AVD) on $[t_0,T]$ and that the properties hold on $[t_0,T]$ for each arbitrary 
$T>t_0$. So let us fix $T>t_0$. As we have already noticed, (g-DIN-AVD) can be 
written as a Lipschitz perturbation \eqref{eq:fos3} of the differential inclusion governed 
by the  subdifferential of a proper lower-semicontinuous convex function. A direct 
application of \cite[Proposition 3.12]{Bre1} (see also \cite[Theorem 4.1]{AMR}) gives the existence and uniqueness of a strong 
global solution $Z=(x,y):[t_0,T]\to\Hb\times\Hb$ to \eqref{eq:fos3}, equivalent to (g-DIN-AVD), with 
initial condition $Z(t_0)=(x(t_0),y(t_0))=(x_0,y_0)$. Due to the simple form of perturbation $D$, the solution $(x,y)$ enjoys further properties:

\noindent (i) For almost every $t\geq t_0$ we have $\dot y(t)=g(t)$, where 
$
g(t)=
  (\frac{1}{\beta}-\frac{\alpha}{t}+\frac{\alpha\beta}{t^2})x(t)
   -\frac{1}{\beta} y(t)
$
is continuous on $[t_0,T]$. Since $y$ is absolutely continous on 
$[s,t]\subseteq]t_0,T]$, we have $y(t)-y(s)=\int_s^t g(\tau)d\tau$, and by 
continuity of $y$: $y(t)-y(t_0)=\int_{t_0}^t g(\tau)d\tau$. Whence 
$\dot y(t)=g(t)$ for all $t\geq t_0$ (with $\dot y(t_0)$ the right 
derivative).

To prove the next items, we introduce the differential inclusion 
\begin{equation}\label{inclaux}
 \dot z(t) +  \beta \partial \phi (z(t))\ni f(t),
\end{equation}
to be satisfied by the unknown function $z$, where 
$f(t)=\left(\frac{1}{\beta}-\frac{\alpha}{t}\right)x(t)-\frac{1}{\beta}y(t)$. 
The function $f$ is continuous on $[t_0,T]$ and absolutely continuous on any 
compact subinterval of $]t_0,T]$ (properties 
\eqref{defsolforte-1} and \eqref{defsolforte-2} in Definition \ref{defsolforte}).

\noindent (ii) Consider the inclusion \eqref{inclaux} on $[t_0,T]$ with the initial condition 
$z(t_0)=x_0\in\mbox{dom}\phi\subseteq\overline{\mbox{dom}\partial\phi}$ 
(see \cite[Proposition 2.11]{Bre1} for the set inclusion). The assumptions of 
\cite[Theorem 3.4]{Bre1} are met: $x_0\in\overline{\mbox{dom}\partial\phi}$, 
$f\in L^1(t_0,T;\Hb)$; hence inclusion \eqref{inclaux} has a unique strong 
solution $z$ which obviously coincides with $x$ on $[t_0,T]$. Then 
\cite[Theorem 3.6]{Bre1} states that $\dot z=\dot x$ belongs to 
$L^2(t_0,T;\Hb)$, since the assumptions $f\in L^2(t_0,T;\Hb)$ and 
$x_0\in\mbox{dom}\phi$ are fulfilled. Now, in view of the absolute continuity 
of $x$, for $[s,t]\subset]t_0,T]$, we have 
$x(t)-x(s)=\int_s^t\dot x(\tau)d\tau$, and by continuity of $x$ 
$x(t)-x(t_0)=\int_{t_0}^t\dot x(\tau)d\tau$. Hence $x$ is absolutely 
continuous on $[t_0,T]$ since 
$\dot x\in L^2(t_0,T;\Hb)\subseteq L^1(t_0,T;\Hb)$. 

\noindent (iii) As before, let $z$ be the solution to \eqref{inclaux} on $[t_0,T]$ with 
initial condition $z(t_0)=x_0$. Then, with \cite[Theorem 3.7)]{Bre1}, 
$z(t)=x(t)$ lies in $\mbox{dom}\partial\phi$ for all $t\in]t_0,T]$ because $f$ 
has inherited the absolute continuity of $x$ on $[t_0,T]$ and 
$x_0\in\mbox{dom}\phi$. 

\noindent (iv) For any $\tau\in]t_0,T]$ consider the inclusion \eqref{inclaux} on $[\tau,T]$ 
with initial condition $z(\tau)=x(\tau)\in\mbox{dom}\partial\phi$. Obviously 
$z$ coincides with $x$ on 
$[\tau,T]$. Then \cite[Proposition 3.3 (or Theorem 3.17)]{Bre1} states that 
$z=x$ is Lipschitz continuous on $[\tau,T]$, because $f$ is of bounded 
variation on $[\tau,T]$ and $x(\tau)\in\mbox{dom}\partial\phi$. As a 
consequence, $x$ is Lipschitz continuous on any compact subinterval of 
$]t_0,T]$.  This also gives (v). 

\noindent (vi) Assertions (a)(b) are consequences of $(x,y)$ being a global strong solution 
of (g-DIN-AVD) and of \eqref{existence-3}, while (c) is a consequence of (b) 
and \eqref{existence-2}. Now, the hypotheses of \cite[Lemma 3.3]{Bre1} are met on $[t_0,T]$ 
({\it i. \!\!e.} $x$ absolutely continuous on $[t_0,T]$ with $\dot x$ and 
$\xi$ in $L^2(t_0,T;\Hb)$) and we can conclude that the function 
$t\in[t_0,T]\to\phi(x(t))$ is absolutely continuous and that (d) holds almost 
everywhere on $[t_0,T]$ hence on $[t_0,+\infty[$.
\end{proof}

\begin{remark}  As a remarkable property of the semi-group of contractions generated by the subdifferential of a convex lower semicontinuous proper function, there is a regularization effect on the initial data. This property has been extended to the case of a Lipschitz perturbation of a convex subdifferential in \cite[Proposition 3.12]{Bre1}. As a consequence, the existence and uniqueness of a 
strong solution to (g-DIN-AVD) with Cauchy data 
$(x_0, y_0) \in  \overline{\rm{dom} \phi} \times \mathcal H $ is still valid, 
but some properties stated in Theorem \ref{Thm-existence} have to be weakened.
\end{remark}

As a direct consequence of Theorem \ref{Thm-existence}, we obtain the existence and uniqueness result for (DIN-AVD):

\begin{corollary}\label{Cor-existence}
Suppose that $\Phi: \mathcal H \rightarrow \mathbb R$ is a convex
 $\mathcal C ^2$ function. For any $t_0 > 0$, and  any Cauchy data
$(x_0, \dot{x}_0) \in  \mathcal H \times \mathcal H $,
there exists a unique classical global solution
 $x: \ [t_0, +\infty[ \rightarrow \mathcal H $ to 
$$
\hbox{\rm(DIN-AVD)} \quad\quad\quad 
\ddot{x}(t) + \frac{\alpha}{t} \dot{x}(t) + \beta \nabla^2 \Phi (x(t))\dot{x} (t) + \nabla \Phi (x(t)) = 0,  	   
$$
with $x(t_0)=x_0$, $\dot{x}(t_0)= \dot{x}_0$ .
\end{corollary}

\begin{proof}
First, use the equivalence between (DIN-AVD) and the first-order system (g-DIN-AVD), as given by  Theorem \ref{Thm-first-order system}, and then apply Theorem \ref{Thm-existence} with 
$y(t_0)=-(\dot x_0+\beta\nabla\Phi(x_0))/\beta+(1-\beta\alpha/t_0)x_0$.
\end{proof}

\begin{remark}
It may be useful to sum up, for future reference, in addition to the 
regularity properties, the equalities satisfied by a global strong solution 
$(x,y)$ to (g-DIN-AVD)
\begin{eqnarray}
& & 
\dot x(t)+\beta\xi(t)-\left(\frac{1}{\beta}-\frac{\alpha}{t}\right)x(t)
  +\frac{1}{\beta}y(t)=0 
\label{gDINAVD1} \\
& & 
\dot y(t)
  -\left(\frac{1}{\beta}-\frac{\alpha}{t}+\frac{\alpha\beta}{t^2}\right)x(t)
  +\frac{1}{\beta}y(t)=0 
\label{gDINAVD2} \\
& & 
\frac{d}{dt}\phi(x(t))=\langle\xi(t),\dot x(t)\rangle.
\label{gDINAVD3} 
\end{eqnarray}
Recall that \eqref{gDINAVD1}, \eqref{gDINAVD3} are true for almost every 
$t>t_0$ ,while \eqref{gDINAVD2} is true for all $t\geq t_0$.
\end{remark}
The following sections are devoted to showing that most properties of the 
classical solution $x$ of (DIN-AVD) hold for the global strong solution of 
(g-DIN-AVD) (actually, those that do not require $x$ to be twice 
differentiable).

\subsection{Generalized (DIN-AVD): minimizing properties}
Let $(x,y):[t_0,+\infty[\to\Hb\times\Hb$ be the global strong solution to (g-DIN-AVD) with Cauchy data $(x(t_0),y(t_0))=(x_0,y_0)\in\mbox{dom}\,\phi\times\Hb$. Let us show that the results in Subsection \ref{SS:Lyapunov} remain valid.

For $t\geq t_0$ define
\begin{equation} \label{u0}
u(t)=\int_{t_0}^t
  \left(
    \left(\frac{1}{\beta}-\frac{\alpha}{s}\right)x(s)
  -\frac{1}{\beta}y(s)\right)ds.
\end{equation}
With \eqref{gDINAVD1}, $u$ also satisfies
$$
u(t)=x(t)-x_0+\int_{t_0}^t\beta\xi(s)ds.
$$
By its definition, $u$ is continuously differentiable, with $\dot u$ satisfying
\begin{eqnarray}
\dot u(t)
& = & 
\left(\frac{1}{\beta}-\frac{\alpha}{t}\right)x(t)-\frac{1}{\beta}y(t), 
  \quad\forall t\geq t_0,
\label{guniter} \\
& = &
\dot x(t)+\beta\xi(t),\mbox{\hspace{6em}for almost all }t>t_0.
\label{gubis}
\end{eqnarray}
With parts (i) and (ii) of Theorem \ref{Thm-existence}, equality \eqref{guniter} shows that {\it$\dot u$ is absolutely continuous on 
any compact subinterval of $[t_0,+\infty[$}, hence differentiable almost everywhere on $[t_0,+\infty[$. Therefore, 
$$
\ddot u(t)=
  \frac{\alpha}{t^2}x(t)
  +\left(\frac{1}{\beta}-\frac{\alpha}{t}\right)\dot x(t)
  -\frac{1}{\beta}\dot y(t).
$$
The equality above, combined with 
$\dot y(t)=\alpha\beta x(t)/t^2-\dot x(t)-\beta\xi(t)$ (an easy consequence of 
\eqref{gDINAVD1} and \eqref{gDINAVD2}), yields 
\begin{equation} \label{dguniter0}
\ddot u(t)=-\frac{\alpha}{t}\dot x(t)-\xi(t),
\end{equation}
for almost all $t>t_0$. Using \eqref{gubis}, we also obtain
\begin{eqnarray} 
\ddot u(t)
& = & 
  \left(\frac{1}{\beta}-\frac{\alpha}{t}\right)\dot x(t)
  -\frac{1}{\beta}\dot u(t),
\label{dguniter} \\
\ddot u(t)+\frac{\alpha}{t}\dot u(t)
& = &
-\left(1-\frac{\alpha\beta}{t}\right)\xi(t),
\label{guter}
\end{eqnarray}
for almost all $t>t_0$. We will need the following energy function of the system, defined for all 
$t\geq t_0$ (recall \eqref{guniter}):
\begin{equation}\label{gE:W}
W(t)=\frac{1}{2}\|\dot u(t)\|^2+\phi (x(t)).
\end{equation}

We are now in a position to prove

\begin{theorem} \label{Thm-g-weak-conv2}
Let $\alpha>0$, and suppose $(x,y):[t_0,+\infty[\rightarrow\Hb\times\Hb$ is a 
global strong solution to (g-DIN-AVD). Then
\begin{itemize}
\item [(i)] \label{gweak-conv2-1}
$W$ is nonincreasing.
\item [(ii)] \label{gweak-conv2-2} 
$
\lim_{t\to+\infty}W(t)=\lim_{t\to+\infty}\phi(x(t))
  =\inf\phi\in\R\cup\{-\infty\}$.
\item [(iii)] \label{gweak-conv2-3}
As $t\to+\infty$, every sequential weak cluster point of $x(t)$ lies in $\argmin\phi$.
\item [(iv)] \label{gweak-conv2-4}
If $\|x(t)\|\not\to+\infty$ as $t\to+\infty$, then $\argmin\phi\neq\emptyset$. 
\item [(v)] \label{gweak-conv2-5} 
If $\phi$ is bounded from below, then 
  $\int_{t_0}^\infty\frac{1}{t}\|\dot x(t)\|^2dt<+\infty$, 
  $\int_{t_0}^\infty\frac{1}{t}\|\dot u(t)\|^2dt<+\infty$ and 
  $\lim_{t\to+\infty}\|\dot u(t)\|=0$. 
\item [(vi)] \label{gweak-conv2-6} 
If $\argmin\phi\neq\emptyset$ then 
  \begin{itemize}
  \item [(a)] \label{gweak-conv2-6a}
  $\phi(x(t))-\min\phi=\mathcal O\left(\frac{1}{\ln t}\right)$ and 
  $\|\dot u(t)\|=\mathcal O\left(\frac{1}{\sqrt{\ln t}}\right)$,
  \item [(b)] 
  $\int_{t_0}^\infty\frac{1}{t}(\phi(x(t))-\min\phi)dt\leq
    \int_{t_0}^\infty\frac{1}{t}(W(t)-\min\phi)dt<+\infty$.
  \end{itemize}
\end{itemize}
\end{theorem}

\begin{proof}
Once parts (i) and (ii) are proved, the rest of the arguments in Subsection \ref{SS:Lyapunov} can be applied for the remainder. The proof of parts (i) and (ii) is formally the same as in the smooth case, but we must be careful of equalities and inequalities that are true almost everywhere.

Since we are interested in asymptotic properties of $x$, we can assume $t\geq t_1=\max\{t_0,2\alpha\beta\}$ throughout the proof.

\noindent (i) With Theorem \ref{Thm-existence}, the energy $W$ is absolutely continuous on 
the compact subintervals of $[t_0,+\infty[$. Use \eqref{gDINAVD3} to obtain  
$$
\dot W(t)=\langle\dot u(t),\ddot u(t)\rangle+\langle\xi(t),\dot x(t)\rangle,
$$
for almost every $t>t_0$. Now use \eqref{gubis} and \eqref{dguniter} to obtain
\begin{eqnarray*}
\dot W(t)
& = & 
-\frac{1}{\beta}\|\dot x(t)\|^2-\frac{1}{\beta}\|\dot u(t)\|^2
  +\left(\frac{2}{\beta}-\frac{\alpha}{t}\right)
    \langle\dot u(t),\dot x(t)\rangle
\nonumber \\ 
& \leq &
  -\frac{\alpha}{2t}\|\dot x(t)\|^2-\frac{\alpha}{2t}\|\dot u(t)\|^2,
\end{eqnarray*}
for almost every 
$t>t_0$. Hence $W$ is nonincreasing, since it is absolutely continuous on the compact 
subintervals of $[t_0,+\infty[$. 

\noindent (ii) Given $z\in\mathcal H$, we define $h:[t_0,+\infty [ \to\R$ by
$$
h(t)=\frac{1}{2}\|u(t)-z\|^2.
$$
Function $h$ is continuously differentiable with
$$
\dot h(t) = \langle u(t) - z , \dot{u}(t)  \rangle,
$$
and the function $\dot h$ is absolutely continuous on compact subintervals of $[t_0,+\infty[$ (since $\dot u$ is) and satisfies 
$$
\ddot h(t) = \langle u(t) - z , \ddot{u}(t)  \rangle + \| \dot{u}(t) \|^2
$$
for almost every $t>t_0$. Using \eqref{guter}, we obtain
\begin{eqnarray*} 
\ddot h(t) + \frac{\alpha}{t} \dot h(t)  
& = & 
  \|\dot u(t)\|^2
  -\left(1-\frac{\alpha\beta}{t}\right)\langle u(t)-z,\xi(t)\rangle
\\
& = & 
  \|\dot u(t)\|^2
  -\left(1-\frac{\alpha\beta}{t}\right)\langle x(t)-z,\xi(t)\rangle
  -\beta\left(1-\frac{\alpha\beta}{t}\right)
    \left\langle-x_0+\int_{t_0}^t\xi(s)ds\,,\,\xi(t)\right\rangle,
\end{eqnarray*}
for almost every $t>t_0$. If we set $I(t)=\frac{1}{2}\left\|-x_0+\int_{t_0}^t\xi(s)ds\right\|^2$, 
then $I$ is absolutely continuous on $[t_0,T]$ for all $T>t_0$ and we can write 
$\dot I(t)=\langle-x_0+\int_{t_0}^t\xi(s)ds,\xi(t)\rangle$, almost 
everywhere, because $\xi\in L^2(t_0,T;\Hb)\subseteq L^1(t_0,T;\Hb)$
(part (vi)-(c) of Theorem \ref{Thm-existence}). So we have 
$$
\ddot h(t) + \frac{\alpha}{t} \dot h(t)= 
  \|\dot u(t)\|^2
  -\left(1-\frac{\alpha\beta}{t}\right)\langle x(t)-z,\xi(t)\rangle
  -\beta\left(1-\frac{\alpha\beta}{t}\right)\dot I(t),
$$
for almost 
every $t>t_0$. 

The rest of the proof runs as in the smooth case (see Subsection \ref{SS:Lyapunov}) with $\xi$ in place of $\nabla\Phi\circ x$. 
We must notice that the integrations by parts used to obtain \eqref{E:h_dot} are legitimate because $\dot h$, $W$ and $I$ are absolutely continuous. 
\end{proof}

\subsection{Fast convergence of the values for $\alpha\ge 3$}

Let $(x,y):[t_0,+\infty[\to\Hb\times\Hb$ be the global strong solution to 
(g-DIN-AVD) with Cauchy data 
$(x(t_0),y(t_0))=(x_0,y_0)\in\mbox{dom}\phi\times\Hb$. Let us show that the conclusions presented in Subsection \ref{SS:age3_smooth} remain valid, except, of course, for the convergence to zero of the acceleration, which depends on the Lipschitz continuity of the Hessian (see the last part of Proposition \ref{P:tGrad_L2}).

Suppose $\alpha\geq3$ and $x^\ast\in\argmin\Phi$. For $\lambda\in[2,\alpha-1]$ 
we define the function $\Elambda:[t_0,+\infty[\rightarrow\R$ by  
\begin{equation}\label{grfast2}
\Elambda(t)=
  t(t-\beta(\lambda+2-\alpha))(\Phi(x(t))-\min\Phi)+
  \frac{1}{2}\|\lambda(x(t)-x^\ast)+t\dot u(t)\|^2+
  \lambda(\alpha-\lambda-1)\frac{1}{2}\|x(t)-x^\ast\|^2,
\end{equation}
where $u$ is defined on $[t_0,+\infty[$ by \eqref{u0} and $\dot u$ is given by \eqref{guniter}. Function $\Elambda$ is the sum of three terms, 
each of which is at least absolutely continuous on $[t_0,T]$ for all $T>t_0$. Hence $\Elambda$ is differentiable almost everywhere. To compute 
$\frac{d}{dt}\Elambda(t)$ we first differentiate each term of $\Elambda$ in turn. 

\smallskip
With \eqref{gDINAVD3} we have 
$$
\frac{d}{dt}[t(t-\beta(\lambda+2-\alpha))(\Phi(x(t))-\min\Phi)] 
  =(2t-\beta(\lambda+2-\alpha))(\Phi(x(t))-\min\Phi)+ 
  t(t-\beta(\lambda+2-\alpha))\langle\xi(t),\dot x(t))\rangle,
$$
for almost all $t>t_0$. Next, with \eqref{dguniter0}, we have 
\begin{eqnarray*}
\frac{d}{dt}\frac{1}{2}\|\lambda(x(t)-x^\ast)+t\dot u(t)\|^2 
& = & 
  \langle\lambda(x(t)-x^\ast)+t\dot u(t),
  \lambda\dot x(t)+\dot u(t)+t\ddot u(t)\rangle \\
& = & 
  \langle\lambda(x(t)-x^\ast)+t\dot u(t),
  (\lambda+1-\alpha)\dot x(t)-(t-\beta)\xi(t)\rangle \\
& = & 
  \lambda(\lambda+1-\alpha)\langle x(t)-x^\ast,\dot x(t)\rangle 
  -t(\alpha-\lambda-1)\|\dot x(t)\|^2 
  -\beta t(t-\beta)\|\xi(t)\|^2 \\
& & 
  -\lambda(t-\beta)\langle x(t)-x^\ast,\xi(t)\rangle
  -t(t-\beta(\lambda+2-\alpha))\langle\xi(t),\dot x(t)\rangle,
\end{eqnarray*}
for almost all $t>t_0$. Lastly,
$$
\frac{d}{dt}\lambda(\alpha-\lambda-1)\frac{1}{2}\|x(t)-x^\ast\|^2=
\lambda(\alpha-\lambda-1)\langle x(t)-x^\ast,\dot x(t)\rangle.
$$
Collecting these results, we obtain
\begin{eqnarray}
\frac{d}{dt}\Elambda(t)
& = & 
  (2t-\beta(\lambda+2-\alpha))(\Phi(x(t))-\min\Phi)
  -\lambda(t-\beta)\langle x(t)-x^\ast,\xi(t)\rangle 
  \label{gdrfast2} \\
& &
  -t(\alpha-\lambda-1)\|\dot x(t)\|^2 
  -\beta t(t-\beta)\|\xi(t)\|^2, \nonumber  
\end{eqnarray}
for almost all $t>t_0$. Since $\xi(t)\in\partial\phi(x(t))$ for all $t>t_0$ (part (vi)-(a) of Theorem \ref{Thm-existence}), we have 
$$\langle\xi(t),x(t)-x^\ast\rangle\geq\Phi(x(t)-\Phi(x^\ast),$$ 
and we deduce, from \eqref{gdrfast2}, that
\begin{equation} \label{grfastoche}
\frac{d}{dt}\Elambda(t)\leq
  -((\lambda-2)t-\beta(\alpha-2))(\Phi(x(t))-\min\Phi)
  -t(\alpha-\lambda-1)\|\dot x(t)\|^2-\beta t(t-\beta)\|\xi(t)\|^2,
\end{equation}
for almost all $t\geq t_1=\max\{t_0,\beta\}$.\\

The arguments used in Section \ref{S:smooth} can be modified accordingly (using $\xi$ in 
place of $\nabla\Phi\circ x$) to give

\begin{theorem} \label{Thm-g-fast-conv2}
Let $\alpha\geq3$ and $\argmin\Phi\neq\emptyset$. Suppose $(x,y):[t_0,+\infty[\rightarrow\Hb\times\Hb$ is the global strong solution to 
(g-DIN-AVD) with initial value $(x(t_0),y(t_0))=(x_0,y_0)\in\mbox{dom }\phi\times\Hb$. Let $\lambda\in[2,\alpha-1]$ and $t_1=\max\{t_0,\beta\}$. Then
\begin{enumerate}
\item\label{gfast-conv2-1}
   $\lim_{t\rightarrow+\infty}\Elambda(t)$ exists. 
\item\label{gfast-conv2-2}
  For $t\geq s>t_1$ we have 
  $\Phi(x(t))-\min\Phi\leq
   \frac{1}{t^2}(\frac{s}{s-\beta})^{\alpha-2}\Elambda(s)= 
   \mathcal O(t^{-2})$.
\item\label{gfast-conv2-3}
   $x$ is bounded.
\item\label{gfast-conv2-4} 
  $\int_{t_0}^\infty t^2\|\xi(t)\|^2dt<+\infty$ and 
  $\int_{t_0}^\infty\|\xi(t)\|dt<+\infty$.
\item\label{gfast-conv2-5} 
  $\|\dot x(t)+\beta\xi(t)\|=\mathcal O(t^{-1})$.
\end{enumerate}
\end{theorem}

\subsection{Weak convergence of trajectories and faster convergence of the values for $\alpha>3$}

In this section we state results quite similar, with their proofs, to Lemma 
\ref{superconv}, and Theorems \ref{T:weak_convergence} and \ref{T:superconv} of 
the smooth case. Proofs are omitted, except for part (iii) of Lemma \ref{g-superconv} below. 
\begin{lemma}\label{g-superconv}
Let $\alpha>3$ and $x^*\in\argmin\Phi$. Let 
$(x,y):[t_0,+\infty[\rightarrow\Hb\times\Hb$ be the global strong solution to 
(g-DIN-AVD) with initial value 
$(x(t_0),y(t_0))=(x_0,y_0)\in\mbox{dom }\phi\times\Hb$. Then,
\begin{itemize}
\item [(i)] $\int_{t_0}^\infty t(\Phi(x(t))-\min\Phi)dt<\infty$ and 
  $\int_{t_0}^\infty t\|\dot x(t)\|^2dt<\infty$.
\item [(ii)] $\int_{t_0}^\infty t\langle x(t)-x^*,\xi(t)\rangle dt<\infty$ and 
  $\int_{t_0}^\infty\langle x(t)-x^*,\xi(t)\rangle dt<\infty$.
\item [(iii)] $\lim_{t\rightarrow+\infty}\|x(t)-x^\ast\|$ and $\lim_{t\rightarrow+\infty}
    t\langle x(t)-x^\ast,\dot x(t)+\beta\xi(t)\rangle$ exist. 
\end{itemize}
\end{lemma}

\begin{proof}
As mentioned above, we only prove part (iii). Take two distinct values $\lambda$ and $\lambda'$ in $[2,\alpha-1]$. For all 
$t\geq t_0$, we have (recall the definition \eqref{grfast2} of $\Elambda$ and 
equality \eqref{guniter} giving $\dot u$)
\begin{equation}\label{g-diffElambda}
\mathcal E_{\lambda'}(t)-\Elambda(t)=(\lambda'-\lambda)
  \left(
  -\beta t(\Phi(x(t))-\min\Phi)
  +t\langle x(t)-x^\ast,\dot u(t)\rangle
  +(\alpha-1)\frac{1}{2}\|x(t)-x^\ast\|^2
  \right).
\end{equation}
Define for $t\geq t_0$
\begin{eqnarray}
h(t) & = & \frac{1}{2}\|x(t)-x^\ast\|^2 \nonumber \\
k(t) & = & 
t\langle x(t)-x^\ast,\dot u(t)\rangle+(\alpha-1)h(t) \label{g-k} \\
q(t) & = & 
h(t)+\inttz\langle x(s)-x^\ast,\beta\xi(s)\rangle ds. \nonumber
\end{eqnarray}
Function $q$ is absolutely continuous on $[t_0,T]$ for all $T>t_0$. Indeed $h$ 
is, and the integrand $\langle x-x^\ast,\beta\xi\rangle$ belongs to 
$L^1(t_0,T)$ because $x-x^\ast$ and $\xi$ belong to $L^2(t_0,T;\Hb)$. Hence 
$q$ is differentiable almost everywhere and satisfies 
$$
\dot q(t)
  =\langle x(t)-x^*,\dot x(t)\rangle+\langle x(t)-x^*,\beta\xi(t)\rangle
  =\langle x(t)-x^*,\dot u(t)\rangle,
$$
which shows that $q$ is actually continuously differentiable. 

On the one hand, equation \eqref{g-k} shows that $k(t)$ has a limit as 
$t\to+\infty$: this is a consequence of Theorem 
\ref{Thm-g-fast-conv2}\eqref{gfast-conv2-1}\eqref{gfast-conv2-2} and 
\eqref{g-diffElambda}. On the other hand, we can rewrite $k(t)$ as
$$
k(t)=t\dot q(t)+(\alpha-1)q(t)
-(\alpha-1)\inttz\langle x(s)-x^\ast,\beta\xi(s)\rangle ds,
$$
where the integral has a limit as $t\to+\infty$, by part (ii). 
Hence $t\dot q(t)+(\alpha-1)q(t)$ has a limit, hence (Lemma \ref{elemutil}) 
$q(t)$ has a limit, hence $h(t)$ has a limit, hence, with \eqref{g-k},  
$t\langle x(t)-x^\ast,\dot x(t)+\beta\nabla\Phi(x(t))\rangle$ has a limit.
\end{proof}

The arguments of Section \ref{S:smooth} can be applied to obtain the following results:

\begin{theorem} \label{T:weak_convergence-nonsmooth}
Let $\alpha>3$ and $\argmin\Phi\neq\emptyset$, and let
$(x,y):[t_0,+\infty[\rightarrow\Hb\times\Hb$ be the global strong solution to 
(g-DIN-AVD) with initial value 
$(x(t_0),y(t_0))=(x_0,y_0)\in\mbox{dom }\phi\times\Hb$. Then $x(t)$ 
converges weakly, as $t\to+\infty$, to a point in $\argmin\Phi$.
\end{theorem}

\begin{theorem} \label{T:superconv-nonsmooth}
Let $\alpha>3$ and $\argmin\Phi\neq\emptyset$. Let
$(x,y):[t_0,+\infty[\rightarrow\Hb\times\Hb$ be the global strong solution to 
(g-DIN-AVD) with initial value 
$(x(t_0),y(t_0))=(x_0,y_0)\in\mbox{dom }\phi\times\Hb$. Then 
\begin{eqnarray*}
\Phi(x(t))-\min\Phi & = & o\left(t^{-2}\right) \\
\|\dot x(t)+\beta\xi(t)\| & = & o\left(t^{-1}\right). 
\end{eqnarray*}
\end{theorem}

\subsection{Strong convergence}

The results of Section \ref{S-conv-forte} about a smooth potential, can also be established for a 
lower semicontinuous potential in a straightforward manner, using $\xi$ in place of $\nabla\Phi\circ x$ (observe that integrations by parts 
are legitimate by the absolute continuity of the functions involved). We only state the theorems and omit their proofs. 

\begin{theorem} 
Suppose $\alpha>3$, $\beta>0$ and let $\phi:\Hb\to\R\cup\{+\infty\}$ be 
proper, lower-semicontinuous, convex and even. Let $(x,y):[t_0,+\infty[\rightarrow\Hb\times\Hb$ be the global strong solution to 
(g-DIN-AVD) with initial value $(x(t_0),y(t_0))=(x_0,y_0)\in\mbox{dom }\phi\times\Hb$. Then, $x(t)$ converges strongly, as $t\to+\infty$, to some $x^*\in\argmin\Phi$.
\end{theorem}

\begin{theorem} 
Suppose $\alpha >3$, $\beta >0$ and let $\phi:\Hb\to\R\cup\{+\infty\}$ be a proper lower-semicontinuous convex function satisfying {\rm int}$(\argmin \Phi) \neq \emptyset $. Let
$(x,y):[t_0,+\infty[\rightarrow\Hb\times\Hb$ be the global strong solution to (g-DIN-AVD) with initial value $(x(t_0),y(t_0))=(x_0,y_0)\in\mbox{dom }\phi\times\Hb$. Then, $x(t)$ converges strongly, as $t\to+\infty$, to some $x^*\in\argmin\Phi$. Moreover,
$$
\int_{t_0}^{\infty}  t\|\xi(t)\|dt < +\infty.
$$
\end{theorem}

\begin{theorem} 
Suppose $\alpha >3$, $\beta >0$ and let $\phi:\Hb\to\R\cup\{+\infty\}$ be a
boundedly inf-compact proper lower semicontinuous convex function. Let
$(x,y):[t_0,+\infty[\rightarrow\Hb\times\Hb$ be the global strong solution to
(g-DIN-AVD) with initial value
$(x(t_0),y(t_0))=(x_0,y_0)\in\mbox{dom }\phi\times\Hb$. Then, $x(t)$ converges strongly, as $t\to+\infty$, to some $x^*\in\argmin\Phi$.
\end{theorem} 

In the smooth case, the proof of Theorem \ref{Thm-sc-basic} relies on 
inequality \eqref{sc1}, which, in the nonsmooth case, has to be replaced by
$$
\phi(y)\geq\phi(x)+\langle\xi,y-x\rangle+\frac{\mu}{2}\|y-x\|^2
$$
for all $x$, $y$ in $\mbox{dom}\phi$ and all $\xi\in\partial\phi(x)$. We obtain:

\begin{theorem} 
Suppose $\alpha >3$, $\beta >0$ and let $\phi:\Hb\to\R\cup\{+\infty\}$ be a
strongly convex proper lower semicontinuous function. Let
$(x,y):[t_0,+\infty[\rightarrow\Hb\times\Hb$ be the global strong solution to
(g-DIN-AVD) with initial value
$(x(t_0),y(t_0))=(x_0,y_0)\in\mbox{dom }\phi\times\Hb$. Then, $\argmin\Phi$ is 
reduced to a singleton $x^*$, and the following properties hold: 
\begin{eqnarray*}
\Phi(x(t))-\min_{\mathcal H}\Phi
& = & 
\mathcal O\left(t^{-\frac{2\alpha}{3}}\right) \\
\|x(t)-x^*\|
& = & 
O\left(t^{-\frac{\alpha}{3}}\right).
\end{eqnarray*}
\end{theorem}

\section{Asymptotic behavior of the trajectory under perturbations}

In this section, we analyze the asymptotic behavior, as $t\to+\infty$, of the solutions of the differential equation
\begin{equation}\label{E:perturbed}
\ddot{x}(t) + \frac{\alpha}{t} \dot{x}(t) + \beta \nabla^2 \Phi (x(t))\dot{x} (t) + \nabla \Phi (x(t))  = g(t),
\end{equation}
where  the second member $g: [t_0, +\infty[ \to \mathcal H $ of (\ref{E:perturbed}) is supposed to be locally integrable, and acts as a perturbation of (DIN-AVD). We restrict ourselves to the smooth case for simplicity. Therefore, we assume that $\Phi : \mathcal H \rightarrow \mathbb R$ is convex, twice continuously differentiable, and $\nabla\Phi$ is Lipschitz-continuous on bounded sets. From the Cauchy-Lipschitz-Picard Theorem, for any initial condition $(x_0,\dot x_0)\in\Hb\times\Hb$, we deduce the existence and uniqueness of a maximal local solution $x$ to \eqref{E:perturbed}, 
with $\dot x$ locally absolutely continuous. If $\Phi$ is bounded from below, the global existence follows from the energy estimate proved in Proposition \ref{P:Lyapunov-pert} below. This being said, our main concern here is to obtain sufficient conditions on $g$ ensuring that the convergence properties established in the previous section are preserved. The analysis follows very closely the arguments given in Section \ref{S:smooth}. Therefore, we shall state the main results and sketch the proofs, underlining the parts where additional techniques are required.

\subsection{Lyapunov analysis and minimizing properties of the solutions for $\alpha>0$} 

Let $x:t\in[t_0,\infty[\to \Hb$ satisfy (\ref{E:perturbed}) with Cauchy data $x(t_0)=x_0$, $\dot x(t_0)=\dot x_0$. Let $\theta\in[0,\beta]$, and $T >t_0$. For $t_0 \leq t \leq T$, define the energy function, $W_{\theta , g, T}:[t_0,\infty[\to\R$ by
\begin{equation} \label{E:W_theta-pert}
W_{\theta , g, T}(t)=\Phi(x(t))+\frac{1}{2}\|\dot x(t)+\theta\nabla\Phi(x(t))\|^2+\frac{\theta(\beta-\theta)}{2}\|\nabla \Phi(x(t))\|^2 + \int_t^T \langle  \dot{x}(\tau) + \theta \nabla \Phi(x(\tau)) ,  g(\tau)   \rangle d\tau.
\end{equation} 
We have the following:

\begin{proposition}\label{P:Lyapunov-pert}
Let $\alpha>0$, and suppose $x:[t_0,+\infty[\rightarrow\Hb$ is a solution of (\ref{E:perturbed}). Then, for each $\theta\in[0,\beta]$ and $t\ge\max\{t_0,\frac{\alpha\theta}{2}\}$, we have
$$\dot W_{\theta , g, T}(t)\le-\frac{\alpha}{2t}\|\dot x(t)\|^2 -\frac{\alpha}{2t}\|\dot{x}(t) + \theta \nabla \Phi(x(t))\|^2.$$
\end{proposition}

The proof goes along the lines of Proposition \ref{P:Lyapunov}. In the computation of $\dot{W}_{\theta ,g, T} (t)$, the terms containing $g$ cancel out.\\

We now prove an auxiliary result, that will be useful later on.

\begin{lemma}\label{L:energy-thm-1}
Suppose that $\Phi$ is bounded from below. Let $x: [t_0, +\infty[ \rightarrow \mathcal H$ be a solution of \eqref{E:perturbed} with $\alpha >0$ and $\displaystyle{\int_{t_0}^{\infty} \|g(t)\|\,dt < + \infty}$. Then,  \ $\sup\limits_{t\geq t_0} \|  \dot{x}(t)   \| < + \infty$, and  $\sup\limits_{t \geq t_0} \|  \nabla\Phi(x(t))  \| < + \infty$.
\end{lemma}

\begin{proof} 
	Let us first fix  $T >t_0$. By Proposition \ref {P:Lyapunov-pert},  for
any  $\theta\in[0,\beta]$,   $W_{\theta ,g, T}(\cdot)$ is a decreasing function on  $[t_0 , T]$. In particular, $ W_{\theta ,g, T}(t) \leq W_{\theta ,g, T}(t_0)$ for $t_0 \leq t \leq T$, that is 
\begin{align*}
\Phi(x(t))&+\frac{1}{2}\|\dot x(t)+\theta\nabla\Phi(x(t))\|^2 +\frac{\theta(\beta-\theta)}{2}\|\nabla \Phi(x(t))\|^2 + \int_t^T \langle  \dot{x}(\tau) + \theta \nabla \Phi(x(\tau)) ,  g(\tau)   \rangle d\tau \\
 &\leq \Phi(x_0)+\frac{1}{2}\|\dot x(t_0)+\theta\nabla\Phi(x_0)\|^2+\frac{\theta(\beta-\theta)}{2}\|\nabla \Phi(x_0)\|^2 + \int_{t_0}^T \langle  \dot{x}(\tau) + \theta \nabla \Phi(x(\tau)) ,  g(\tau)   \rangle d\tau .
\end{align*}
As a consequence
\begin{equation}\label{pre-Bellman}
 \frac{1}{2}\| \dot x(t)+\theta\nabla\Phi(x(t)) \|^2  \leq C   + \int_{t_0}^t \|\dot x(\tau)+\theta\nabla\Phi(x(\tau))\| \| g(\tau) \| d\tau,
\end{equation}
with
$$
C := \Phi(x_0) - \inf \Phi +\frac{1}{2}\|\dot x(t_0)+\theta\nabla\Phi(x_0)\|^2+\frac{\theta(\beta-\theta)}{2}\|\nabla \Phi(x_0)\|^2,
$$
which does not depend on $T$. As a consequence, inequality \eqref{pre-Bellman} holds true for any $t \geq t_0$.
Applying Gronwall-Bellman Lemma (see \cite[Lemme A.4]{Bre1}), we obtain 
$$\| \dot{x}(t)+\theta\nabla\Phi(x(t)) \| \leq  \sqrt{ 2C}   + \int_{t_0}^t \| g(\tau) \| d\tau.$$
Using the integrability of $g$, it follows that $\sup_{t \geq t_0} \| \dot{x}(t)+\theta\nabla\Phi(x(t)) \| < +\infty$. Then, taking $\theta=0$, we obtain $\sup\limits_{t \geq t_0} \|  \dot{x}(t)   \| < + \infty$. 
Finally, using $\theta=\beta$ and the triangle inequality, we obtain  $\sup\limits_{t \geq t_0} \|  \nabla\Phi(x(t))  \| < + \infty$.
\end{proof} 

If $g$ is integrable on $[t_0, +\infty[$, Lemma \ref{L:energy-thm-1} allows us to define a function $W_{\theta, g} : \, [t_0,+\infty[\to\R$ by
\begin{equation}\label{Liap100}
W_{\theta , g}(t) = \Phi(x(t))+\frac{1}{2}\|\dot x(t)+\theta\nabla\Phi(x(t))\|^2+\frac{\theta(\beta-\theta)}{2}\|\nabla \Phi(x(t))\|^2 + \int_t^{\infty} \langle  \dot{x}(\tau) + \theta \nabla \Phi(x(\tau)) ,  g(\tau)   \rangle d\tau.
\end{equation}
For each $T \geq t_0$, $W_{\theta, g}$ and $W_{\theta , g, T}$ differ by a constant, and have the same derivative, given by Proposition \ref{P:Lyapunov-pert}. When $\theta\in\{0,\beta\}$, which is our main concern in the next theorem, definition \eqref{Liap100} of $W_{\theta , g}$ reduces to
$$
W_{\theta , g}(t) = \Phi(x(t))+\frac{1}{2}\|\dot x(t)+\theta\nabla\Phi(x(t))\|^2 + \int_t^{\infty} \langle  \dot{x}(\tau) + \theta \nabla \Phi(x(\tau)) ,  g(\tau)   \rangle d\tau.
$$
We are now in a position to prove the following perturbed version of Theorem \ref{T:limit_W_Phi} and Proposition \ref{P:speed_to_0}. 

\begin{theorem} \label{T:limit_W_Phi-pert}
Let $\alpha>0$, and suppose $x:[t_0,+\infty[\rightarrow\Hb$ is a solution of (\ref{E:perturbed}). Suppose that $\Phi$ is bounded from below and  $\displaystyle{\int_{t_0}^{\infty} \,\|g(t)\|\, dt < + \infty}$. Then
$$\lim_{t\to+\infty}W_{0,g}(t)=\lim_{t\to+\infty}W_{\beta,g}(t)=\lim_{t\to+\infty}\Phi(x(t))=\inf\Phi .$$
Moreover
$$\lim_{t\to+\infty}\|\dot x(t)\|=\lim_{t\to+\infty}\|\nabla\Phi(x(t))\|=0.$$
\end{theorem}
\begin{proof} 
We follow the lines of the proof of Theorem \ref{T:limit_W_Phi}, adopting the same notations. Some modifications are introduced by the perturbation term $g$. Instead of inequality \eqref{E:h_dot}, we obtain
\begin{equation} \label{E:h_dot_pert}
\frac{1}{t}\dot h(t)+\int_{t_1}^t\left(\frac{1}{s}-\frac{\alpha\theta}{s^2}\right)\big(W_{\theta , g}(s)-\Phi(z)\big)\,ds \le -\int_{t_1}^t\left(\frac{3}{\alpha}-\frac{\theta}{s}\right)\dot W_\theta(s)\,ds+C + K_1(t) + K_2 (t),
\end{equation} 
where
$$
K_1(t)= \int_{t_1}^t \langle  \frac{1}{s}  g(s), u_{\theta} (s) -z \rangle ds\qquad\hbox{and}\qquad K_2 (t)= \int_{t_1}^t \left(\frac{1}{s}-\frac{\alpha\theta}{s^2}\right)  \int_s^{\infty} \langle  \dot u_\theta(\tau),  g(\tau) \rangle d\tau ds.
$$
Let us  majorize  $K_1$ and $K_2$. The relation $\Vert u_{\theta}(s)-z\Vert \leq \Vert u_{\theta}(t_1)-z\Vert +\int_{t_1}^{s}\Vert\dot{u}_{\theta}(\tau)\Vert d\tau$, and Lemma \ref{L:energy-thm-1} together imply
$$K_1(t)\leq \int_{t_{1}}^{t}\frac{1}{s} \| g(s)\| \|u_{\theta} (s) -z \|  ds \leq \left(  \frac{\Vert u_{\theta}(t_1)-z\Vert}{t_{1}}+
\sup_{t\ge t_1}\Vert \dot{u}_{\theta}(\tau)\Vert\right) \int_{t_{1}}^{+\infty}\Vert g(s)\Vert ds\leq C < +\infty.$$
For $K_2$, we use integration by parts and Lemma \ref{L:energy-thm-1} to obtain
$$
K_2 \leq C \int_{t_{1}}^{t} \left( \frac{1}{s}\int_{s}^{\infty}\Vert g(\tau)\Vert d\tau\right)  ds \leq C\left( \ln t \int_{t}^{\infty}\Vert g(\tau)\Vert d\tau + \int_{t_{1}}^{t} \Vert g(\tau)\Vert\ln \tau  \ d\tau+1 \right) .
$$
Then, we continue just as in the proof of Theorem \ref{T:limit_W_Phi}, but, instead of inequality \eqref{int2}, we obtain
\begin{align*} 
\big(W_{\theta , g}(t)-\Phi(z)\big)&\left(t\ln t+(D-1)t+E\ln t+  F\right)ds \\
&\leq C'\left(       t + t \ln t \int_{t}^{\infty}\Vert g(\tau)\Vert d\tau  + 2\int_{t_{2}}^{t} \Vert g(\tau)\Vert \tau \ln \tau  \ d\tau  +   t \int_{t_{1}}^{t}\Vert g(\tau)\Vert \ln \tau d\tau \right)  +G,
\end{align*}
for some appropriate constants $D,E,F,G\in\R$. Divide by $t \ln t$, let $t\to+\infty$, and use Lemma \ref{basic-int}, to obtain $\lim_{t\to+\infty}W_{\theta , g}(t)\le\Phi(z)$. 
The integrability of $g$ and Lemma \ref{L:energy-thm-1} yield $\lim_{t\to+\infty} \int_t^{\infty} \langle  \dot{x}(\tau) + \theta \nabla \Phi(x(\tau)) ,  g(\tau)   \rangle d\tau =0$. As a consequence, 
$$
\lim_{t\to+\infty}\left(  \Phi(x(t))+\frac{1}{2}\|\dot x(t)+\theta\nabla\Phi(x(t))\|^2 \right) \leq \Phi (z)
$$
for each $z \in \mathcal H$. We deduce that $\lim_{t\to+\infty}\Phi(x(t))=\inf\Phi $, and $\lim_{t\to+\infty} \|\dot x(t)+\theta\nabla\Phi(x(t))\|  =0$. Taking successively $\theta=0$, and $\theta = \beta$, we finally obtain 
$\lim_{t\to+\infty}\|\dot x(t)\|=\lim_{t\to+\infty}\|\nabla\Phi(x(t))\|=0$.
\end{proof}

\subsection{Fast convergence of the values for $\alpha\geq3$ and convergence of the trajectories for $\alpha>3$.}

\begin{theorem}\label{fastconv-thm}
Let $\argmin\Phi\neq\emptyset$, and let $x: [t_0, +\infty[ \rightarrow \mathcal H$ be a solution of \eqref{E:perturbed} with $\alpha\ge 3$ and $\displaystyle{\int_{t_0}^{\infty} t\,\|g(t)\|\, dt < + \infty}$. Then $\Phi(x(t))- \min_{\mathcal H}\Phi = \mathcal O  \left( t^{-2}\right)$. 
\end{theorem}

\begin{proof}  
Take $x^\ast\in\argmin\Phi$, and let $x$ be a solution of \eqref{E:perturbed} with Cauchy data $(x(t_0),\dot x(t_0))=(x_0,\dot x_0)\in\Hb\times\Hb$ and $\alpha\geq3$. For 
$\lambda\in[2,\alpha-1]$, and $t_0 <T <+\infty$, we define the function $ \mathcal E_{\lambda,g, T}:[t_0,T]\rightarrow\R$ by  
\begin{align*}
 \mathcal E_{\lambda,g, T} (t):= & t(t-\beta(\lambda+2-\alpha))(\Phi(x(t))-\min\Phi)+
  \frac{1}{2}\|\lambda(x(t)-x^\ast)+t\dot u_\beta(t)\|^2+
  \lambda(\alpha-\lambda-1)\frac{1}{2}\|x(t)-x^\ast\|^2 \\
  & + \int _t^T \tau \langle \lambda( x(\tau) - x^{*}) + \tau  \dot u_\beta(\tau) ,  g(\tau) \rangle d\tau,
\end{align*}
where $u_\beta$ is given by \eqref{u0reg}, with $\theta=\beta$. When derivating $\mathcal E_{\lambda, g,T}$, the terms containing $ g $  cancel out. As in Section \ref{S:smooth}, we obtain \eqref{rfastoche} with $\mathcal E_{\lambda,g, T}$ instead of $\mathcal E_\lambda$. It follows that $\mathcal E_{\lambda,g, T}$ is decreasing on $[t_0 , T]$. In particular, $ \mathcal E_{\lambda,g, T}(t) \leq \mathcal E_{\lambda,g, T}(t_0)$ for $t_0 \leq t \leq T$. This gives 
\begin{align}\label{Gronw2}
\frac{1}{2}\|\lambda(x(t)-x^\ast)+t\dot u_\beta(t)\|^2 \leq C + \int_{t_0}^t \| \lambda(x(t)-x^\ast)+t\dot u_\beta(t)\|  \| \tau g(\tau) \| d\tau ,
\end{align}
with
$$
C= t_0(t_0-\beta(\lambda+2-\alpha))(\Phi(x(t_0))-\min\Phi)+  \frac{1}{2}\|\lambda(x(t_0)-x^\ast)+t_0\dot u_\beta(t_0)\|^2+   \lambda(\alpha-\lambda-1)\frac{1}{2}\|x(t_0)-x^\ast\|^2,
$$
which does not depend on $T$. As a consequence, inequality \eqref{Gronw2} holds true for any $t \geq t_0$. Applying Gronwall-Bellman Lemma (see \cite[Lemme A.4]{Bre1}) to \eqref{Gronw2}, and using the integrability of $t\mapsto tg(t)$, it follows that
\begin{equation} \label{energy-033} 
\sup_{t \geq t_0} \| \lambda(x(t)-x^\ast)+t\dot u_\beta(t) \|  \leq  \sqrt{ 2C}   + \int_{t_0}^\infty \| \tau g(\tau) \| d\tau< +\infty.
\end{equation} 
As a consequence, we can define the energy function
\begin{align*}
 \mathcal E_{\lambda,g} (t):= &\ t(t-\beta(\lambda+2-\alpha))(\Phi(x(t))-\min\Phi)+
  \frac{1}{2}\|\lambda(x(t)-x^\ast)+t\dot u_\beta(t)\|^2+
  \lambda(\alpha-\lambda-1)\frac{1}{2}\|x(t)-x^\ast\|^2 \\
  & + \int _t^{\infty} \tau \langle \lambda( x(\tau) - x^{*}) + \tau  \dot u_\beta(\tau) ,  g(\tau) \rangle d\tau,
\end{align*}
which has the same derivative as $\mathcal E_{\lambda,g, T}$. Hence $\mathcal E_{\lambda,g} (t) \leq \mathcal E_{\lambda,g} (t_0)$. Combined with \eqref{energy-033}, this gives
$$
t(t-\beta(\lambda+2-\alpha))(\Phi(x(t))-\min\Phi)
\leq C + \sup_{t \geq t_0} \| \lambda(x(t)-x^\ast)+t\dot u_\beta(t) \|  \int _{t_0}^{\infty} \|\tau   g(\tau) \| d\tau < + \infty,
$$
and the result follows.
\end{proof}

Finally, we have the following perturbed version of Theorem \ref{T:weak_convergence}:

\begin{theorem} \label{Thm-weak-conv}
Let $\argmin\Phi\neq\emptyset$, and let $x: [t_0, +\infty[ \rightarrow \mathcal H$ be a solution of \eqref{E:perturbed} with $\alpha>3$ and $\displaystyle{\int_{t_0}^{\infty} t\,\|g(t)\|\, dt < + \infty}$. Then, $x(t)$ converges weakly, as $t\to+\infty$, to a point in $\argmin\Phi$.
\end{theorem}

\begin{proof}
As in the proof of Theorem \ref{T:weak_convergence}, the result follows easily once we obtain the estimations given in Lemma \ref{superconv}, which are all obtained following the same arguments, replacing $\mathcal E_\lambda$ by $\mathcal E_{\lambda,g}$. 
\end{proof}

\section{Inertial forward-backward algorithms} \label{section-algorithm}
When applied to structured optimization, time discretization of the (g-DIN-AVD) dynamic provides  a new class of inertial forward-backward algorithms, which enlarge the field of FISTA methods.

\subsection{(g-DIN-AVD) for structured minimization}

In many situations,  we are dealing with a structured convex minimization problem
\begin{equation}\label{eq:algo1}
\min \left\lbrace  \phi (x) + \Psi (x): \ x \in \mathcal H   \right\rbrace   
\end{equation}
involving  the sum of two potential functions, namely  $\Psi$ smooth, and $\phi$  nonsmooth. Precisely, 

\smallskip

\noindent $\bullet$ $\phi: \mathcal H \to  \mathbb R \cup \lbrace   + \infty  \rbrace  $ is a convex, lower semicontinuous proper function  (which possibly takes the value $+ \infty$);

\smallskip

\noindent$\bullet$ $\Psi: \mathcal H \to  \mathbb R  $ is a convex, continuously differentiable function, whose gradient is Lipschitz continuous on bounded sets. 

\smallskip

In order to  highlight the asymmetrical role played by the two potential functions, the smooth potential is indicated by  a capital letter $\Psi$, and the nonsmooth potential by $\phi$.
Since $\Psi$ is continuous, by the classical additivity rule for the subdifferential of a sum of convex functions (the Moreau-Rockafellar Theorem), we have
\begin{equation}\label{stat}
\partial (\phi + \Psi) = \partial \phi +  \nabla \Psi . 
\end{equation}
Thus, the (g-DIN-AVD) system writes
\begin{equation}\label{eq:fos2b}
\hbox{\rm(g-DIN-AVD)} \ \left\{
\begin{array}{l}
\dot x(t) +  \beta \partial \phi (x(t)) + \beta \nabla \Psi (x(t)) - \left( \frac{1}{\beta} -  \frac{\alpha}{t} \right)  x(t) - y(t) \ni 0;  	 \\
\rule{0pt}{20pt}
 \dot{y}(t)    + \frac{1}{\beta} \left( \frac{1}{\beta} -  \frac{\alpha}{t} +  \frac{\alpha \beta}{t^2}  \right) x(t)
 +\frac{1}{\beta} y(t) =0.
\end{array}\right.
\end{equation}

Keeping this in mind, we can devise an algorithm for the numerical minimization of the function $\phi+\Psi$ by discretizing \eqref{eq:fos2b}. In view of the asymmetric regularity properties of the two functions, we  are going to discretize \eqref{eq:fos2b} implicitely with respect to the nonsmooth function $\phi$, and explicitely with respect to the smooth function $\Psi$. More precisely, take a time step size $h>0$, and $t_k= kh$, $x_k = x(t_k)$ ,
$y_k = y(t_k)$. We start with $(x_0,y_0)\in \mathcal H\times \mathcal H$. At the $k$-th iteration, given $(x_k,y_k)$ compute $x_{k+1}$ and then $y_{k+1}$ using the following rule:
$$\left\{\begin{array}{ccl}
0 & \in & \displaystyle\frac{x_{k+1}-x_k}{h}+ \beta \partial\phi(x_{k+1})+ \beta  \nabla\Psi(x_{k})-  \left(\frac{1}{\beta}- \frac{\alpha} {kh}\right)x_k -y_k \medskip\\
0 & = & \displaystyle\frac{y_{k+1}-y_k}{h}+ \frac{1}{\beta}\left(\frac{1}{\beta}  - \frac{\alpha }{  kh} + \frac{\alpha \beta}{  k^2 h^2}\right) x_{k+1} +\frac{1 }{\beta }y_{k+1}.
\end{array}\right.\leqno{\hbox{(IFB-AVD)}}$$

\medskip


\medskip

\noindent The acronym (IFB-AVD) stands for Inertial Forward-Backward algorithm with Asymptotic Vanishing Damping. 

Using the proximity operator $\hbox{prox}_{\beta h\phi}$ (see, for instance, \cite{BC} or \cite{Pey_book}), we can write
\begin{equation}\label{E:IFB}
\left\{\begin{array}{ccl}
x_{k+1} & = &
  \displaystyle\mbox{prox}_{\beta h\phi}
    \left(\left(1+h\left(\frac{1}{\beta}- \frac{\alpha} {kh}\right)\right)x_k-\beta h \nabla \Psi (x_k)+ h y_k\rule{0em}{.9em}\right) \\
y_{k+1} & = & \displaystyle \frac{\beta}{\beta +h}y_k -  \frac{h}{\beta +h}\left(\frac{1}{\beta}  - \frac{\alpha }{  kh} + \frac{\alpha \beta}{  k^2 h^2}\right) x_{k+1} .
\end{array}\right.
\end{equation}
So, at the $k$-th iteration, given $(x_k,y_k)$, we first compute $x_{k+1}$ with the help of the gradient of $\Psi$ (explicit, forward step), then apply the proximity mapping associated to $\phi$ (implicit, backward step), and finally compute $y_{k+1}$.\\

From a computational viewpoint, when comparing (IFB-AVD) with the classical forward-backward algorithms, the inertial and damping features inherited from
the continuous-time counterpart (DIN-AVD) induce only the addition of some terms whose computation is essentially costless. However, in the light of the results for the continuous-time trajectories, it is reasonable to expect interesting convergence properties. This goes beyond the scope of this paper, and will be the subject of future research. A somehow related (but different) inertial forward-backward algorithm was initiated in \cite{APR} in the case of a fixed  viscous parameter.

\begin{remark} A major interest to such a direct link between differential equations and algorithms is twofold: on the one hand, it suggests several properties of the later that would be difficult to detect otherwise, and, on the other one, it provides a strategy of proof.
\end{remark}

\section{Conclusions}

We have presented a second-order system (DIN-AVD) that combines properties of a nonlinear oscillator with two types of damping: an asymptotically vanishing isotropic viscosity term, and a more geometrical Hessian-driven damping. 

Its trajectories (global solutions) have several interesting properties, namely:
\begin{itemize}
	\item They minimize the objective function $\Phi$, and give $o(k^{-2})$ convergence of the values if $\argmin\Phi\neq\emptyset$. \smallskip
	\item Each trajectory converges weakly to a minimizer of $\Phi$ whenever there are any, and strong convergence holds in several important cases. \smallskip
	\item The gradient of $\Phi$ vanishes along the trajectories with order $\mathcal O(k^{-1})$. \smallskip
	\item Strong global solutions exist, even if the objective function is not differentiable, since (DIN-AVD) is equivalent to a first-order system in time and space (see below). \smallskip
	\item Both the velocity and the acceleration (if $\nabla\Phi$ is Lipschitz-continuous on bounded sets) vanish asymptotically. 
\end{itemize}

The system is closely related to forward-backward algorithms with Nesterov's acceleration scheme, through the system (AVD) studied in \cite{SBC} and \cite{APR1}, and also to Newton's (and Levenberg-Marquardt) method, in view of the presence of the Hessian of the objective function. However, it exhibits some particular important features, especially:
\begin{itemize}
	\item In (AVD), the damping is homogeneous and isotropic, and thus ignores the geometry of the potential function $\Phi$, which is to minimize. By contrast, (DIN-AVD) exhibits an additional geometric damping term which is controlled by the Hessian of $\Phi$. This type of damping is linked to Newton's method, and confers (DIN-AVD) some further favorable optimization properties. In particular, we obtain that the gradient of $\Phi$ goes to zero fast as $t\to+\infty$. Also, the acceleration vanishes asymptotically. These properties are not known for (AVD), and endow (DIN-AVD) with additional stability.
	This fact was also confirmed by some numerical experiments, although this is not the main objective of this study. \smallskip
	\item A second prominent property of (DIN-AVD) is that the system can be naturally extended to the case of a nonsmooth potential (g-DIN-AVD). This relies on the fact that the system can be equivalently formulated as a first-order system both in space and time. The main properties of (DIN-AVD), which are known for a smooth potential, extend to the case of a nonsmooth potential except for the fact that the acceleration goes to zero, which depends on the Lipschitz continuity of $\nabla\Phi$. This generalization has important consequences in mechanics and partial differential equations in relation with the modeling of nonelastic shocks. This is not explored in this paper, where our main concern is optimization. \smallskip
	\item In the third place, by considering structured potentials $\Phi+\Psi$, with $\Phi$ smooth and $\Psi$ nonsmooth, the explicit-implicit discretization of (g-DIN-AVD) gives rise to new potentially fast inertial forward-backward algorithms, which complement FISTA-like methods (see \cite{Nest1}, \cite{BT}, etc.). The study of the continuous-time setting, which is the subject of this work, is very useful in this respect since it gives a possible strategy of proof, an idea of natural candidates for a Lyapunov function, and the types of properties that can be expected for the algorithm. This goes beyond the scope of the present paper and is subject of future research. \smallskip
	\item Finally, in view of the first-order equivalent formulation, where the Hessian does not appear, the complexity per iteration is essentially that of a first-order gradient-like method! 
\end{itemize}

Summarizing, (DIN-AVD) is a second-order system (in time and space) that uses subtle information on the geometry of the objective function, for which the trajectories have remarkable convergence properties. However, from the point of view of implementation, it behaves as a first-order system (again, in time and space).

In view of the parameters $\alpha$ and $\beta$ involved in the description of (DIN-AVD), our study raises some interesting questions both from theoretical and practical perspectives. Here, we mention two:
\begin{itemize}
	\item Is $\alpha=3$ critical? It would be interesting to know if there is a function $\Phi$ for which the fast convergence property of the values does not hold with some $\alpha<3$, or if there exists a nonconvergent trajectory for $\alpha\le 3$. \smallskip
	\item Is there an optimal choice of $\alpha$ and $\beta$? There might be a rule, possibly based on the function $\Phi$, a training scheme, or a heuristic (besides the one given in Remark \ref{R:initial_condition}), to select the combination of the parameters and initial conditions that yields the best rates of convergence.
\end{itemize}

\section*{Appendix}

\begin{lemma}[Opial]\label{Opial} Let $S$ be a non empty subset of $\mathcal H$
and $x:[0,+\infty[\to \mathcal H$ a map. Assume that 
\begin{itemize}
	\item [(i)] for every $z\in S$, $\lim_{t\to+\infty}\|x(t)-z\|$ exists;
	\item [(ii)] every weak sequential cluster point of the map $x$ belongs to $S$.
\end{itemize}
Then $x(t)$ converges weakly, as $t\to+\infty$, to some $x_{\infty}\in S$.
 \end{lemma}

\begin{lemma}\label{elemutil}
Let $\Hb$ be a Hilbert space. Let $x:[t_0,+\infty[\to\Hb$ a continuously 
differentiable function satisfying $u(t)+\frac{t}{\alpha}\dot u(t)\to L$, 
$t\to+\infty$, with $\alpha>0$ and $L\in\Hb$. Then $u(t)\to L$, $t\to+\infty$. 
\end{lemma}
\begin{proof}
Set $v=u-L$ and fix $\varepsilon>0$. There exists $T\geq t_0$ such that for 
$t\geq T$
\begin{equation*}
\left\|v(t)+\frac{t}{\alpha}\dot v(t)\right\|<\varepsilon.
\end{equation*}
Multiplying by $\alpha t^{\alpha-1}$ we obtain
\begin{equation*}
\left\|\alpha t^{\alpha-1}v(t)+t^\alpha\dot v(t)\right\|
  =\left\|\frac{d}{dt}(t^\alpha v(t))\right\|
  <\varepsilon\alpha t^{\alpha-1}.
\end{equation*}
Integrating between $T$ and $t\geq T$ we get
\begin{equation*}
\left\|t^\alpha v(t)-T^\alpha v(T)\right\|
  =\left\|\int_T^t\frac{d}{ds}(s^\alpha v(s))ds\right\|
  \leq\int_T^t\left\|\frac{d}{ds}(s^\alpha v(s))\right\|ds
  <\varepsilon(t^\alpha-T^\alpha).
\end{equation*}
Hence
\begin{equation*}
\|v(t)\|\leq
\left(\frac{T}{t}\right)^\alpha\|v(T)\|
  +\varepsilon\left(1-\left(\frac{T}{t}\right)^\alpha\right),
\end{equation*}
whence we deduce $\limsup_{t\to+\infty}\|v(t)\|\leq\varepsilon$. The proof is 
complete. 
\end{proof}

\begin{lemma} \label{L:int_bounded}
Let $\tau,p>0$ and let $\psi:]\tau,+\infty[\to\R$ be twice continuously differentiable and bounded from below. Then, 
$$\inf_{t>\tau}\int_{\tau}^{t}\frac{\dot\psi(s)}{s^p}\,ds>-\infty\quad\hbox{and}\quad \inf_{t>\tau}\int_{\tau}^{t}\frac{\ddot\psi(s)}{s^p}\,ds-\frac{\dot\psi(t)}{t^p}>-\infty.$$
\end{lemma}

\begin{proof}
By subtracting $\inf\psi$ we may assume that $\psi$ is nonnegative. Using integration by parts, we obtain
$$\int_{\tau}^{t}\frac{\dot\psi(s)}{s^p}\,ds=\frac{\psi(t)}{t^p}-\frac{\psi(\tau)}{\tau^p}+p\int_\tau^t\frac{\psi(s)}{s^{p+1}}\,ds\ge -\frac{\psi(\tau)}{\tau^p}.$$
In a similar fashion, we deduce that
$$\int_{\tau}^{t}\frac{\ddot\psi(s)}{s^p}\,ds-\frac{\dot\psi(t)}{t^p}=\frac{\dot\psi(\tau)}{\tau^p}+p\int_\tau^t\frac{\dot\psi(s)}{s^{p+1}}\,ds \ge \frac{\dot\psi(\tau)}{\tau^p}-p\frac{\psi(\tau)}{\tau^{p+1}},$$
and we conclude.
\end{proof}

\begin{lemma}\label{basic-int} 
Take $\delta >0$, and let $f \in L^1 (\delta , +\infty)$ be nonnegative and continuous. Consider a nondecreasing function $\psi:]\delta,+\infty[\to]0,+\infty[$ such that $\lim\limits_{t\to+\infty}\psi(t)=+\infty$. Then, 
$$\lim_{t \rightarrow + \infty} \frac{1}{\psi(t)} \int_{\delta}^t \psi(s)f(s)ds =0.$$ 
\end{lemma}

\end{document}